\documentclass[12pt, reqno]{amsart}
\usepackage{amsmath}
\usepackage{amssymb}
\usepackage{amsthm}
\usepackage{epsfig}
\usepackage{graphicx}
\usepackage{color}
\usepackage{fullpage}
\usepackage{comment}
\definecolor{shadecolor}{gray}{0.875}
\usepackage{amscd}
\usepackage{stmaryrd}
\usepackage{threeparttable}
\usepackage{mathtools}

\usepackage{ulem}
\usepackage{amsthm}
\usepackage{csquotes}
\usepackage{enumitem}
\usepackage{graphics}
\usepackage{tikz-cd}
\usepackage{hhline}
\usepackage{multirow}
\usepackage{makecell}
\usepackage{booktabs}
\usetikzlibrary{backgrounds,calc,positioning}
\usepackage{multicol} %用于实现在同一页中实现不同的分栏
\usepackage{xcolor}
\usepackage[colorlinks=true,pagebackref]{hyperref}
\usepackage{colonequals}
\usepackage{stmaryrd}

\def\tmp#1#2#3{%
	\definecolor{Hy#1color}{#2}{#3}%
	\hypersetup{#1color=Hy#1color}}
\tmp{link}{HTML}{800006}
\tmp{cite}{HTML}{2E7E2A}
\tmp{file}{HTML}{131877}
\tmp{url} {HTML}{8A0087}
\tmp{menu}{HTML}{727500}
\tmp{run} {HTML}{137776}
\def\tmp#1#2{%
	\colorlet{Hy#1bordercolor}{Hy#1color#2}%
	\hypersetup{#1bordercolor=Hy#1bordercolor}}
\tmp{link}{!60!white}
\tmp{cite}{!60!white}
\tmp{file}{!60!white}
\tmp{url} {!60!white}
\tmp{menu}{!60!white}
\tmp{run} {!60!white}

\tikzset{sgplattice/.style={inner sep=1pt,norm/.style={red!50!blue},char/.style={blue!50!black},
		lin/.style={black!50}},cnj/.style={black!50,yshift=-2.5pt,left=-1pt of #1,scale=0.5,fill=white}}
	
	\usepackage{here}
	\usepackage{caption}
	\usepackage{subcaption}

\let\oldtocsection=\tocsection

\let\oldtocsubsection=\tocsubsection

\let\oldtocsubsubsection=\tocsubsubsection

\renewcommand{\tocsection}[2]{\hspace{0em}\oldtocsection{#1}{#2}}
\renewcommand{\tocsubsection}[2]{\hspace{1em}\oldtocsubsection{#1}{#2}}
\renewcommand{\tocsubsubsection}[2]{\hspace{2em}\oldtocsubsubsection{#1}{#2}}
\makeatletter

\newcommand{\Rmnum}[1]{\expandafter\@slowromancap\romannumeral #1@}
\makeatother

\numberwithin{equation}{section}

\newcommand{\ie}{\textit{i}.\textit{e}.}

\newcommand{\resp}{\textit{resp}.}

\input xy
\xyoption{all}

\calclayout
\allowdisplaybreaks[3]

\theoremstyle{plain}
\newtheorem{prop}{Proposition}[section]

\newtheorem{theo}[prop]{Theorem}
\newtheorem{coro}[prop]{Corollary}

\newtheorem{lemm}[prop]{Lemma}

\theoremstyle{definition}
\newtheorem{defi}[prop]{Definition}

\newtheorem{rema}[prop]{Remark}

\newtheorem{exam}[prop]{Example}

\newtheorem{nota}[prop]{Notation}

\newlist{steps}{enumerate}{1}
\setlist[steps, 1]{label = Step \arabic*:}

\def\ra{\rightarrow}

\def\cK{{\mathcal K}}

\def\cO{{\mathcal O}}

\def\tS{{\widetilde{S}}}
\def\tT{{\widetilde{T}}}

\def\rk{{\mathrm{rk}}}

\def\bC{{\mathbb C}}
\def\bF{{\mathbb F}}
\def\bG{{\mathbb G}}

\def\bP{{\mathbb P}}
\def\bQ{{\mathbb Q}}
\def\bR{{\mathbb R}}
\def\bZ{{\mathbb Z}}

\def\fS{{\mathfrak S}}
\def\fA{{\mathfrak A}}
\def\fD{{\mathfrak D}}

\DeclareMathOperator{\bfA}{\boldsymbol{\mathrm{A}}}
\DeclareMathOperator{\bfD}{\boldsymbol{\mathrm{D}}}
\DeclareMathOperator{\bfE}{\boldsymbol{\mathrm{E}}}

\def\Br{\mathrm{Br}}
\def\Bl{\mathrm{Bl}}

\DeclareMathOperator{\Eff}{\overline{Eff}}

\def\et{\mathrm{\acute{e}t}}

\def\Nef{\mathrm{Nef}}

\def\Hom{\mathrm{Hom}}

\DeclareMathOperator{\Pic}{Pic}

\DeclareMathOperator{\Sing}{Sing}

\DeclareMathOperator{\Gal}{Gal}
\def\Aut{\mathrm{Aut}}

\def\Proj{\mathrm{Proj}}

\DeclareMathOperator{\Type}{Type}
\DeclareMathOperator{\Cris}{Cris}
\DeclareMathOperator{\discrep}{discrep}
\DeclareMathOperator{\tors}{{tors}}
\DeclareMathOperator{\sm}{{sm}}

\DeclareMathOperator{\red}{{red}}
\DeclareMathOperator{\cont}{{cont}}

\def\piet{\pi_1^{\text{\'et}}}

\makeatother
\makeatletter

\author{Runxuan Gao}
\address{Graduate School of Mathematics, Nagoya University, Furocho Chikusa-ku, Nagoya, 464-8602, Japan}
\email{m20015x@math.nagoya-u.ac.jp}

\title[Examples]{Quasi-\'etale covers of Du Val del Pezzo surfaces and Zariski dense exceptional sets in Manin's conjecture}

\usepackage{microtype}
\begin{document}

\date{\today}
\begin{abstract}
	We construct first examples of singular del Pezzo surfaces with Zariski dense exceptional sets in Manin's conjecture, varying in degrees $1, 2$ and $3$.
    The obstructions arise from accumulating quasi-\'etale covers.
    We classify all quasi-\'etale covers of Du Val del Pezzo surfaces, extending earlier works of Miyanishi-Zhang.
    Then, we identify all potential examples by studying group actions on the pseudo-effective cones, and show that no such example exists in degree more than $3$.
    Relevant results on the geometry and descent problem of quasi-\'etale covers are also established, providing a systematic method to construct other examples.
\end{abstract}
\maketitle
\setcounter{tocdepth}{2}
\tableofcontents

\section{Introduction}

Manin's conjecture studies distribution of rational points on rationally connected varieties.
Let $X$ be a weak Fano variety over a number field $F$.
Let $H:X(F)\ra \bR$ be an anticanonical height function on the set of rational points of $X$.
We define the counting function
$$N(U,H,B)\colonequals\#\{ x\in U(F)\mid H(x)\leq B\}$$
for any $B\geq0$ and subset $U\subseteq X(F)$.
%
%Let $X$ be a geometrically rationally connected variety over a number field $F$. 
%A well-known conjecture of Colliot-Th\'el\`ene implies that the set $X(F)$ of rational points is Zariski dense in $X$ as soon as there exists a rational point on the smooth locus of $X$.
%So it makes sense to study the distribution of them with respect to some height.
%Let $\cL$ be an adelically metrized line bundle whose underlying line bundle $L$ is big and nef.
%There is a unique height function $H_\cL:X(F)\ra\bR$ associated to $\cL$.
%We define the counting function
%$$N(U,\cL,B):=\#\{ x\in U(F)\mid H_\cL(x)\leq B\}$$
%for any $B\geq0$ and any subset $U\subseteq X$.

The closed-set version of Manin's conjecture, first formulated in \cite{FMT89,BM90}, predicts that there exists a proper closed subset $Z$ of $X$ such that
\begin{equation}\label{eq-manin-conjecture}
	N(X\backslash Z,H,B)\sim c(X,H) \cdot B (\log B)^{\rho(X)-1},
\end{equation}
where $\rho(X)$ is the Picard rank of $X$ and $c(X,H)$ is the
Peyre's constant introduced in \cite{Peyre95,BT98}.
%We say \textit{Manin's conjecture} holds for $(X\backslash Z,H)$ when (\ref{eq-manin-conjecture}) holds.

The closed-set version of Manin's conjecture is not true in general.
The first such examples were proposed by Batyrev and Tschinkel in \cite{batyrev1996rational} (also see \cite{gao2024geometric}).
%: they construct for each $n\geq3$ a Fano variety $X$ of dimension $n$, such that
%$$N(U,H,B)\gg B(\log B)^{4-1}$$
%for any Zariski open subset $U$, whereas $\rho(X)=2$.
%See \cite{gao2024geometric} for a study of Batyrev-Tschinkel's example in the spirit of \cite{LST}.
To overcome such counterexamples, Peyre suggested in \cite{Peyre03} that the exceptional set $Z$ should be assumed to be a thin set (see Definition \ref{defi-thin set}).
Nowadays, there are increasing evidences of the thin-set version of Manin's conjecture, such as the established cases \cite{LR19, BHB20},
the geometry consistency of the conjecture \cite{LTDuke,LST},
and its compatibility with Manin's conjectures for stacks \cite{DY22,ESZ23}.

On the other hand, no counterexample in dimension $2$ was found for a long time, despite the many established cases, such as \cite{BT98toric,delaBretche2002,de2012manin}.
%Manin's conjecture has been widely studied in dimension $2$, especially for Du Val del Pezzo surfaces (\ie~surfaces that have an ample anticanonical divisor and admit only canonical singularities).
%In the later case,
%the conjecture reads as follows.
%\begin{conj}\label{conj:maninfordvdp}
%    Let $X$ be a Du Val del Pezzo surface over a number field $F$ and let $\cL=-\cK_X$ be the anticanonical line bundle of $X$ with an adelic metric.
%    Suppose that $X(F)$ is not a thin set.
%    Then there exists a thin subset $Z$ of $X(F)$ such that
%    $$N(X\backslash Z,\cL,B)\sim c \cdot B (\log B)^{\rho(\tX)-1},$$
%where $\tX\ra X$ is the minimal resolution of $X$.
%\end{conj}
%For instance, the conjecture has been established for all toric weak del Pezzo surfaces and for many special weak del Pezzo surfaces.
%The only established case of degree $\leq2$ is a Du Val del Pezzo surface of degree $2$ with a $E_7$ singularity \cite{BB13}.
See \cite{Derenthal2013} for a list of established cases in dimension $2$.
There are also new cases established after the publication of \cite{Derenthal2013}, such as \cite{manincase16,manincase17,manincase19,derenthal2020split}.

In all established cases in dimension $2$, the exceptional set is a proper closed subset: it is nothing but the union of negative curves.
In \cite{gao2023zariski}, the author construct the first counterexample to the closed-set version of Manin's conjecture in dimension $2$.
%However, there do exist smooth del Pezzo surfaces of degree $1$ for which the closed-set version of Manin's conjecture does not hold , where the $c$-constant will be violated if one merely removes a proper closed exceptional set.
It is a minimal del Pezzo surface of degree $1$, and the Zariski dense exceptional set comes from a family of singular conics parameterized by an elliptic curve.
% which can not produce a counterexample for a singular del Pezzo surface, as such a surface has $b$-invariant $\geq2$.

In this paper, we thoroughly study another class of surfaces with Zariski dense exceptional sets, arising from quasi-étale covers (i.e., étale in codimension $1$ covers; see Section \ref{sec:quasi etale}).
Similar phenomena in higher dimensions were investigated by Le Rudulier \cite{le2019points}.
Surprisingly, this also occurs frequently for weak del Pezzo surfaces.

Our starting point is the following theorem, where (1) builds on \cite{BL17}, and (2) is elementary.
\begin{theo}\label{theo-guiding-principle}
	Let $\pi:T\ra S$ be a quasi-\'etale cover between Du Val del Pezzo surfaces over a number field $F$.
	Suppose that $T(F)\neq \emptyset$.
	Let $\rho_S:X\ra S$ and $\rho_T:Y\ra T$ be minimal resolutions of singularities.
		\begin{enumerate}
		\item  Suppose that $\rho(Y)=\rho(X)$ and that $Y(F)\neq\emptyset$ is equidistributed in some dense open subset of $Y$.
		Then $X(F)$ is not equidistributed in any dense open subset of $X$.
		In particular, for any open subset $U$ of $S$, Manin's conjecture for $(U,H)$ fails for some anticanonical height $H$. \label{theo-guiding-principle(1)}
		
		\item  Let $H$ be an anticanonical height on $X$.
		Suppose that $\rho(Y)>\rho(X)$, and that
		\begin{equation}\label{eq2}
			N(V,\pi^\ast H,B)\gg B(\log B)^{\rho(Y)-1},
		\end{equation}
		for a dense open subset $V$ of $T$.
		Then we have
		$$N(U,H,B)\gg B(\log B)^{\rho(X)-1},$$ \label{theo-guiding-principle(2)}
		for any dense open subset $U$ of $S$.
	\end{enumerate}
\end{theo}

One of the key points in the above theorem is a quasi-\'etale cover with equal or higher $b$-invariant.
One would wonder if such cover $\pi$ exists. Indeed, when $S$ is split (\ie\,Galois action on $\Pic(S)$ is trivial), there is no such cover for trivial reasons.
However, when $S$ is non-split, such a cover does exist.
Thanks to the proven cases of Manin’s conjecture \cite{BT98toric,chambert2010integral,browning2006density}, we construct explicit examples satisfying the full hypotheses of Theorem \ref{theo-guiding-principle}, leading to counterexamples to the proper-closed version of Manin’s conjecture.
\begin{theo}\label{theo-examples}
	The following assertions hold true.
	\begin{itemize}
		\item 	Let $S$ be either
		\begin{enumerate}
			\item (degree $3$, type $4A_1$) the singular cubic surface in $\bP^3_\bQ$ given by
			$$X^3 + 2XYW + XZ^2 -Y^2Z + ZW^2=0,$$\label{enum:1a}
			\item (degree $2$, type $D_4+3A_1$) the hypersurface in $\bP_\bQ(1,1,1,2)$ given by
			$$W^2-XY(Z^2+Y^2)=0.$$\label{enum:1b}
		\end{enumerate}
		Then for any proper closed exceptional set, there exists an anticanonical height $H$ on $S$ such that Manin's conjecture for $S$ does not hold.
		\item Let $S$ be either
		\begin{enumerate}[resume]
			\item (degree $2$, type $3A_2$) the hypersurface in $\bP_\bQ(1,1,1,2)$ given by
			$$W^2+(3X^2-YZ)W+9X^4-6X^2YZ+X(Y^3+Z^3)=0,$$\label{enum:2a}
			\item (degree $1$, type $E_6+A_2$) the hypersurface in $\bP_\bQ(1,1,2,3)$ given by
			$$W^2+Z^3+X^4Y^2=0.$$\label{enum:2b}
		\end{enumerate}
		Then for any dense open subset $U$ and any anticanonical height $H$ on $X$,
		$$N(U,H,B)\gg B(\log B)^{\rho(X)-1}.$$
	\end{itemize}
\end{theo}
In particular, (\ref{enum:2a}) or (\ref{enum:2b}) of Theorem \ref{theo-examples} provide the first examples in dimension $2$ where Manin's conjecture fails for \textit{all} anticanonical height function.
If Manin’s conjecture holds for these examples, it follows that there exists accumulating \textit{covers}, rather than only accumulating subvarieties, on these surfaces; see Section \ref{sec:examples}.

%In particular, (\ref{enum:2a}) or (\ref{enum:2b}) of Theorem \ref{theo-examples} provide the first counterexamples to equidistribution (over a Zariski open subset) in the sense of \cite{Peyre95} in dimension $2$.
%Let $S$ be a weak del Pezzo surface over $\bQ$ and let $\tau$ be the Tamagawa measure associated to an anticanonical metric.
%Recall that $S(\bQ)$ is \textit{equidistributed} on a subset $U\neq\emptyset$ of $S(\bfA_\bQ)$ if for any open subset $W$ of $S(\bfA_\bQ)$ with $\tau(\partial S)=0$, we have
%$$\lim_{B\ra \infty}\frac{N(U\cap W,H,B)}{N(U,H,B)}=\frac{\tau(W\cap\tau(X(\bfA_\bQ)^{\Br}) )}{\tau(X(\bfA_\bQ)^{\Br})}.$$
%
%\begin{coro}\label{coro:equidistribution}
%	Let $S$ be either (\ref{enum:2a}) or (\ref{enum:2b}) in Theorem \ref{theo-examples}.
%	Then rational points are not equidistributed on any Zariski open subset of $S$ with respect to any anticanonical height.
%\end{coro}
%\begin{proof}
%	Because Manin's conjecture for all anticanonical height is equivalent to equidistribution \cite[Proposition 5.0.1]{Peyre95}.
%\end{proof}

To systematically study such examples, we classify quasi-\'etale covers of Du Val del Pezzo surfaces up to singularity types, which do not seem to exist in the existing literature.
\begin{theo}[Theorem \ref{theo:classification}]\label{theo-classification-of-covers}
	Let $S$ be a Du Val del Pezzo surface over an algebraically closed field of characteristic $0$.
	Then, isomorphism classes of quasi-\'etale covers of $S$ are in one-to-one correspondence with the paths that terminate at $\Type(S)$ in the diagrams after Theorem \ref{theo:classification}.
\end{theo}
We observe that to classify such covers is equivalent to classify non-primitive root sublattices of del Pezzo lattices (Corollary \ref{coro-primitive-sublattice}).
Theorem \ref{theo-classification-of-covers} completes the study of adjoint rigid $a$-covers (see Definition \ref{defi:adjoint rigid a cover}) in dimension $2$, which would be useful in both Manin's conjecture and Geometric Manin's conjecture.
%When $S$ is a Du Val del Pezzo surface over an arbitrary field $F$ of characteristic $0$,
%additional information is required to determine whether quasi-\'etale covers of  $S_{\overline{F}}$ descend to $S$. This will be discussed in Section ???.

For weak del Pezzo surfaces, we characterize when the ``geometric consistency" fails on the complement of a proper closed subset.
The case of non-adjoint-rigid $a$-covers was studied in \cite{LT19}, and a counterexample was constructed in \cite{gao2023zariski}.
For the adjoint rigid case,
the following theorem gives a characterization.

\begin{theo}\label{theo-geometric-consistency-fails}
	Let $X$ be a weak del Pezzo surface of degree $d$ over a field $F$ of characteristic $0$ and write $L=-K_X$.
	Then there exists an adjoint rigid $a$-cover $f:Y\ra X$ with
	\begin{equation}\label{eq-ab}
			b(Y,f^\ast L)\geq b(X,L)\quad (\resp~b(Y,f^\ast L)> b(X,L))
	\end{equation}
	if and only if there exists a quasi-\'etale cover $\overline{\pi}:\overline{Y}\ra X_{\overline{F}}$ and a subgroup $H\times G_X\in\Cris(\overline{\pi})$ satisfying $\rho(Y,H)\geq \rho(G_X)$ $(\resp~\rho(Y,H)> \rho(G_X))$.
	When $d\geq2$, the pairs $(\Type(X),G_X)$ satisfying the equivalent condition are recorded in the tables in the Appendix.
	In particular, such $f$ does not exist when $d\geq4$ (\resp~$d\geq3$).
\end{theo}

Building on the version of minimal model program developed in \cite{LTDuke,LST}, this completes the study of the same problem for all geometrically rational surfaces.
\begin{theo}\label{theo-geometrically-rational}
	Let $X$ be a smooth geometrically rational surface over a field $F$ of characteristic $0$ and let $L$ be a big and nef line bundle on $X$.
	Then any adjoint rigid $a$-cover $f:Y\ra X$ of $(X,L)$ satisfying (\ref{eq-ab})
	is birationally equivalent to a quasi-\'etale cover $\pi:T\ra S$ of a Du Val del Pezzo surface $S$ with polarization $-K_S$ satisfying (\ref{eq-ab}).
\end{theo}

We introduce a model to study quasi-\'etale covers of normal surfaces (Proposition \ref{prop:model}).
In Theorem \ref{theo-guiding-principle}, the Picard rank of $X$ and $Y$ are determined by the Galois action on the negative curves on them.
A natural question is whether the image or preimage of a negative curve remains a negative curve.
For Du Val del Pezzo surfaces, we give a complete answer to this question (see Definition \ref{defi-star-line}).
\begin{theo}[$\subset$ Theorem \ref{theo-corr}]\label{theo-sub-of-theo-corr}
	Let $S$ be a Du Val del Pezzo surface over $\bC$ and let $\pi:T\ra S$ be a quasi-\'etale cover of prime degree $p$.
	\begin{enumerate}
		\item $\pi:T\ra S$ maps a line $C$ to a line, except when $C$ is contained in a $p$-circle of lines: in this case, $C$ is mapped to a nodal rational curve in $\lvert -K_S\rvert$.
		\item The preimage of a line under $\pi:T\ra S$ is a disjoint union of $p$ lines, a $3$-star of lines, or a curve with non-negative self-intersection.
	\end{enumerate}
\end{theo}
As a corollary, the induced map $Y\dasharrow X$ sends an extremal ray of $\Eff^1(Y)$ to an extremal ray of $\Eff^1(X)$ when the degree of $S$ is at least $2$.
On the other hand, there are many counterexamples when $S$ is of degree $1$, see Remark \ref{rema:counterexample}.

Another problem is how to descend $\pi$ from $\mathbb{C}$ to a non-closed field with prescribed Galois action.
In Section \ref{sec:descend1} and \ref{sec:descend}, we study the existence of a cover $\pi:\overline{Y}\ra \overline{X}$ of fixed type and splitting group over a non-closed field $F$, when either $\overline{Y}$ or $\overline{X}$ is defined over $F$.
Using the language of finite descent type of Harpaz and Wittenberg \cite{harpaz2024supersolvable}, we show the following result, which allows one to prove the existence of an accumulating quasi-étale cover from the database in Appendix \ref{sec:app}, without knowing its defining equation.
\begin{prop}[Proposition \ref{prop-descend}]
	Let $S$ be a Du Val del Pezzo surface over a field $F$ of characteristic $0$ and let $\overline{T}$ be a finite descent type over $S_{\sm}$.
	Suppose $S_{\sm}(F)\neq\emptyset$.
	Then there exists a quasi-\'etale cover $\pi:T\ra S$ of type $\overline{T}$ such that
	$\Gal(\overline{F}/F)$ acts trivially on the quotient group $\Pic(T^\prime_{\overline{F}})_\bQ/f^\ast\Pic(X_{\overline{F}})_\bQ$,
	where $T^\prime$ denotes the normalization of $S$ in $F(T)$.
\end{prop}

\noindent
{\bf Notations.}
Throughout this paper, we will use the following notations:
\begin{itemize}
%	\item $C_n$ is a cyclic group of order $n$.
%	\item $\fS_n$ is a symmetric group on $n$ letters.
%	\item $\fD_n$ is a dihedral group of order $2n$.
%	\item $H\wr S$ is the wreath product of groups $H$ and $S\subseteq \fS_n$.
%	\item $A_{\tors}$ is the torsion subgroup of an abelian group $A$.
%	\item $\rho(X)$ is the Picard rank of a variety $X$.
%	\item $K_X$ is the canonical line bundle of a variety $X$.
%	\item $\cL$ is an adelically metrized line bundle whose underlying line bundle is $L$. 
	\item $f(B)\sim g(B)$ means that $\lim_{B\ra \infty}({f(B)}/{g(B)})=1$.
	\item $f(B)\ll g(B)$ means that $\limsup_{B\ra\infty}({f(B)}/{g(B)})<+\infty$.
	\item $f(B)\asymp g(B)$ means that $f(B)\ll g(B)$ and $g(B)\ll f(B)$.
	%	$\exists~k,\exists~ B_0$, $\forall~B>B_0$: $ \left|f(B)\right|\leq kg(B)$.
	% \item $f(B)\asymp g(B)$ means $f(B)\ll g(B)\ll f(B)$.
\end{itemize}

\noindent
{\bf Acknowledgments.}
The author would thank his advisor Sho Tanimoto for constant support, and Tim Browning, Ulrich Derenthal, Daniel Loughran, J\'anos Koll\'ar and Yuri Tschinkel for useful conversations and remarks.
The author was partially supported by JST FOREST program Grant Number JPMJFR212Z,
JSPS Bilateral Joint Research Projects Grant Number JPJSBP120219935
and Grant-in-Aid for JSPS Fellows Number 24KJ1234.

\section{Preliminary}\label{sec-preliminary}
In this section, we recall well-known facts on Manin's conjectures and weak del Pezzo surfaces.
\subsection{Thin set}\label{sec-thin-set}
The notion of thin sets was proposed by Serre in the study of inverse Galois problem, see \cite{serre2016topics}.
\begin{defi}\label{defi-thin set}
	Let $X$ be a variety over a field $F$.
	A subset $Z$ of $X(F)$ is called \textit{thin} if it is contained in a finite union of sets in the following two types:
	\begin{itemize}
		\item $Y(F)$ for a closed subscheme $Y$ of $X$ with dimension less than $X$.
		\item  $f(Y(F))$ for a variety $Y$ and a dominant morphism $f:Y\to X$ of degree at least $2$.
	\end{itemize}
\end{defi}

\subsection{$a$- and $b$-invariants}
The $a$- and $b$-invariants are defined as follows.
\begin{defi}\label{defi-a}
	Let $X$ be a projective variety over a field of characteristic $0$ and $L$ be a big and nef $\bQ$-divisor on $X$.
	If $X$ is smooth, the $a$-invariant is defined as
	$$a(X,L):=\inf\{ t\in\bR\mid K_X+tL\in \Eff^1(X)\},$$
	and the $b$-invariant $b(X,L)$ of $X$ is the dimension of the linear space generated by
	$$\{\alpha\in\Nef_1(X) \mid (K_X+a(X,L)L)\cdot\alpha=0\}.$$
	If $X$ is singular, then we define the two invariants by $(\widetilde{X},\beta^\ast L)$ for a resolution $\beta:\widetilde{X}\ra X$.
\end{defi}
\begin{rema}
	By \cite[Proposition 2.7]{HTT15} and \cite[Proposition 4.13]{LST}, the definitions are independent of the choice of $\beta$.
\end{rema}

\begin{defi}\label{defi:adjoint rigid a cover}
	Under the above settings,
	we say a dominant thin morphism $f:Y\ra X$ is an \textit{$a$-cover} if
	$a(Y,f^\ast L)=a(X,L)$.
	We say $f$ is an \textit{adjoint rigid $a$-cover} if, further, $\kappa(Y,K_Y+a(Y,f^\ast L)f^\ast L)=\kappa(X,K_X+a(X,L)L)$.
\end{defi}

\subsection{Quasi-\'etale cover}\label{sec:quasi etale}
A variety over a ﬁeld $F$ is an integral separated scheme of ﬁnite type over $F$.
Let $X$ be a normal variety.
A \textit{cover} of $X$ is a morphism $\pi:Y\ra X$ where $Y$ is a normal variety, and $\pi$ is finite and surjective.
A \textit{morphism of covers} $\pi:Y\ra X$ and $\pi^\prime:Y^\prime\ra X$ is a morphism $f:Y\ra Y^\prime$ such that $\pi^\prime\circ f=\pi$.
A cover $\pi:Y\ra X$ of varieties is called \textit{Galois} if $\#\Aut(Y/X)=\deg \pi$.
We have $\Aut(Y/X)=\Aut(k(Y)/k(X))$, and $\pi$ is Galois if and only if the field extension $k(Y)/k(X)$ is Galois.
A morphism $f:Y\ra X$ between normal varieties is called \textit{quasi-\'etale} if $f$ is quasi-finite and \'etale in codimension $1$.
A cover between integral normal varieties is either \'etale or ramified.
Thus a cover is quasi-\'etale if there is no branch divisor.

We have the following geometric version of Galois correspondence.
\begin{prop}[\cite{corvaja2022} Proposition 2.6]%need quasi-projective?
	Let $\pi:Y\to X$ be a Galois cover with Galois group $G$.
	Then there is a one-to-one correspondence between subgroups $H$ of $G$ and intermediate covers $Y\to Z\to X$ with $Z$ normal and integral, given by:
	\begin{itemize}
		\item Given $H$, define $Z$ as the normalization of $X$ in $k(Y)^H$. The induced morphism $Y\to Z$ is a geometric quotient, and we may also denote $Z$ by $Y/H$.
		\item Given $Z$, define $H$ as the Galois group $\Gal(k(Y)/k(Z))$.
	\end{itemize}
\end{prop}

To study quasi-\'etale covers of $X$, it is sufficient to study the fundamental group of the smooth part of $X$.
\begin{lemm}\label{lemm-smooth-part}
	Let $X$ be a normal projective variety and let $X^\circ$ be any open subvariety of $X\backslash\Sing(X)$ that is isomorphic in codimension $2$.
	Then there is a one-to-one correspondence between isomorphic classes of quasi-\'etale covers of $X$ and that of \'etale covers of $X^\circ$.
\end{lemm}
\begin{proof}
	Since $X$ is normal, $\Sing(X)$ has codimension $\geq2$ in $X$.
	Let $\pi:Y\ra X$ be an quasi-\'etale cover.
	By Zariski-Nagata purity, $\pi$ is \'etale over $S\backslash\Sing(X)$.
	Conversely, let $\pi^\circ:Y^\circ\ra X^\circ$ be an \'etale cover.
	By Zariski's main theorem, $\pi^\circ$ can be uniquely extended to a finite morphism $\pi:Y\ra X$.
	The variety $Y$ is projective since $X$ is.
	So $\pi$ is surjective and in particular a quasi-\'etale cover.
\end{proof}
%\begin{rema}\label{rema:Zariski main theorem}
%	By Zariski’s main Theorem in the equivariant setting \cite[Theorem 3.8]{GKP16},
%	if the cover $\pi^\circ:Y^\circ\ra X^\circ$ is Galois with group $G$, then $\pi:Y\ra X$ is also Galois with group $G$, \ie~the group action of $G$ on $X^\circ$ can be extended to $X$ uniquely.
%\end{rema}

\subsection{Weak del Pezzo surface}\label{sec1.1-dvdp}
In the following three sections, we recall basic properties of weak del Pezzo surfaces.
We refer to \cite{derenthal2008nef,Derenthal2013,Dolgachev2012} for more details.
In this section, we work over an arbitrary field $F$ of characteristic $0$.
A \textit{Du Val del Pezzo surface} $S$ is a projective geometrically integral surface with ample anticanonical line bundle and at worst canonical singularities.
When $S$ is smooth, we call it a \textit{del Pezzo surface}.
A \textit{weak del Pezzo surface} is a smooth projective geometrically integral surface with big and nef anticanonical line bundle.
Let $X$ be a surface in any of the above classes.
Then there is an intersection pairing on $\Pic(X_{\overline{F}})$, and we call the self intersection of the canonical class $K_X$ the \textit{degree} of $X$.

Let $X$ be a weak del Pezzo surface.
The Proj construction of the graded ring
$$\bigoplus_{m\geq 0} H^0(X,-mK_X)$$
induces a birational morphism $\rho:X\ra S$, which is known as the \textit{anticanonical model map}.
The image $S$ is a Du Val del Pezzo surface and conversely, for any $S$ a Du Val del Pezzo surface, $\rho:X\ra S$ is the minimal resolution of $S$.

%Let $X$ be a weak del Pezzo surface.
%We call $d:=K_X^2$ the degree of $X$.
%It turns out that $1\leq d\leq 9$ and $X_{\overline{}}$ is isomorphic to $\bP^1_{\overline{F}}\times\bP^1_{\overline{F}}$,
%the Hirzebruch surface $\bF_{2}$,
%or the blow up of $\bP^2_{\overline{F}}$ at $9-d$ points in almost general positions,
%where $\overline{F}$ is an algebraic closure of $F$.
%The Proj construction of the graded ring
%$$\bigoplus_{m\geq 0} H^0(X,-mK_X)$$
%defines a birational morphism $X\ra S$ which is known as the anticanonical model map.
%It turns out that $S$ is a Du Val del Pezzo surface and $X\ra S$ coincides with the minimal resolution of $S$.
%Conversely, the minimal resolution of any Du Val del Pezzo surface is a weak del Pezzo surface \cite[Theorem 8.3.2]{Dolgachev2012}.
%So we may abuse the notation of a Du Val del Pezzo surface $S$ and its associated weak del Pezzo surface $X$.

\subsection{Geometry of weak del Pezzo surface}
In this section, we work over an algebraically closed field $F$ of characteristic $0$.
Let $X$ be a weak del Pezzo surface of degree $d$.

For $n\in\bZ$, an integral curve $C$ on $X$ is called an \textit{$(n)$-curve} if $C^2=n$ and $K_X\cdot C=2-n$,
and a divisor class $D\in\Pic(X)$ is called an \textit{$(n)$-class} if $D^2=n$ and $K_X\cdot D=2-n$.
By Adjunction Formula, an $(n)$-curve is isomorphic to $\bP^1$.
We say $C$ or $D$ is \textit{negative} if $n<0$.
We say a class $D\in\Pic(X)$ is \textit{effective} if it contains an effective divisor $C$.
Here, $C$ is unique when $D$ is a negative class.
Every $(-1)$-class is effective, and it is integral if and only if $D\times C\geq 0$ for every $(-2)$-curve $C$ of $X$.
The only negative curves on $X$ are the $(-1)$- and $(-2)$-curves,
and the anticanonical model map $\rho:X\ra S$ exactly contracts the $(-2)$-curves.
A $(-1)$-curve is also called a \textit{line}.

Except when $d=8$ and $X$ is the Hirzebruch surfaces $\bF_0=\bP^1\times\bP^1$ or $F_2$,
we have $1\leq d\leq 9$ and a choice of $9-d$ linear disjoint lines $E_1,\dots,E_{9-d}$ on $X$ defines a morphism $\varphi:X\ra \bP^2$ which is a blow up at $9-d$ points in almost general position.
This choice induces an orthogonal basis $H,E_1,\dots,E_{9-d}$ of $\Pic(X)$, where $H$ is the pullback of the hyperplane class in $\bP^2$ by $\varphi$.

Let $C$ be an integral curve on $X$.
We call $C$ an \textit{extremal curve} if the class of $C$ generates an extremal ray of the pseudo-effective cone $\smash{\Eff^1(X)}$.
When $d\leq7$, the extremal curves on $X$ are precisely the negative curves; but when $d=8$ or $9$, they can also include $(0)$- or $(1)$-curves.

The \textit{dual graph} of $X$, denoted by $\Gamma(X)$, is a graph where each vertex represents an extremal curve class with weight equal to its self-intersection number, and each edge connects two classes with positive intersection, with weight equal to the intersection number.

It is elementary to check the following lemma.
\begin{lemm}\label{lemm-dP2}
	Let $X$ be a weak del Pezzo surface.
	\begin{enumerate}
		\item Let $C$ be a $(-1)$-curve and $D$ be a connected union of $(-2)$-curves on $X$.
		Then $C\cdot D\leq2$ and at most one $D$ achieves equality.
		In particular, $C\cdot D=2$ implies that $D$ is a maximal linear chain of $(-2)$-curves on $X$, and that $C$ intersects $D$ at the two endpoints of $D$. 
	\end{enumerate}
	Suppose $X$ is of degree at least $2$.
	\begin{enumerate}[resume]
		\item There are no $(-1)$-curves $C_1$ and $C_2$ with $C_1\cdot C_2\geq3$.
		\item There are no three $(-1)$-curves $C_i$ with $C_i\cdot C_j\geq2$ for $i=1,2,3$.
	\end{enumerate}
\end{lemm}

\subsection{Types}\label{sec1.2-types}

A \textit{lattice} is a free abelian group with an integral-valued symmetric bilinear form.
An \textit{isometry} between two lattice is a linear map that preserve the symmetric bilinear form.
We recall the notion of Cremona isometry appeared in \cite[Section 8.2.8]{Dolgachev2012}.
\begin{defi}\label{def-cremona}
	Let $X$ and $X^\prime$ be two weak del Pezzo surfaces.
	Then an isometry $\sigma:N^1 (X)\ra N^1 (X^\prime)$ is called a \textit{Cremona isometry} if $\sigma$ restricts to a bijection between $\Eff^1(X)$ and $\Eff^1(X^\prime)$ which sends $[K_X]$ to $[K_{X^\prime}]$.
	We say that $X$ and $X^\prime$ have the same \textit{type} whenever there exists a Cremona isometry between them.
	We write $\Cris(X)$ for the group of Cremona isometries of $X$ to itself.
\end{defi}
A vector $v$ of a lattice is called a \textit{root} if $v^2=-2$.
A negative definite lattice generated by roots is called a \textit{root lattice}.
Any root lattice can be decomposed into irreducible ones.
These are denoted by $\bfA_n\,(n\geq1)$, $\bfD_n\,(n\geq4)$, $\bfE_n\,(6\leq n\leq 8)$.

Let $X$ be a weak del Pezzo surface of degree $d\leq7$ over an algebraically closed field $F$ of characteristic $0$.
Let $R_d$ be the orthogonal complement of the sublattice $\bZ[K_X]$ of $\Pic\,(X)$ when $d\leq6$, and be that of $\bZ[K_X]+\bZ[H]$ when $d=7$.
It is a root lattices recorded in Table \ref{tab-root-lattice}.
We call $R_d$ the \textit{del Pezzo lattice} of degree $d$.
\begin{table}[htbp]
	\begin{tabular}{cccccccc}
		\hline
		$d$&7&6&5&4&3&2&1 \\
		\hline
		$R_d$&$\bfA_1$&$\bfA_2+\bfA_1$&$\bfA_4$&$\bfD_5$&$\bfE_6$&$\bfE_7$&$\bfE_8$\\
		\hline
	\end{tabular}
	\caption{The root lattices $R_d$}
	\label{tab-root-lattice}
\end{table}
For simplicity, from now on, we assume that $d\leq6$.
The $(-2)$-curves on $X$ generate a root sublattice $R$ of $R_d$.
Two weak del Pezzo surfaces have the same type if and only if their associated root sublattices in $R_d$ are isomorphic.
Conversely, let $R\subseteq R_d$ be a root sublattice and let $X$ be a weak del Pezzo surface.
A \textit{marking} on $X$ an isomorphism $R_d\to (\bZ[K_X])^\perp$ of lattices which restricts to an isomorphism between $R$ and the sublattice of $\Pic(X)$ generated by $(-2)$-curves.
Every root sublattice of $R_d$ corresponds to a weak del Pezzo surface, except for the following four cases: $7\bfA_1$ in degree $2$, and $7\bfA_1,8\bfA_1$ and $\bfD_4+4\bfA_1$ in degree $1$.
(Weak del Pezzo surfaces of these four types do exist in characteristic $2$.)
%The classification of types of weak del Pezzo surfaces has been done in degree $3$ by Schl\"afli \cite{Sch1863} and Cayley \cite{Cay1869}, and in degrees $1$ and $2$ by Du Val \cite{DuVal1934}.

Two root sublattices of $R_d$ may be isomorphic yet fail to be isomorphic as sublattices of $R_d$.
However, the number of lines on the surface is sufficient to distinguish between these cases.
This leads to the following notations.
\begin{nota}\label{nota:type}
    We denote by $S_d(\mathbf{ADE})$ for the type of a Du Val del Pezzo surface of degree $d$ with ADE type $\mathbf{ADE}$.
    When the degree and the ADE type are not enough to describe the type, we add the number of lines.
For example, $S_4(2\bfA_1(8l))$ (\resp\,$S_4(2\bfA_1(9l))$) denote the type of Du Val del Pezzo surfaces of degree $4$ with two $\bfA_1$ singularities and $8$ lines (\resp\,$9$ lines).
We abuse the same notations for the corresponding weak del Pezzo surfaces.
We use $S_d$ to denote the type of a (smooth) del Pezzo surface of degree $d$.
\end{nota}
From the arguments above, the types of Du Val del Pezzo surfaces of degree  $d \leq 7 $ correspond bijectively to the root sublattices of $R_d$.
For $d=8$ or $9$, representatives of each type are $\bP^2,\bF_0=\bP^1\times\bP^1,\bF_1=\Bl_1\,\bP^2,$ and $\bF_2$.

\section{Constructions}
In this section, we establish a general framework for the study of quasi-\'etale covers of Du Val del Pezzo surfaces.

\subsection{fundamental group of smooth part}\label{sec1.3-fundamental-groups}
We fix a field $F$ of characteristic $0$ and write $\overline{F}$ for its algebraic closure.

\begin{lemm}\label{lemm-T-is-DVdP}
	Let $\pi:T\ra S$ be a quasi-\'etale cover of a Du Val del Pezzo surface $S$ over $F$.
	Then $T$ is a Du Val del Pezzo surface.
\end{lemm}
\begin{proof}
	Since $\pi$ is finite, the Hurwitz formula $K_T\sim \pi^\ast K_S$ implies that $-K_T$ is ample.
	Since $\discrep(T)\geq \discrep(S)=0$ \cite[Proposition 5.20 (3)]{Kollar-Mori},
	$T$ has at worst canonical singularities.
\end{proof}

By Lemma \ref{lemm-smooth-part}, quasi-\'etale covers of $S$ correspond to \'etale covers of the smooth part $S_{\sm}$, which are parametrised by the \'etale fundamental group of $S_{\sm}$.
To proceed, we first determine the fundamental group and classify the étale covers over the complex field.

In \cite{MZ88,MZ93}, Miyanishi and Zhang computed $\pi_1(S_\mathrm{sm})$ and the singularity type of the universal cover for $S$ which is relative minimal.
They concluded that $\pi_1(S_\mathrm{sm})$ is a finite abelian group since the group is preserved under MMP steps.
Their computation made use of the $\bP^1$-bundle structures.
We provide a more direct method than theirs to compute $\pi_1(S_\mathrm{sm})$, which can be easily implemented on a computer.
\begin{prop}\label{prop:fundamental group}
	Let $S$ be a Du Val del Pezzo surface over $\bC$ and $\rho:X\ra S$ be the minimal resolution.
	Write $R$ for the sublattice of $\Pic(X)$ generated by the $(-2)$-curves on $X$.
	Then
	$$\pi_1(S_\mathrm{sm})\cong (\Pic(X)/R)_{\tors}.$$
\end{prop}

\begin{proof}
%	A standard argument reduces us to the case of $S$ over the complex field $\bC$.
	Let $E$ be the exceptional divisor of $\rho$.
	Consider the long exact sequence:
% https://q.uiver.app/#q=WzAsNixbMSwwLCJIXzIoRDtcXGJaKSJdLFsyLDAsIkhfMihYO1xcYlopIl0sWzMsMCwiSF8yKFgsRTtcXGJaKSJdLFs0LDAsIkhfMShEO1xcYlopIl0sWzAsMCwiXFxjZG90cyJdLFs1LDAsIlxcY2RvdHMiXSxbMCwxLCJpIl0sWzEsMl0sWzIsM10sWzQsMF0sWzMsNV1d
\[\begin{tikzcd}[column sep=small]
	\cdots & {H_2(E;\bZ)} & {H_2(X;\bZ)} & {H_2(X,E;\bZ)} & {H_1(E;\bZ)} & \cdots
	\arrow[from=1-1, to=1-2]
	\arrow["i", from=1-2, to=1-3]
	\arrow[from=1-3, to=1-4]
	\arrow[from=1-4, to=1-5]
	\arrow[from=1-5, to=1-6]
\end{tikzcd}\]
	Write $E=\sum_{i=1}^{r}E_i$ for the irreducible decomposition of $E$.
	Then $H_2(E;\bZ)$ is isomorphic to the free $\bZ$-module generated by the class of $E_i$.
	Since the matrix $(E_i\cdot E_j)$ is negative definite \cite[Lemma 3.40]{Kollar-Mori},
	$E_i$ are linearly independent in $H_2(X;\bZ)$.
	So the map $i$ is injective.
	Note that $H_1(E;\bZ)=0$ and $H_2(X,E;\bZ)\cong H^2(S_{\sm};\bZ)$ by Lefschetz duality.
	We conclude that $H^2(S_{\sm};\bZ)\cong \Pic(X)/R$, where $ \Pic(X)\cong H^2(X;\bZ)$ since $X$ is a rational surface.
	By the universal coefficient theorem, we have that $H^2(S_{\sm};\bZ)_{\tors}\cong H_1(S_{\sm},\bZ)_{\tors}$.
	
	To conclude, it suffices to show that $\pi_1(S_{\sm})$ is an abelian group.
	The following short proof was observed by Koll\'ar \cite[Remark 28]{Kol2009}.
	Let $\pi:T\ra S$ be a quasi-\'etale cover.
	Let $C$ be a general member of $\left|-K_S\right|$ which is nonsingular.
	By the adjunction formula, $C$ is a smooth curve of genus $1$.
	The cover $\pi$ restricts to an \'etale cover of $C$.
	Therefore, $\Aut(T/S)$ is a quotient of the fundamental group of $C$, which are abelian. The assertion follows.
\end{proof}
%\begin{rema}
%	There is a version of Proposition \ref{prop:fundamental group} for weak Fano pairs with klt singularities, see \cite[Corollary 4]{braun2021local}.
%\end{rema}
Proposition \ref{prop:fundamental group} has a conceptual interpretation by lattice theory.
\begin{coro}\label{coro-primitive-sublattice}
	Let $S$ be a Du Val del Pezzo surface of degree $d$ over $F$,
	whose type is given by the root sublattice $R$ of $R_d$.
	Then $S$ admit a nontrivial quasi-\'etale cover if and only if $R$ is a non-primitive sublattice of $R_d$.\qed
\end{coro}

We then compute the fundamental group for all types of $S$ using $\mathtt{Magma}$ \cite{magma}, and the result will be presented in Theorem \ref{theo:classification}.
A similar computation was done by Bright \cite{bright2013brauer}, which was to determine the Brauer group of $S$.
By the following corollary, one may consult the tables in \cite{bright2013brauer} for the fundamental group of $S$ for each types.
\begin{coro}\label{coro-Brauer-group}
	Let $S$ be a Du Val del Pezzo surface over $F$.
	Then $\piet(S_\mathrm{sm})\cong \Br(S)$.
\end{coro}
\begin{proof}
	Combine Proposition \ref{prop:fundamental group} and \cite[Corollary 5]{bright2013brauer}.
\end{proof}

In fact, the above computation tells us more.
Any torsion element $B$ of $\Pic(X)/R$ induces a cyclic cover of $X$ with branch divisor $B$.
Since quasi-\'etale covers of Du Val singularities are well-known (see \cite[Example]{Artin_1977} for the degree $2$ case), the divisor $B$ determines the singularity type of $Y$.
When there are two types associated to the singularity type, further analysis is required to determine $\Type(Y)$.
%It will be clear after Proposition \ref{prop:model} that $\Type(Y)$ is uniquely determined by $\Type(X)$ and $B$.

\begin{defi}
	Let $S$ and $S^\prime$ be Du Val del Pezzo surfaces of degree $d\leq4$ of the same type, corresponding to the root sublattice $R\subseteq R_d$ under some fixed markings.
	We say two quasi-\'etale covers $\pi:T\ra S$ and $\pi:T^\prime\ra S^\prime$ are of the same \textit{type} if they are induced by the same torsion element of $R_d/R$.
	We denote the type of $\pi$ by $\Type(\pi)$.
\end{defi}
We assume $d\leq 4$ since there is no nontrivial quasi-\'etale cover of $S$ when $d\geq5$.

\subsection{A model of quasi-\'etale covers}\label{sec1.4-negative}

We work over an arbitrary field $F$ of characteristic $0$.
Let $\pi:T\ra S$ be a quasi-\'etale cover of normal projective surfaces and let $Y\ra T$ and $X\ra S$ be the minimal resolutions.
Let $f:T^\prime\ra X$ be the normalization of $X$ in the function field of $T$.
Let $g:T^\prime\dashrightarrow Y$ be the rational map defined by composition.
Denote by $\rho_{T^\prime}: T_0 \to T^\prime$ the minimal resolution of $T^\prime$.

The following proposition can be seen as a global version of \cite[Proposition 1.5]{Artin_1977}.
\begin{prop}\label{prop:model}
	The birational map $g:T^\prime\dashrightarrow Y$ is a morphism.
	The ramification locus of $f$ equals to the exceptional divisor of $g$.
	% https://q.uiver.app/#q=WzAsNixbMiwwLCJZIl0sWzMsMCwiVCJdLFszLDEsIlMiXSxbMiwxLCJYIl0sWzEsMCwiVF5cXHByaW1lIl0sWzAsMCwiVF8wIl0sWzQsMCwiZyJdLFswLDMsIlxcd2lkZXRpbGRle1xccGl9IiwwLHsic3R5bGUiOnsiYm9keSI6eyJuYW1lIjoiZGFzaGVkIn19fV0sWzQsMywiZiIsMl0sWzMsMiwiXFxyaG9fUyJdLFsxLDIsIlxccGkiXSxbMCwxLCJcXHJob19UIl0sWzUsNCwiXFxyaG9fe1ReXFxwcmltZX0iXV0=
	\begin{equation}
		\begin{tikzcd}\label{diag-model}
			{T_0} & {T^\prime} & Y & T \\
			&& X & S
			\arrow["{\rho_{T^\prime}}", from=1-1, to=1-2]
			\arrow["g", from=1-2, to=1-3]
			\arrow["f"', from=1-2, to=2-3]
			\arrow["{\rho_T}", from=1-3, to=1-4]
			\arrow["{\widetilde{\pi}}", dashed, from=1-3, to=2-3]
			\arrow["\pi", from=1-4, to=2-4]
			\arrow["{\rho_S}", from=2-3, to=2-4]
		\end{tikzcd}
	\end{equation}
\end{prop}

\begin{proof}
	By Zariski's main theorem, $T^\prime$ factors through $T$ by a morphism,
    and a divisor contracted by $T^\prime\ra T$ must lie above the singular locus of $S$.
	So let $D$ be a prime divisor above $T$ and let $C$ denote its image above $S$.
	Let $r$ be the ramification index along $D$.
	Then $a(D,T)+1=r(a(C,S)+1)$, where $a(D,T)$ means the discrepancy of $D$ with respect to $T$; see the proof of \cite[Proposition 5.20]{Kollar-Mori}.
	Since $S$ has only canonical singularities, we have $a(C,S)\geq0$.
	Thus $a(D,T)=0$ if and only if $r=1$ and $a(C,S)=0$.
    The prime divisors above $S$ with discrepancy $0$ are exactly the exceptional divisors of $X\ra S$.
    And similarly, the prime divisors above $T$ with discrepancy $0$ are exactly the exceptional divisors of $Y\ra T$.
    Thus a prime divisor $D$ above $T$ is contracted by $g$ if and only if it is ramified above a $(-2)$-curve on $X$.
	Hence $g$ is well-defined and precisely contracts the ramification divisor of $f$.
\end{proof}

\begin{lemm}\label{lemm:ramification-of-model}
	Assume the quasi-\'etale cover $\pi$ is cyclic of prime degree $p$.
	Let $B$ be the branch divisor of $f$.
	Then $f$ is totally ramified, and each connected component of $B_{\red}$ is a chain of $(-2)$-curves of length $p$.
	Let $R\coloneqq f^\ast B$ be the ramification divisor.
	Then $R_{\red}$ is a disjoint copy of a chain of $(-1)$-curves of length $p$.
%	and
%	$p\in T^\prime$ is a singular point of $T^\prime$ if and only if $p$ is a singular point of $R_{\red}$.
	When $p$ is odd, the central point of the $R_{x,\red}$ is a singular point of $T^\prime$ of type $1/p(1,1)$.
\end{lemm}
\begin{proof}
	The ramification index of $\pi:T^\prime\ra X$ at any point of $X$ divides the degree $p$ of $\pi$.
	Thus $\pi$ is totally ramified at its ramification locus.
	Write $G\coloneqq \bZ/p\bZ$.
	The $G$-action on $T$ extends to a $G$-action on the smooth surface $Y$.
	The stabilizer subgroup of $G$ at any point of $y \in Y$ is either the trivial group or $G$.
	Consequently, the image $x$ of $y$ in the quotient $Y/G$ is either a smooth point or a cyclic quotient singularity of degree $p$.
	Since $\rho_S:X\ra S$ resolves $x$ with exceptional divisors of discrepancy $0$, we conclude that $x$ is a Du Val cyclic singularity, \ie\, an $\bfA_p$ singularity.
	
	The exceptional divisor $B_{x,\red}$ above $x$ is a chain of $(-2)$-curves of length $p$, which is a connected component of $B_{\red}$, any every connected component of $B_{\red}$ are obtained in this way.
	By projection formula, the preimage of $B_{x,\red}$ in $T^\prime$ is a chain of $(-1)$-curves of length $p$.
	It is a standard property of cyclic covers that the singular points of $T^\prime$ are exactly that of $R_{\red}$.
	Let $B_i$ denotes the $i$-th $(-2)$-curve on $B_{x,\red}$.
	Then up to the order, the branch divisor $B_x$ above $x$ is $B_1+2B_2+\dots+(p-1)B_{p-1}$.
	When $p$ is an odd prime, there is a central singular point $x_0$ of $B_{x,\red}$ which is the intersection point of $B_{(p-1)/2}$ and  $B_{(p+1)/2}$.
	Thus the singular point $y_0$ above $x_0$ is defined by $z^p-x^{(p-1)/2}y^{(p+1)/2}=0,$
	which is analytic-locally isomorphic to $z^p-xy^{p-1}=0$, \ie\, an $1/p(1,1)$-singularity.
\end{proof}

\begin{exam}
	Artin \cite[Example]{Artin_1977} described the branch locus of every double quasi-\'etale covers of Du Val singularity,
	which verifies Lemma \ref{lemm:ramification-of-model} in the case of $p=2$: the branch locus $B_{\red}$ is a disjoint union of $(-2)$-curves.

	When $p=3$, the only singular points on $T^\prime$ are $\frac{1}{3}(1,1)$-singularities,
	whose exceptional divisor is a single $(-3)$-curve.
	
	When $p=5$, the situation reduced to an $A_4$-singularity being covered by a smooth point.
	The three singular points of $B_{x,\red}$ are given by
	\begin{align*}
		z^5-xy^2=0,\quad x^5-x^2y^3=0,\quad z^5-x^3y^4=0,
	\end{align*}
	respectively.
	They are analytic-locally isomorphic to the singularities defined by
	$$z^5-xy^2=0,\quad x^5-xy^4=0,\quad z^5-xy^3=0,$$
	respectively.
	By Hirzebruch–Jung continued fractions (see \cite{barth2015compact}), the dual graph of the preimage of $x$ in $T_0$ is
	\[\begin{tikzpicture}
		\def \x {1.5};
		\node (01) at (0,0) {$(-1)$};
		\node (02) at (\x,0) {$(-2)$};
		\node (03) at (2*\x,0) {$(-3)$};
		\node (04) at (3*\x,0) {$(-1)$};
		\node (05) at (4*\x,0) {$(-5)$};
		\node (06) at (5*\x,0) {$(-1)$};
		\node (07) at (6*\x,0) {$(-3)$};
		\node (08) at (7*\x,0) {$(-2)$};
		\node (09) at (8*\x,0) {$(-1),$};
		\draw (01)--(02)--(03)--(04)--(05)--(06)--(07)--(08)--(09);
	\end{tikzpicture}\]
	where the $(-1)$-curves form the ramification locus.
\end{exam}

\begin{rema}\label{rema-resolution-of-ramification}
	The birational morphism $\rho_{T}\circ g\circ\rho_{T^\prime}:T_0\ra T$ can be seen as the minimal resolution of ramification of $\pi:T\ra S$, in the sense that the fixed locus of $\sigma$ on $T_0$ is of purely $1$-dimensional, and $T_0$ is the minimal smooth surface above $T$ with this property.
\end{rema}

\subsection{Correspondence of extremal curves}\label{sec:corr}

To find quasi-\'etale covers of Du Val del Pezzo surfaces with equal or higher $b$-invariants, we first investigate the correspondence between extremal curves under the cover.
\begin{defi}\label{defi-corr}
	Using the notations from diagram (\ref{diag-model}),
	We define
	 $$\Gamma(\pi)=\left\{ (C,D)\in \Gamma(Y)\times\Gamma(X)\middle\vert 
		\widetilde{\pi}(C)= D \right\}.$$
	We call $\Gamma(\pi)$ the \textit{correspondence of extremal curves} induced by $\pi$.
\end{defi}

We define some special configurations of $(-1)$-curves on surfaces.
\begin{defi}\label{defi-star-line}
	Let $X$ be a normal projective surface, and let
	$Z\coloneqq\bigcup_{i\in\bZ/d\bZ} C_i$
	be a union of $(-1)$-curves $C_i$ for $d\geq2$.
	Suppose $C_i$ and $C_{i+1}$ intersects transversely at one point $p_i$ and there is no other intersection among $C_i$ for each $i\in\bZ/d\bZ$.
	We call $Z$ a \textit{$d$-circle} if $p_i$ $(i\in\bZ/d\bZ)$ are distinct points.
	We call $Z$ a \textit{$d$-star} if $p_i$ are the same point for all $i\in\bZ/d\bZ$.
	See Figure \ref{fig-circle-star}.
\end{defi}

\begin{figure}[htbp]
	\centering
	% First subfigure

	\begin{subfigure}[b]{0.25\linewidth}
		\centering
		\begin{tikzpicture}
			% Define the central point
			\def\x{1}
			\coordinate (P) at (-0.2*\x, 0);
			\coordinate (Q) at (0.2*\x, 0);
			\draw[thick] (P) arc[start angle=45, end angle=-45, radius=\x];
			\draw[thick] (Q) arc[start angle=180-45, end angle=180+45, radius=\x];
			% Draw the central point
		\end{tikzpicture}
		\caption{$2$-circle}
	\end{subfigure}
	\hspace{0.05\linewidth}
	\begin{subfigure}[b]{0.25\linewidth}
		\centering
		\begin{tikzpicture}
			% Define the vertices of the equilateral triangle
			\def\x{1};
			\coordinate (A) at (0, 0);
			\coordinate (B) at (\x, 0);
			\coordinate (C) at ($(A)!0.5!(B) + (0, {\x*sqrt(3)/2})$);
			
			% Draw the triangle
			\draw[thick,bend right=20] ($(B)!1.5!(A)$) to ($(A)!1.5!(B)$);
			\draw[thick,bend right=20] ($(C)!1.5!(B)$) to ($(B)!1.5!(C)$);
			\draw[thick,bend right=20] ($(A)!1.5!(C)$) to ($(C)!1.5!(A)$);
		\end{tikzpicture}
		\caption{$3$-circle}
	\end{subfigure}
	\hspace{0.05\linewidth}
	\begin{subfigure}[b]{0.25\linewidth}
		\centering
		\begin{tikzpicture}
			% Define the central point
			\def\x{1}
			\coordinate (O) at (0, 0);
			\coordinate (A) at (0, \x);
			\coordinate (B) at ({\x*sin(120)}, {\x*cos(120)});
			\coordinate (C) at ({\x*sin(240)}, {\x*cos(240)}); 
			\draw[thick] ($(A)!1.25!(O)$) to (A);
			\draw[thick] ($(B)!1.25!(O)$) to (B);
			\draw[thick] ($(C)!1.25!(O)$) to (C);
			
			% Draw the central point
		\end{tikzpicture}
		\caption{$3$-star}
	\end{subfigure}
	
	\caption{}
	\label{fig-circle-star}
\end{figure}

An important feature of a $d$-circle $Z$ is that there exists an \'etale cover from $Z$ to a nodal rational curve (see \cite[Excercise III.10.6]{Hartshorne1977AlgebraicG}).

Let $\pi:T\ra S$ be a quasi-\'etale cover of Du Val del Pezzo surfaces with $\deg(\pi)\geq2$.
Then $\pi$ is Galois and can be decomposed into Galois quasi-\'etale subcovers of degree $2,3$ or $5$.
We summarize the main results of this section in the following theorem.
\begin{theo}\label{theo-corr}
	Let $\pi:T\ra S$ be a quasi-\'etale cover of Du Val del Pezzo surfaces over an algebraically closed field of characteristic $0$.
	Suppose $\deg(\pi)=p=2,3$ or $5$.
	Using the notations from diagram (\ref{diag-model}), we have the following assertions:
	\begin{enumerate}
		\item $\widetilde{\pi}:Y\dasharrow X$ maps a $(-2)$-curve to a $(-2)$-curve. \label{theo-corr-1}
		\item $\pi:T\ra S$ maps a $(-1)$-curve $C$ to a $(-1)$-curve, except when $C$ is contained in a $p$-circle of $(-1)$-curves, in which case $C$ is mapped to a nodal cubic.
%		The later condition is satisfied when $X$ is of degree at least $2$.
		\label{theo-corr-2}
		\item The preimage of a $(-2)$-curve under $\widetilde{\pi}:Y\dasharrow X$ is a union of $0,1$ or $p$ $(-2)$-curves.\label{theo-corr-3}

		\item \label{theo-corr-4}The preimage of a $(-1)$-curve under $\pi: T \to S$ is one of the following:
		\begin{enumerate}
			\item a disjoint union of $p$ $(-1)$-curves,
			\item a $p$-star of $(-1)$-curves, or
			\item a single curve with non-negative self-intersection, having at most one singularity, which is either a node or a cusp.
		\end{enumerate}
	\end{enumerate}
\end{theo}

As a corollary, we have:
\begin{coro}\label{coro-corr}
	When $X$ is of degree at least $2$, the correspondence $\Gamma(\pi)$ is a function $\Gamma(Y)\to\Gamma(X)$ and only depends on $\Type(\widetilde{\pi})$.
	In particular, the induced map $\widetilde{\pi}_\ast:\Eff^1(Y)\ra \Eff^1(X)$ sends extremal rays to extremal rays.
\end{coro}

\begin{rema}\label{rema:counterexample}
	On the other hand, none of the conclusions of Corollary \ref{coro-corr} hold in general when $X$ is a weak del Pezzo surface of degree $1$.
	Section \ref{sec:deg1,E6+A2} provides such an example.
	A more straightforward example is the quasi-\'etale cover of type $S_3(A_2)\ra S_1(A_8)$. In this case, there are $15$ $(-1)$-curves on $S_3(A_2)$ but only $2$ $(-1)$-curves on $S_1(A_8)$.
\end{rema}

The following lemma is well-known.
\begin{lemm}\label{lemm-action-on-p1}
	Any finite order automorphism $\sigma$ of $\bP^1$ has two distinct fixed points.
	If $\sigma$ has three distinct fixed points, then it is the identity morphism.\qed
\end{lemm}

We first consider the image of a Galois-invariant curve.
\begin{prop}\label{prop-image-of-curve}
	Let $\pi:T\to S$ be a quasi-étale cover of normal projective surfaces with at worst Du Val singularities.
	Suppose $\pi$ is Galois with $\Gal(T/S)$ is a cyclic group of prime order $p$.
	We use the notations from diagram (\ref{diag-model}).
	Let $C$ be an $(n)$-curve on $T$.
	If $C$ is fixed by $\sigma$, then $\pi(C)$ is a $(k)$-curve with $n=pk+r$ for some $r\geq2$.
\end{prop}
\begin{proof}
	Let $D\coloneqq \pi(C)$ denote the image of $C$.
	The Riemann–Hurwitz formula implies that
	$$g_a(D)\leq g_a(C)=0.$$
	Hence $g_a(D)=0$ and $D$ is a smooth rational curve since it is irreducible.
	
	Since $C$ is fixed by $\sigma$,
	the Galois cover $\pi$ restricts to a Galois cover $\pi|_C:C\ra D$ whose Galois group $\Gal(C/D)$ is generated by $\sigma|_C$ and is isomorphic to $\bZ/p\bZ$.
	By Lemma \ref{lemm-action-on-p1}, the automorphism $\sigma|_C$ of $C$ has exactly two fixed points.
	
	Let $\rho_{T^\prime}:T_0\to T^\prime$ be the minimal resolution of $T^\prime$.
	Let $R$ be the pull back of the ramification locus of $f$ along $\rho_{T^\prime}$.
	By Proposition \ref{prop:model}, $R$ equals the exceptional divisor of $\circ g\circ \rho_{T^\prime}$.
	
	By Lemma \ref{lemm:ramification-of-model},
	every connected component of $R$ is a chain of negative curves whose configuration only depends on the degree $p$ of $\pi$.
	Write $R=\bigsqcup_{i=1}^{N}R_i$ and $r_p$ for the number of irreducible components of $R_i$.
	There are exactly two connected components $R_1$ and $R_2$ of $R$ lies above the two fixed points of $\sigma|_C$ on $C$.
	Let $C_0\subset T_0$ be the strict transform of $C$ along $\rho_T\circ g\circ \rho_{T^\prime}:T_0\ra T$.
	We must have $R_i\cdot C_0=1$ for $i=1,2$ as $R_i\cdot C_0\geq2$ would imply that $C$ is singular, which contradicts the assumption.
	
	The birational morphism $T_0\ra T$ passes a singular surface $T^\prime$ when $p\geq3$.
	However, since $T_0\ra T$ is a birational morphism between smooth surfaces,
	it can be decomposed into a sequence
	$$T_0\ra T_{1,1}\ra\dots\ra T_{1,r_p}=T_1\ra T_{2,1}\ra\dots T_{2,r_p}=T_2\ra\dots\ra T_{N,r_p}=T_N=Y,$$
	where each step is a blowup at a smooth point and $T_{i-1}\ra T_{i}$ precisely contracts the image of $R_i$.
	Let $C_{i,j}$ be the strict transform of $C$ in $T_{i,j}$.
	For $i=1,2$, we have $E_{i,j}\cdot C_{i,j}=0$ or $1$, and let $r_i$ denote the cardinality of the set
	$$\{ 0\leq j\leq r_p-1\mid C_{i,j}\cdot E_{i,j}=1 \}.$$
	The integer $r_i$ depends only on the irreducible component of $R_i$ on which the intersection point of $C_0$ and $R_i$ lies.
	We have $r_i\geq1$ for $i=1,2$.
	The curve $C_0$ has self-intersection number $n-r_1-r_2$.
	The projection formula implies that $D^2=(n-r_1-r_2)/p$.
	Thus $D$ is an $((n-r_1-r_2)/p)$-curve since it is smooth.
\end{proof}

We then consider the image of a Galois-non-invariant curve.
\begin{prop}\label{prop-image-of-curve-2}
	Under the setting of Proposition \ref{prop-image-of-curve}, let $C$ be a $(-1)$-curve on $T$ which is not fixed by $\sigma$.
	If $C\cdot \sigma(C)=0$, then the image $\pi(C)$ is a $(-1)$-curve.
	Otherwise, $C\cdot \sigma(C)>0$ and we have the following assertions.
	When $p=2$, the image $\pi(C)$ is a nodal cubic curve on $S$, and $\pi$ restricts to an \'etale cover $C\cup \sigma(C)\ra \pi(C)$.
	When $p=3$, there are two cases:
	\begin{enumerate}
		\item If $C\cap \sigma(C)\cap\sigma^2(C)=\emptyset$, then $\pi(C)$ is a nodal cubic curve on $S$.
		\item If $C\cap \sigma(C)\cap\sigma^2(C)\neq\emptyset$, then it is a one-point set and $\pi(C)$ is a $(-1)$-curve on $S$ through a singular points of $S$.
	\end{enumerate}
\end{prop}
\begin{proof}
	Since $T$ and $S$ has at worst Du Val singularities, the intersection number of divisors on them can be computed by strict transform to $Y$ and $X$.
	Thus we may assume that $T$ is smooth in $C$.
	Write $D\coloneqq \pi(C)$.
	The Galois cover $\pi:T\ra S$ pulled back to a Galois cover $Z:=\bigcup_i \sigma^i(C)\ra D$ with Galois group $\bZ/p\bZ$.
	
	{\bf Case A.} Suppose $C\cdot \sigma(C)=0$.
	Then $D$ lies on the \'etale locus of $\pi$.
%	and so $\pi$ restricts to an a generically finite cover $C\ra D$.
	By permanence properties of base changes, $Z\ra D$ is a finite \'etale cover.
	Then $C\cdot \sigma(C)=0$ implies that $Z$ is a disjoint union of $(-1)$-curves,
	and thus $D$ is a smooth rational curve.
	Hence $D$ is a $(-1)$-curve by projection formula.
	
	{\bf Case B.} Suppose $C\cdot \sigma(C)\neq0$.
	Then $C\cdot \sigma(C)>0$ since $C$ is not fixed by $\sigma$.
	Let $p\in C\cap \sigma(C)$ be a point.
	It is a smooth point of $T$ by assumption.

	{\bf Case B1.} Suppose $p=2$.
%	Firstly, I assert that $C$ and $\sigma(C)$ intersects transversely at $p$.
%	Indeed, $T_0=T^\prime\ra T$ is a resolution of ramification divisor of $\pi$ (see Remark \ref{rema-resolution-of-ramification}).
%	Let $r$ be the intersection multiplicity of $C$ and $\sigma(C)$ at $p$.
%	Then the strict transform $C_0$ of $C$ in $T_0$ has self-intersection at least $-1-r$.
%	The finite cover $\pi$ is unramified at $C_0$.
%	Thus $f(C_0)$ is a rational curve with self intersection at least $-(1+r)/2$.
%	Since there is no rational curve on $S$ of self-intersection less than $-1$,
%	we conclude that $r=1$.
%	Thus $C$ and $\sigma(C)$ intersects transversely at $p$.
	I assert that $p$ is not fixed by $\sigma$.
	Suppose the opposite.
%	Then resolving the ramification point $p$ will produce an unramified smooth rational curve in $T_0$ with self-intersection at most $-3$.
%	Indeed, 
	Let $E$ be the exceptional $(-1)$-curve in the final step of the embedded resolution of $C\cup \sigma(C)$.
	Then $E$ is fixed but not pointwise fixed by $\sigma$, and there are two fixed points on $E$ not above $C$ and $\sigma(C)$.
	We must blowup these two points to resolve ramification, and this will produce an unramified $(-3)$-curve from $E$.
	Thus there is an unramified $(r)$-curve in $T_0$ with $r\leq -3$,
	which is mapped to an $(r/2)$-curve in $S$.
	But there is no negative curve on a Du Val del Pezzo surface with self-intersection less than $-1$.
	
	Thus $C\cdot \sigma(C)$ is positive and even.
	Since $T$ is a Du Val del Pezzo surface of degree at least $2$, the intersection of two $(-1)$-curves is at most $2$.
	Therefore, $C$ and $\sigma(C)$ intersect transversely at $2$ points, and so $C$ lies in the \'etale locus of $\pi$.
	Their image $D$ in $S$ is a rational curve with self intersection $1$, \ie\, a nodal cubic curve.
	
	{\bf Case B2.} Suppose $p=3$ and there is a fixed point $p$ of $\sigma$ lying on $C$.
	By Lemma \ref{lemm-dP2}, we have $C\cdot\sigma(C)=1$.
	Thus one blow-up at $p$ resolve the ramification,
	with a fixed but not pointwise fixed exceptional $(-1)$-curve $E$.
	Thus further blow up at the two fixed points of $E$ is needed, and thus the strict transform $E_0$ of $E$ in $T_0$ is a $(-3)$-curve (see Lemma \ref{lemm:ramification-of-model}).
	Let $C_0$ be the strict transform of $C$ in $T_0$ and abuse the notation of $D$ with its strict transform in $Y$.
	By the projection formula applied to $\varphi:f\circ\rho_{T^\prime}:T_0\ra X,$
	$$3D^2=(\varphi^\ast D)^2=(C_0+\sigma(C_0)+\sigma^2(C_0)+E_0)^2=-3,$$
	it follows that $D$ is a $(-1)$-curve on $S$ passing through a singular point of $S$.

	{\bf Case B3.} Suppose $p=3$ and there is no fixed point $p$ of $\sigma$ lying on $C$.
	Then $C\cdot\sigma(C)=1$ implies that $Z^2=3$.
	By the projection formula applied to $\pi,$
	$$3D^2=(\varphi^\ast D)^2=(C+\sigma(C)+\sigma^2(C))^2=3,$$
	it follows that $D$ is a nodal cubic curve.
\end{proof}

We finally prove Theorem \ref{theo-corr}.
Although a similar argument applies to the case of $p=5$,
we prefer to present the only type of cover in this case in Example \ref{exam-dP5}.
\begin{proof}[Proof of Theorem \ref{theo-corr}]
	
	Let $C$ be a $(-2)$-curve on $Y$.
	By Proposition \ref{prop:model}, every $(-2)$-curve on $Y$ is unramified.
	Thus either $C$ is fixed by $\sigma$ with two fixed points or $C\cap\sigma(C)=0$.
	In both cases, the projection formula implies that $\pi(C)$ is a $(-2)$-curve.
	Thus (\ref{theo-corr-1}) of the theorem follows.
	
	Let $D$ be a $(-2)$-curve on $X$.
	Let $B_{\red}\subset X$ denote the branch locus of $f$.
	If $D\subset B_{\red}$, then $f^{-1}(D)$ is contracted by $g$ by Proposition \ref{prop:model}.
	Suppose the opposite.
	If $D\cdot B_{\red}=0$, then $f^\ast D$ is a disjoint union of $p$ $(-2)$-curves.
	Suppose $D\cdot B_{\red}>0$.
	Since $(-2)$-curves intersects transversally in the ADE configurations,
	the pullback $f^\ast D$ is a $(-4)$-curve.
	By Lemma \ref{lemm-action-on-p1}, the fixed locus of $\sigma$ on $f^\ast D$ consists of two points.
	Thus $D\cdot B_{\red}=2$, and contracting $B_{\red}$ maps $D$ to a $(-2)$-curve in $Y$. 
	This proves (\ref{theo-corr-3}) of the theorem.
	
	The only possible $\pi$ of degree $5$ is represented in Example \ref{exam-dP5}, from which (\ref{theo-corr-2}) and (\ref{theo-corr-4}) are clear.
	So from now on, we assume that $p=2$ or $3$.
	
	Let $C$ be a $(-1)$-curve on $Y$.
	Suppose $C$ is fixed by $\sigma$, by Proposition \ref{prop-image-of-curve}, we have $-1=pk+r\geq2k+2$, where $\pi(C)$ is a $(k)$-curve.
	Since there every $(k)$-curve on $X$ satisfies $k\geq-1$,
	we get $-1\geq0$, absurd.
	Thus there is no $(-1)$-curve on $Y$ fixed by $\sigma$.
	We then reduce to the setting of Proposition \ref{prop-image-of-curve-2}, from which (\ref{theo-corr-2}) of the theorem follows.
	
	Let $D$ be a $(-1)$-curve on $X$.
	If $D\cdot B_{\red}=0$, then $f^\ast D$ is a disjoint union of $p$ $(-1)$-curves.
	Suppose $D\cdot B_{\red}>0$.
	Then $D\cdot B_{\red}\geq2$ since $B$ is divisible by $p$ in $\Pic(X)$.
	Thus, if $D$ intersects $B_{\red}$ transversely, then $\pi^{-1}(D)$ is irreducible, has non-negative self-intersection, and has at most one singular point, which is either a node or a cusp, by Lemma \ref{lemm-dP2};

	Otherwise, either
	\begin{itemize}
		\item $p=2$ and $D$ is tangent to $B_{\red}$ at a point.
		\item $p=3$ and $D$ intersects $B_{\red}$ at a node of $B_{\red}$.
	\end{itemize}
	In the former case, $\pi^{-1}(D)$ is the union of two smooth conics intersecting at a point with multiplicity $2$;
	In the latter case, $\pi^{-1}(D)$ is a $3$-star of $(-1)$-curves.
	This proves (\ref{theo-corr-4}) of the theorem.
\end{proof}

\begin{exam}\label{exam-dP5}
	Let $T\ra S$ be a quasi-\'etale cover of type $S_5\ra S_1(2A_4)$: a $\bZ/5\bZ$-quotient of the del Pezzo surface of degree $5$.
	The branch locus $B_{\red}$ of $f$ is the union of all $(-2)$-curve on $X$.
	There are six lines on $S$, each intersecting $B_{\mathrm{red}}$ with multiplicity 2.
	A similar argument to the proof of Theorem \ref{theo-corr} shows that the preimages of these six lines are smooth conics,
	and the ten lines on $T$ form two $5$-circles, mapping to two nodal rational curves on $S$ intersecting at one point.
\end{exam}

To understand the Galois actions on  the quasi-étale covers, we study equivariant geometry of them. We introduce a relative version of Cremona isometry defined in \cite[Section 8.2.8]{Dolgachev2012} for quasi-\'etale covers.

\begin{defi}\label{defi:correspondence}
Let $\pi: T \to S$ be a quasi-étale cover of Du Val del Pezzo surfaces, and let $Y$ and $X$ denote the minimal resolutions of $T$ and $S$, respectively. The \textit{group of Cremona isometries} $\Cris(\pi)$ is defined as the subgroup of $\Cris(T) \times \Cris(S)$ consisting of elements $(\tau, \sigma)$ such that $(\tau(C), \sigma(D)) \in \Gamma(\pi)$ for all $(C, D) \in \Gamma(\pi)$.
\end{defi}

This definition depends only on the type of $\pi$ wherever $\Gamma(\pi)$ does,
which holds when $S$ is of degree at least $2$ (Corollary \ref{coro-corr}).
Since all morphisms in (\ref{diag-model}) is $\Gal(\overline{F}/F)$-equivariant, the following corollary is immediate.

\begin{coro}\label{coro:equivariant}
	Let $F$ be the base field of the morphism $\pi:T\ra S$.
	Then the image of $\rho:\Gal(\overline{F}/F)\ra \Cris(S)\times\Cris(T)$ is contained in $\Cris(\pi)$.\qed
\end{coro}
We call the image of $\rho$ in Corollary \ref{coro:equivariant} the \textit{splitting group of $\pi$}.

For a finite group $G$ acting on $\Pic(X)$, write $\rho(X,G)\coloneqq \rk\Pic(X)^G$ for the $G$-invariant Picard rank.
When the realization problem has an affirmative answer, we recover the Picard rank over $F$.
\begin{prop}{\rm (\cite[Proposition 6.2]{derenthal2008nef})}\label{prop:picard rank}
	Let $H_S\times H_T$ denote the image of $\rho$ defined in Corollary \ref{coro:equivariant} and let $\overline{F}$ be an algebraic closure of $F$.
	Suppose that there exists a rational point on the smooth locus of $T$, then $\rho(X)=\rho(X_{\overline{F}},H_S)$ and $\rho(Y)=\rho(Y_{\overline{F}},H_T)$.
\end{prop}

It is sufficient to compute $\rho(X_{\overline{F}},H_S)$ up to conjugacy classes by the following result.
\begin{prop}
	Let $H$ and $H^\prime$ be conjugate subgroups of $\Cris(S)$.
	Then $\rho(X,H)=\rho(X,H^\prime)$.
\end{prop}
\begin{proof}
	By definition, the group $H$ acts on the set of negative curves on $S$.
	Let $P$ be the set of orbits under this action.
	Then $\rho(X,H)$ equals the dimension of the subspace generated by the orbit sums $\sum_{D\in O}D$, where $O\in P$.
	Let $\sigma\in\Cris(S)$ such that $\sigma H\sigma^{-1}=H^\prime$.
	The orbits under the action of $H'$ are $\sigma O$ for $O \in P$.
	Thus, $\rho(X,H')$ equals the dimension of the subspace generated by $\sigma(\sum_{D \in O} D)$, where $O \in P$. 
	Since an invertible matrix preserves the rank of a set of vectors, the assertion follows.
\end{proof}

\subsection{A combinatorial description of Cremona isometries}
For a weighted graph $\Gamma$, we denote by $\Aut(\Gamma)$ the group of permutations of vertices and edges that preserve their weights.
The following observation converts problems in linear algebra to those in combinatorics.
\begin{prop}\label{prop:Cris=Aut}
	Let $X$ be a weak del Pezzo surface.
	Then the groups $\Cris(X)$ and $\Aut(\Gamma(X))$ are isomorphic.
\end{prop}
\begin{proof}
	When $X$ is of degree $d\leq 6$,
	by \cite[Theorem 23.9]{manin1986cubic}, the following two groups are isomorphic:
	\begin{itemize}
		\item the group $A$ of isometries of $N^1 (X)$ that preserve $K_X$, and
		\item the group $B$ of permutations of $(-1)$-classes preserving the pairwise intersection numbers between them.
	\end{itemize}
	The group $\Cris(X)$ consists of the elements in $A$ that preserve effective $(-1)$- and $(-2)$classes,
	and $\Aut(\Gamma(X))$ is the group of permutations of the set of $(-1)$- and $(-2)$-classes that preserve their pairwise intersections.
	
	We first construct a group homomorphism $\Psi:\Aut(\Gamma(X))\ra B$.
	Let $\sigma\in \Aut(\Gamma(X))$ and let $l$ be a $(-1)$-class of $X$.
	Then $l$ is numerically equivalent to a linear combination of $(-1)$-curves and $(-2)$-curves with rational coefficients \cite[Proposition 3.6 and 3.7]{derenthal2008nef}.
	Write $l\equiv \sum c_iC_i+\sum d_jD_j$ where $C_i$ are $(-1)$-curves and $D_i$ are $(-2)$-curves and $c_i,d_j$ are rational numbers.
	Then we define $(\Psi(\sigma))(l)$ to be $\sum c_i\sigma(C_i)+\sum d_j\sigma(D_j)$.
	By linearity, $(\Psi(\sigma))(l)$ is still a $(-1)$-class.
	Since $\Psi(\sigma)$ is invertible by definition, it is a well-defined element of $B$.
	This defines a group homomorphism $\Psi:\Aut(\Gamma(X))\ra B$ by linearity again.
	
	Then, we prove that $\ker(\Psi)=0$.
	Let $\Psi(\sigma)\in B$ fixes each $(-1)$-classes.
	Then it trivially fixes each $(-1)$-curves.
	Suppose there exists a $(-2)$-curve $D$ with $(\Psi(\sigma))(D)\neq D$.
	Since the dual graph $\Gamma(X)$ of negative curves is connected and there exists at least one $(-1)$-curve on $X$, there is a chain of negative curves
	\[\begin{tikzpicture}
		\def \x {1.5};
		\node (01) at (0,0) {$(-2)$};
		\node (02) at (\x,0) {$(-2)$};
		\node (03) at (2*\x,0) {$\cdots$};
		\node (04) at (3*\x,0) {$(-2)$};
		\node (05) at (4*\x,0) {$(-1)$};
		\draw (01)--(02)--(03)--(04)--(05);
	\end{tikzpicture}\]
	such that the leftmost $(-2)$-curve is $D$.
	The linear sum of the chain is a $(-1)$-class, and thus is fixed by $\Psi(\sigma)$.
	This implies that the linear combination of the $(-2)$-curves in the chain is also fixed by $\Psi(\sigma)$.
	Moreover, $(\Psi(\sigma))(D)$ can not be contained in this chain; otherwise, the $(-1)$-curve in the chain would not be fixed by $\Psi(\sigma)$.
	Thus the $(-2)$-curves in the chain are mapped to entirely different $(-2)$-curves by $\Psi(\sigma)$.
	This is absurd since the $(-2)$-curves on $X$ are exceptional divisors of a smooth resolution, and thus are linearly independent by Hodge index theorem.
	
	Now we have shown the following relations of groups:
	$$\Cris(X)\subset A\cong B\supset \Aut(\Gamma(X)).$$
	By definition, the image of $\Cris(X)$ in $B$ is contained in $\Aut(\Gamma(X))$, and the image of $\Aut(\Gamma(X))$ in $A$ is contained in $\Cris(X)$.
	Thus we conclude that $\Cris(X)\cong \Aut(\Gamma(X))$.

	%	Then the assertion follows from the fact that any $(-1)$-class is numerically equivalent to one of the following \cite[III. Th\'eor\`eme 2.c]{demazure1980surfaces}:
	%	\begin{itemize}
		%		\item a $(-1)$-curve, or
		%		\item the sum of a $(-1)$-curve and some $(-2)$-curves, or
		%		\item the sum of $-K_X$ and some $(-2)$-curves.
		%	\end{itemize}
	
	When $X$ is of degree at least $7$, one can check the proposition case-by-case.
%	\begin{itemize}
%		\item $\Cris(X)\cong C_2$ when $X$ is $\bP^2$ or $\bP^1\times\bP^1$;
%		\item $\Cris(X)\cong C_1$ when $X$ is the Hirzebruch surface $\bF_1$ or $\bF_2$.
%	\end{itemize}
\end{proof}

\section{Classification of quasi-\'etale covers}
\label{sec3}
\subsection{Over closed field}\label{sec-quasi-et}

%For any quasi-\'etale cover $\pi:T\ra S$ of Du Val del Pezzo surfaces, the restriction of $\pi$ to the smooth part of $S$ is Galois.
%By Zariski’s Main Theorem in the equivariant setting \cite{GKP16}, we have $\pi$ is Galois as well.

We apply Proposition \ref{prop:fundamental group} and other results from previous sections to determine the type of quasi-\'etale covers for all Du Val del Pezzo surfaces.
The computations are performed using $\mathtt{Magma}$ \cite{magma}.

\begin{theo}\label{theo:classification}
	Let $S$ be a Du Val del Pezzo surface over $\bC$.
	Then, up to isomorphism, quasi-\'etale covers of $S$ are in one-to-one correspondence with the paths in the following diagrams that terminate at $\Type(S)$,
	and the first diagram in each entry represents the corresponding Galois group.
\end{theo}

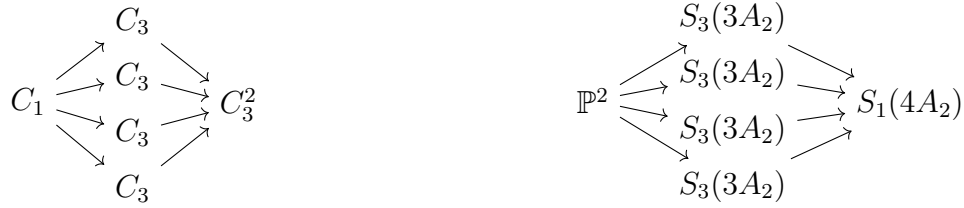
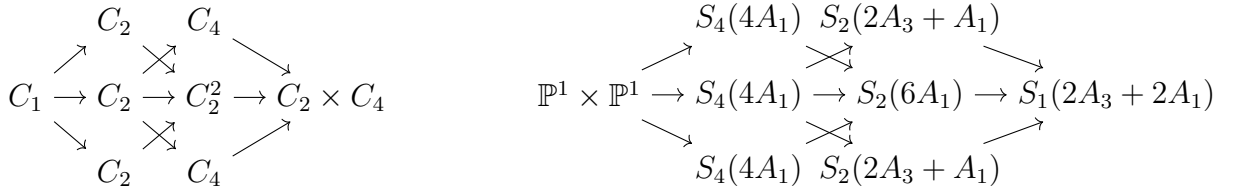
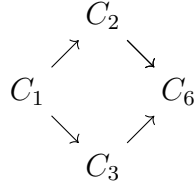
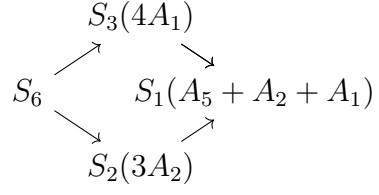
\begin{enumerate}
	\item $C_3^2$
	\begin{figure}[H]
		\centering
		\begin{subfigure}{0.49\textwidth}
			\centering
			\begin{tikzpicture}[node distance={21pt},->]
				\node (01) at (0,0) {$C_1$};
				\node (11) [above right=-9pt and 18pt of 01] {$C_3$};
				\node (10)[above of=11]   {$C_3$};
				\node (12)[below of=11]   {$C_3$};
				\node (13)[below of=12]   {$C_3$};
				\node (21)[below right=-9pt and 18pt of 11]   {$C_3^2$};
				
				\draw (01)--(10); \draw(01)--(11); \draw(01)--(12); \draw(01)--(13);
				\draw(10)--(21);
				\draw(11)--(21); \draw (12)--(21); \draw (13)--(21);
			\end{tikzpicture}
		\end{subfigure}
		\hfill
		\begin{subfigure}{0.49\textwidth}
			\centering
			\begin{tikzpicture}[node distance={21pt},->]
				\node (01) at (7,1) {$\bP^2$};
				\node (11) [above right=-9pt and 18pt of 01] {$S_3(3A_2)$};
				\node (10)[above of=11]   {$S_3(3A_2)$};
				\node (12)[below of=11]   {$S_3(3A_2)$};
				\node (13)[below of=12]   {$S_3(3A_2)$};
				\node (21)[below right=-9pt and 18pt of 11]   {$S_1(4A_2)$};
				
				\draw (01)--(10); \draw(01)--(11); \draw(01)--(12); \draw(01)--(13);
				\draw(10)--(21);
				\draw(11)--(21); \draw (12)--(21); \draw (13)--(21);
			\end{tikzpicture}
		\end{subfigure}
	\end{figure}

	\item $C_2\times C_4$
	\begin{figure}[H]
		\centering
		\begin{subfigure}{0.49\textwidth}
			\centering
			\begin{tikzpicture}[->]
				\node (01) at (0,1) {$C_1$};
				\node (11) [right=12pt of 01] {$C_2$};
				\node (10)[above of=11]   {$C_2$};
				\node (12)[below of=11]   {$C_2$};
				\node (21)[right=12pt of 11]   {$C_2^2$};
				\node (20) [above of=21] {$C_4$};
				\node (22)[below of=21]   {$C_4$};
				\node (31)[right=12pt of 21]   {$C_2\times C_4$};
				
				\draw (01)--(10); \draw(01)--(11); \draw(01)--(12);
				\draw(10)--(21);
				\draw(11)--(20); \draw (11)--(21); \draw (11)--(22);
				\draw (12)--(21);
				\draw(20)--(31); \draw (21)--(31); \draw (22)--(31);
			\end{tikzpicture}
		\end{subfigure}
		\hfill
		\begin{subfigure}{0.49\textwidth}
			\centering
			\begin{tikzpicture}[->]
				\node (01) at (7,1) {$\bP^1\times \bP^1$};
				\node (11) [right=12pt of 01] {$S_4(4A_1)$};
				\node (10)[above of=11]   {$S_4(4A_1)$};
				\node (12)[below of=11]   {$S_4(4A_1)$};
				\node (21)[right=12pt of 11]   {$S_2(6A_1)$};
				\node (20) [above of=21] {$S_2(2A_3+A_1)$};
				\node (22)[below of=21]   {$S_2(2A_3+A_1)$};
				\node (31)[right=12pt of 21]   {$S_1(2A_3+2A_1)$};
				
				\draw (01)--(10); \draw(01)--(11); \draw(01)--(12);
				\draw(10)--(21);
				\draw(11)--(20); \draw (11)--(21); \draw (11)--(22);
				\draw (12)--(21);
				\draw(20)--(31); \draw (21)--(31); \draw (22)--(31);
			\end{tikzpicture}
		\end{subfigure}
	\end{figure}

	\item $C_6$
	\begin{figure}[H]
		\centering
		\begin{subfigure}{0.49\textwidth}
			\centering
			\begin{tikzpicture}[->]
				
				\def\x{-1}
				\def\y{1}
				\node (01) at (\x,\y) {$C_1$};
				\node (10)  at (\x+1,\y+1)  {$C_2$};
				\node (12)  at (\x+1,\y-1) {$C_3$};
				\node (21)  at (\x+2,\y) {$C_6$};
				
				\draw (01)--(10); \draw(01)--(12);
				\draw(10)--(21);
				\draw(10)--(21); \draw (12)--(21);
			\end{tikzpicture}
		\end{subfigure}
		\hfill
		\begin{subfigure}{0.49\textwidth}
			\centering
			\begin{tikzpicture}[->]
						
				\def\x{0}
				\def\y{0}
				\node (01) at (\x,\y) {$S_6$};
				\node (10)  at (\x+1.5,\y+1)  {$S_3(4A_1)$};
				\node (12)  at (\x+1.5,\y-1) {$S_2(3A_2)$};
				\node (21)  at (\x+3,\y) {$S_1(A_5+A_2+A_1)$};
				
				\draw (01)--(10); \draw(01)--(12);
				\draw(10)--(21);
				\draw(10)--(21); \draw (12)--(21);
			\end{tikzpicture}
		\end{subfigure}
	\end{figure}

	\item $C_5$
	\[
	\begin{tabular}{ccc}
		$C_1\ra C_5$\hspace{5.7cm}$ S_5\ra S_1(2A_4)$
	\end{tabular}\]

	\item $C_2^2$
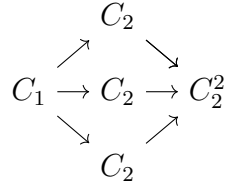
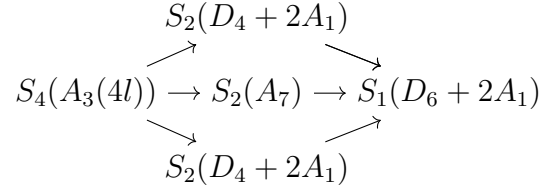
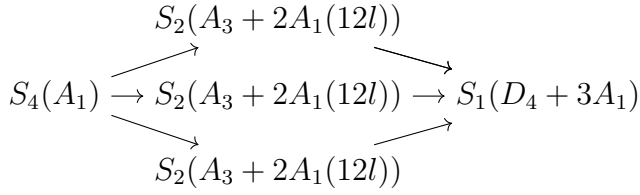
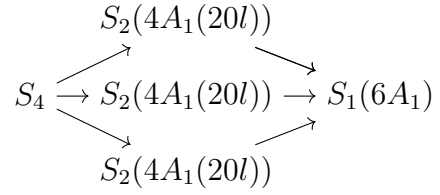
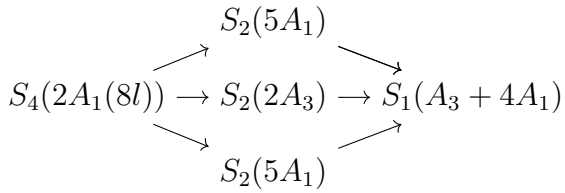
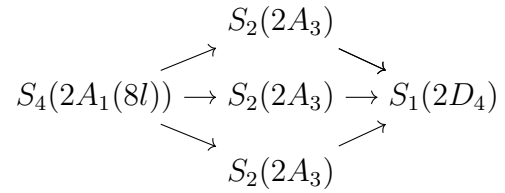
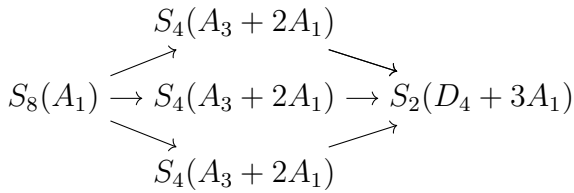
		\begin{figure}[H]
		\centering
		\begin{subfigure}{0.49\textwidth}
			\centering
			\begin{tikzpicture}[->]
				\node (01) at (-1,1) {$C_1$};
				\node (11) [right=12pt of 01] {$C_2$};
				\node (10)[above of=11]   {$C_2$};
				\node (12)[below of=11]   {$C_2$};
				\node (21)[right=12pt of 11]   {$C_2^2$};
				
				\draw (01)--(10); \draw(01)--(11); \draw(01)--(12);
				\draw(10)--(21);
				\draw(10)--(21); \draw (11)--(21); \draw (12)--(21);
			\end{tikzpicture}
		\end{subfigure}
		\hfill
		\begin{subfigure}{0.49\textwidth}
			\centering
			\begin{tikzpicture}[->]
				\node (01) at (6,1) {$S_4(A_3(4l))$};
				\node (11) [right=12pt of 01] {$S_2(A_7)$};
				\node (10)[above of=11]   {$S_2(D_4+2A_1)$};
				\node (12)[below of=11]   {$S_2(D_4+2A_1)$};
				\node (21)[right=12pt of 11]   {$S_1(D_6+2A_1)$};
				
				\draw (01)--(10); \draw(01)--(11); \draw(01)--(12);
				\draw(10)--(21);
				\draw(10)--(21); \draw (11)--(21); \draw (12)--(21);
			\end{tikzpicture}
		\end{subfigure}
	\end{figure}

\begin{figure}[H]
	\centering
	\begin{subfigure}{0.49\textwidth}
		\centering
		\begin{tikzpicture}[->]
			\node (01) at (-3,-2) {$S_4(A_1)$};
		\node (11) [right=12pt of 01] {$S_2(A_3+2A_1(12l))$};
		\node (10)[above of=11]   {$S_2(A_3+2A_1(12l))$};
		\node (12)[below of=11]   {$S_2(A_3+2A_1(12l))$};
		\node (21)[right=12pt of 11]   {$S_1(D_4+3A_1)$};
		
		\draw (01)--(10); \draw(01)--(11); \draw(01)--(12);
		\draw(10)--(21);
		\draw(10)--(21); \draw (11)--(21); \draw (12)--(21);
		\end{tikzpicture}
	\end{subfigure}
	\hfill
	\begin{subfigure}{0.49\textwidth}
		\centering
		\begin{tikzpicture}[->]
			\node (01) at (6,-2) {$S_4$};
			\node (11) [right=12pt of 01] {$S_2(4A_1(20l))$};
			\node (10)[above of=11]   {$S_2(4A_1(20l))$};
			\node (12)[below of=11]   {$S_2(4A_1(20l))$};
			\node (21)[right=12pt of 11]   {$S_1(6A_1)$};
			
			\draw (01)--(10); \draw(01)--(11); \draw(01)--(12);
			\draw(10)--(21);
			\draw(10)--(21); \draw (11)--(21); \draw (12)--(21);
		\end{tikzpicture}
	\end{subfigure}
\end{figure}

				\begin{figure}[H]
			\centering
			\begin{subfigure}{0.49\textwidth}
				\centering
				\begin{tikzpicture}[->]
							\node (01) at (-2,-5) {$S_4(2A_1(8l))$};
					\node (11) [right=12pt of 01] {$S_2(2A_3)$};
					\node (10)[above of=11]   {$S_2(5A_1)$};
					\node (12)[below of=11]   {$S_2(5A_1)$};
					\node (21)[right=12pt of 11]   {$S_1(A_3+4A_1)$};
					
					\draw (01)--(10); \draw(01)--(11); \draw(01)--(12);
					\draw(10)--(21);
					\draw(10)--(21); \draw (11)--(21); \draw (12)--(21);
				\end{tikzpicture}
			\end{subfigure}
			\hfill
			\begin{subfigure}{0.49\textwidth}
				\centering
				\begin{tikzpicture}[->]
							\node (01) at (0,0) {$S_4(2A_1(8l))$};
					\node (11) [right=12pt of 01] {$S_2(2A_3)$};
					\node (10)[above of=11]   {$S_2(2A_3)$};
					\node (12)[below of=11]   {$S_2(2A_3)$};
					\node (21)[right=12pt of 11]   {$S_1(2D_4)$};
					
					\draw (01)--(10); \draw(01)--(11); \draw(01)--(12);
					\draw(10)--(21);
					\draw(10)--(21); \draw (11)--(21); \draw (12)--(21);
				\end{tikzpicture}
			\end{subfigure}
		\end{figure}

				\begin{figure}[H]
			\centering
			\begin{subfigure}{0.49\textwidth}
				\centering
				\begin{tikzpicture}[->]
							\node (01) at (0,0) {$S_8(A_1)$};
					\node (11) [right=12pt of 01] {$S_4(A_3+2A_1)$};
					\node (10)[above of=11]   {$S_4(A_3+2A_1)$};
					\node (12)[below of=11]   {$S_4(A_3+2A_1)$};
					\node (21)[right=12pt of 11]   {$S_2(D_4+3A_1)$};
					
					\draw (01)--(10); \draw(01)--(11); \draw(01)--(12);
					\draw(10)--(21);
					\draw(10)--(21); \draw (11)--(21); \draw (12)--(21);
					
%					\node (01) at (14,0) {};
				\end{tikzpicture}
			\end{subfigure}
			\hfill
			\begin{subfigure}{0.49\textwidth}
				\centering
			\end{subfigure}
		\end{figure}

	\item $C_4$
\begin{table}[H]
	\tabcolsep=0.5cm
	\begin{tabular}{ll}
		$C_1\ra C_2\ra C_4$&$S_4\ra S_2(4A_1(20l))\ra S_1(2A_3+A_1)$\\
		$S_4(A_2)\ra S_2(A_5+A_1(6l))\ra S_1(D_5+A_3)$&
		$S_4\ra S_2(A_3+2A_1(12l))\ra A_1(A_7+A_1)$
%		$S_4(4A_1)\ra S_2(6A_1)\ra S_1(2A_3+2A_1)$
	\end{tabular}
	\end{table}

	\item $C_3$
	\begin{table}[H]
		\tabcolsep=1cm
		\begin{tabular}{ll}
			$C_1\ra C_3$&$S_3(3A_1)\ra S_1(3A_2+A_1)$\\
			$ S_3(A_1)\ra S_1(A_5+A_2)$&$S_3(D_4)\ra S_1(E_6+A_2)$\\
			$S_3\ra S_1(3A_2)$&$S_3(A_2)\ra S_1(A_8)$\\
			$S_6(A_1)\ra S_2(A_5+A_2)$&\\
			%repeated
%			$S_3(3A_2)\ra S_1(4A_2)$$S_6\ra S_2(3A_2)$
		\end{tabular}
	\end{table}
	\item $C_2$
	\begin{table}[H]
		\tabcolsep=1cm
		\begin{tabular}{ll}
			$C_1\ra C_2$&$S_2(A_2)\ra S_1(A_5+A_1(21l))$\\
			$S_2(E_6)\ra S_1(E_7+A_1)$&$S_2(2A_2+A_1)\ra S_1(A_3+A_2+2A_1)$\\
			$S_2(A_1)\ra S_1(A_3+2A_1(44l))$&$S_2(A_2+2A_1)\ra S_1(A_5+2A_1)$\\
			$S_2(A_3)\ra S_1(D_4+2A_1)$&$S_2(A_5(8l))\ra S_1(D_5+2A_1)$\\
			$S_2(3A_1(26l))\ra S_1(A_3+3A_1)$&$S_2(2A_2)\ra S_1(A_2+4A_1)$\\
			$S_2(2A_1)\ra S_1(5A_1)$&$S_2(A_3+A_1(16l))\ra S_1(D_4+A_3)$\\
			$S_2(2A_1)\ra S_1(2A_3(23l))$&$S_2(A_3)\ra S_1(A_7(8l))$\\
			$S_2(D_5)\ra S_1(D_8)$&$S_2\ra S_1(4A_1(77l))$\\
			$S_4(3A_1)\ra S_2(A_3+3A_1)$&$S_4(D_4)\ra S_2(D_6+A_1)$\\
			$S_6(A_1(4l))\ra S_3(A_3+2A_1)$&$S_6(A_2)\ra S_3(A_5+A_1)$
			
			%repeated
%			$S_2(2A_3+A_1)\ra S$
		\end{tabular}
	\end{table}
\end{enumerate}

\begin{rema}
	The singularity type of the universal cover of relative minimal Du Val del Pezzo surfaces was determined in \cite{MZ88,MZ93}.
	After the completion of an earlier draft of this paper, we learned from J\'anos Koll\'ar that our list in Theorem \ref{theo:classification} has an overlap with Table $2$ of \cite{Kol2009}, where the universal cover of relative minimal Du Val del Pezzo surfaces with cyclic quotient singularities is listed.
\end{rema}

\subsection{Quasi-\'etale covers with equal or higher $b$-invariants}

Let $\pi: T \to S$ be a quasi-étale cover of Du Val del Pezzo surfaces, and let $Y$ and $X$ be the minimal resolutions of $T$ and $S$, respectively. Classifying $\pi$ with $\rho(Y) \geq \rho(X)$ reduces to classifying subgroups $H = H_T \times H_S \subseteq \Cris(\pi)$ such that $\rho(Y, H_T) \geq \rho(X, H_S)$, up to conjugation, by the results in Section \ref{sec:corr}.

We use \texttt{Magma} \cite{magma} to compute $\Cris(X)$ and its subgroups up to conjugacy for all types of $X$. The source code and outputs are available on GitHub \cite{github}.
In the following, we only present partial results for $S$ of degree $1$ due to the large number of cases.

\begin{prop}\label{prop:classification2}
	Let $\pi:T\ra S$ be a quasi-\'etale cover of Du Val del Pezzo surfaces such that $S$ is of degree $2$, or of degree $1$ and with at most $4$ lines.
	There exists a subgroup $H=H_T\times H_S\subseteq\Cris(\pi)$ such that $\rho(Y,H_T)\geq\rho(X,H_S)$ (\resp~$\rho(Y,H_T)>\rho(X,H_S)$) if and only if the column $(\geq)$ (\resp~$(>)$) is YES in Table \ref{bigtable}.\qed
\end{prop}

In Appendix \ref{sec:app}, we list all groups $H_S\subset \Cris(X)$ satisfying the conditions in Proposition \ref{prop:classification2} up to conjugacy.

\subsection{Descend to non-closed fields: from top to bottom}\label{sec:descend1}
Let $F$ be a field of characteristic $0$.
In this section, we study the existence of quasi-\'etale quotients of a fixed surface over $F$.
\begin{prop}\label{prop-linearized}
	Let $\pi:T\ra S$ be a quasi-\'etale cover of Du Val del Pezzo surfaces over a field $F$ of characteristic $0$.
	Write $G\coloneqq\Aut(T/S)$.
	Using the notations from (\ref{diag-model}), the following assertions hold.
	\begin{enumerate}
		\item The pullbacks $\pi^\ast:\Pic(S)_\bQ\hookrightarrow \Pic(T)_\bQ$ and $f^\ast:\Pic(X)_\bQ\hookrightarrow Pic(T^\prime)_\bQ$ are injective;
		\item We have $\rho(T)^G=\rho(S)$ and $\rho(T^\prime)^G=\rho(X)$. \label{item2}
	\end{enumerate}
\end{prop}
\begin{proof}
	For a finite group $G$ and a $G$-variety $T$,
	we have the long exact sequence
	\begin{equation}\label{eq-linearizable}
		0\to \Hom(G,F^\times)\to \Pic_G(T)\to \Pic(T)^G\to H^2(G,F^\times)\to\cdots,
	\end{equation}
	where $\Pic_G(T)$ is the group of $G$-linearized line bundles.
	Since the terms at both ends of (\ref{eq-linearizable}) are torsion, (\ref{eq-linearizable}) induces an isomorphism $\Pic_G(T)_\bQ\to \Pic(T)^G_\bQ$.
	Suppose the GIT quotient $S:=T\sslash G$ exists.
	We have an injection $\Pic(S)\hookrightarrow \Pic_G(T)$ which is an isomorphism after tensoring by $\bQ$ \cite[Proposition 4.2]{knop1989picard}.
	We thus have
	$$\Pic(S)_\bQ\cong \Pic_G(T)_\bQ\cong \Pic(T)^G_\bQ\hookrightarrow \Pic(T)_\bQ.$$
	Applying the above arguments to $\pi:T\ra S$ and $f:T^\prime\to X$ in the proposition, we complete the proof.
\end{proof}

We can reduce the realization problem of subgroups of $\Cris(\pi)$ to the existence of a certain $G$-variety.
For instance, let $T$ be a del Pezzo surface of degree $6$ over $F$ and $H\subseteq \fS_4$ be a subgroup.
Then the following assertions are equivalent.
\begin{itemize}
	\item There exists a quasi-\'etale cover of type $S_6\to S_3(4A_1)$ with $\Gal(\overline{F}/F)$ permutes the four $A_1$-points as $H$.
	\item There exists a $C_2$-action on $T$ such that the opposite lines are swapped and that the four fixed points are permuted by $\Gal(\overline{F}/F)$ as $H$.
\end{itemize}
From this perspective, one can also compute $\rho(X)$ quickly as follows.

\begin{exam}
	Let $\pi:T\to S$ be a quasi-\'etale cover of type $S_6\to S_3(4A_1)$, given by the action of $G=C_2$ on $T$.
	Suppose $T$ is split.
	Then $\rho(Y)=\rho(T)=4$ and $\rho(S)=\rho(T)^G=3$ by Proposition \ref{prop-linearized}.
	Suppose further that the four fixed points of $G$ forms a single Galois orbit.
	Then blowing up a closed point of degree $4$ resolves $S$ and $\rho(X)=\rho(S)+1=4$.
\end{exam}

\subsection{Descend to non-closed fields: from bottom to top}\label{sec:descend}
Let $F$ be a field of characteristic $0$.
In this chapter we study the existence of quasi-\'etale covers over a fixed surface over $F$.
%
%First of all, even over a closed field, the morphism $\Aut(X)\ra \Cris(X)$ may not be surjective.
%The automorphism group of Du Val del Pezzo surfaces of certain types was recently determined in \cite{virin2023automorphisms}.

We adopt the language of finite descent theory developed by Harpaz and Wittenberg \cite{harpaz2024supersolvable}.
Let $X$ be the smooth locus of a Du Val del Pezzo surface $X^\prime$ over $F$.
Not every type $\Type(\pi)$ can be realized over $X$.
A \textit{finite descent type} on $X$ is an irreducible finite \'etale $X_{\overline{F}}$-scheme $\overline{Y}$ that is Galois over $X$.
Let $\overline{G}$ be a quotient group of $\pi_1^\et(X_{\overline{F}})$.
The isomorphism classes of finite descent type of $\overline{G}$-torsors over $X$ is parameterized by the set
\begin{equation}\label{eq:finite descent type}
	H^0(F,H^1_{\et}(X_{\overline{F}},\overline{G})).
\end{equation}
A basic fact from \'etale cohomology theory is that
$$H^1_{\et}(X_{\overline{F}},\overline{G})\cong \Hom_{\cont}(\pi_1^\et(X_{\overline{F}}),\overline{G})/\sim,$$
where the right-hand-side is the set of conjugate classes of subgroups $\overline{H}$ of $\pi_1^\et(X_{\overline{F}})$ such that $\pi_1^\et(X_{\overline{F}})/\overline{H}\cong\overline{G}$.
Thus the size of the set (\ref{eq:finite descent type}) is determined by the splitting type of $X^\prime$.
For example, when $X^\prime$ is of type $S_1(4A_2)$,
the set (\ref{eq:finite descent type}) consists of $n$ elements, where $n$ is the number of $A_2$-points in $X^\prime$ that is defined over $k$.

Let $\overline{Y}$ be a finite descent type over $X$ with $\Aut(\overline{Y}/X_{\overline{F}})=\overline{G}$.
Since $\overline{G}$ is abelian, it canonically descents to a finite \'etale group scheme $G$ over $F$.
By \cite[Proposition 2.3]{harpaz2024supersolvable}, the isomorphism classes of torsor of type $\overline{Y}$ correspond bijectively to splittings of the short exact sequence
\begin{equation}\label{eq:short exact sequence}
	1\to \Aut(\overline{Y}/X_{\overline{F}})\to \Aut(\overline{Y}/X)\to \Gal(\overline{F}/F)\to 1,
\end{equation}
up to conjugation.
When $X(k) \neq \emptyset$, a splitting of (\ref{eq:short exact sequence}) exists, and group cohomology theory shows that the splittings are in turn parametrised by $$H^1(\Gal(\overline{F}/F),\Aut(\overline{Y}/X_{\overline{F}}))=H^1(F,G).$$
Thus, taking the neutral element of $H^1(F,G)$, we can get a torsor of type $\overline{Y}$ with trivial Galois action on the vertical part $\Pic(T^\prime_{\overline{F}})_\bQ/f^\ast\Pic(X_{\overline{F}})_\bQ$ of Picard groups.
This shows the following proposition.

\begin{prop}\label{prop-descend}
	We use the notations from (\ref{diag-model}).
	Let $S$ be a Du Val del Pezzo surface over $F$ and let $\overline{T}$ be a finite descent type over $S_{\sm}$.
	Suppose $S_{\sm}(F)\neq\emptyset$.
	Then there exists a quasi-\'etale cover $\pi:T\ra S$ of type $\overline{T}$ such that
	$\Gal(\overline{F}/F)$ acts trivially on the quotient group $\Pic(T^\prime_{\overline{F}})_\bQ/f^\ast\Pic(X_{\overline{F}})_\bQ$.
	In particular, the $b$-invariant of $T$ reaches the maximum in the sense that  $$\rho(Y)=\max\{ \rho(Y,H_T)\mid H_T \text{ satisfies } H_S\times H_T\in\Cris(\pi) \},$$
	where $H_S$ is the splitting group of $S$.\qed
\end{prop}

\section{Examples}\label{sec:examples}
In this section, we describe precise examples with Zariski dense exceptional sets. Throughout this section, we denote by $X$ and $Y$ the minimal resolutions of the surfaces $S$ and $T$, respectively.
%These examples are interesting in different aspects.
%The surface $S$ in Section \ref{sec:deg3,4A1} is the unique singular cubic surface with a Zariski dense exceptional set up to a twist over $\bC$.
%The Cox rings of the surfaces $S$ in Section \ref{sec:deg2,D4+3A1} and \ref{sec:deg1,E6+A2} are defined by only one relation (see \cite{Derenthal2013}).
%Section \ref{sec:deg2,3A2} and \ref{sec:deg1,E6+A2} provide examples where the $b$-invariants are violated by the quasi-\'etale covers.

\subsection{Degree $3$, type $4A_1$}\label{sec:deg3,4A1}
\begin{nota}\label{nota: deg3, type 4A1}
	Let $S$ be the singular cubic surface in $\bP^3$ defined by the equation
	$$X^3 + 2XYW + XZ^2 -Y^2Z + ZW^2=0$$
	%	$$4W^3-(Z^2-4XY)W-(X^2+Y^2)Z=0$$
	over $\bQ$. 
\end{nota}
\begin{prop}\label{prop:deg3, type 4A1}
	Under Notation \ref{nota: deg3, type 4A1}, the following assertions hold.
	\begin{enumerate}
		\item $S$ is a cubic surface with four $A_1$ singularities (a twist form of the Cayley cubic).
		\item We have that $\rho(X)=4$.
		\item There exists a quasi-\'etale $T\ra S$ defined over $\bQ$ satisfying $\rho(Y)=4$.
	\end{enumerate}
\end{prop}
\begin{proof}
	The four singular points on $S$ are
	\begin{align*}
		\Sing(S)=\{(1 : \zeta_8 : \zeta_8^2 : \zeta_8^3),(1 :\zeta_8^3 : \zeta_8^6 : \zeta_8),
		(1 : \zeta_8^5 : \zeta_8^2 : \zeta_8^7), (1 : \zeta_8^7 : \zeta_8^6 : \zeta_8^5)\}.
	\end{align*}
	%	\begin{align*}
		%		\Sing S=\{ (-\zeta_8^3 - \zeta_8 : -\zeta_8^3 - \zeta_8 : 2 : 1), (-\zeta_8^3 + \zeta_8 : \zeta_8^3 - \zeta_8 : -2 : 1), \\
		%		(\zeta_8^3 - \zeta_8 : -\zeta_8^3
		%		+ \zeta_8 : -2 : 1), (\zeta_8^3 + \zeta_8 : \zeta_8^3 + \zeta_8 : 2 : 1) \}.
		%	\end{align*}
	The Galois group $\Gal(\overline{\bQ}/\bQ)$ acts on the set $\Sing(S)$ by $C_2^2(2)\subset \fS_4$ in the notation of Section \ref{deg3,4A1} (\ie~the unique normal subgroup of $\fS_4$ which is isomorphic to $C_2^2$).
	Thus $\rho(S)=4$ by Table \ref{tab-deg3,4A1}.
	Indeed, if the action was $C_2^2(1)$, then the induced action on $\Gamma(X)$ would not fix the three lines $E_5,E_9$, and $E_{12}$.
%	Since there exists a smooth rational point on $S$, the homotopy exact sequence of \'etale fundamental groups implies that the induced quasi-\'etale cover $\overline{T}\ra \overline{S}$ descends to a cover $T\ra S$ over $\bQ$.
%	After possibly a twisting, we may assume that the Galois action on $\Gamma(T)$ is trivial since any automorphisms on $\Gamma(T)$ can be realized as $\Aut(T)$ \cite{corn2005pezzo}.
	Since $S$ admits a smooth rational point, the arguments of Section \ref{sec:descend} imply the existence of a quasi-\'etale cover $T\ra S$, where $T$ reaches the maximal $b$-invariant of $4$ as recorded in Table \ref{tab-deg3,4A1}.
\end{proof}

\begin{exam}[Construction of the example]
	It is well-known that the Cayley cubic surface can be written as a hypersurface $S_0$ in $\bP^3_\bQ$ defined by the equation 
	$$\frac{1}{x}+\frac{1}{y}+\frac{1}{y}+\frac{1}{z}=1.$$
	The surface $S_0$ is split over $\bQ$, and the singular points are
	$$\Sing(S_0)=\{ (1:0:0:0),(0:1:0:0),(0:0:1:0),(0:0:0:1) \}.$$
	Pick a number field $F$ such that $[F:\bQ]=4$ and $\Gal(F/\bQ)$ acts on the four roots $\{ a_0,a_1,a_2,a_3 \}$ of the minimal polynomial by $C_2^2(2)\subset \fS_4$.
	Let
	$$M:=\left(
	\begin{array}{cccc}
		1&1&1&1\\
		a_0&a_1&a_2&a_3\\
		a_0^2&a_1^2&a_2^2&a_3^2\\
		a_0^3&a_1^3&a_2^3&a_3^3
	\end{array}
	\right)$$
	be the Vandermonde matrix associated to $(a_0,a_1,a_2,a_3)$.
	Then the change of variables by $M$ defines a surface $S_1$ with the desired Galois action on $\Sing(S_1)$.
	Since $S_0$ is defined by a symmetric polynomial, the coefficients of the equation for $S_1$ are symmetric polynomials in $a_i$, and therefore they lie in $\mathbb{Q}$. The surface $S$ in Notation \ref{nota: deg3, type 4A1} is obtained in this way by inputting the number field $\mathbb{Q}(\zeta_8)$.
\end{exam}

\subsection{Degree $2$, type $D_4+3A_1$}\label{sec:deg2,D4+3A1}
\begin{nota}\label{nota:D4+3A1}
		Over the field $\bQ$, define a morphism of varieties as follows.
	\begin{align*}
		T=\Proj\left(\frac{k_{(1,1,2,2)}[x,y,z,w]}{y^4-16zw}\right)
		&\ \overset{\pi}{\rightarrow}\ 
		S=\Proj\left(\frac{k_{(1,1,1,2)}[X,Y,Z,W]}{W^2-XY(Z^2+Y^2)}\right),\\
		(x:y:z:w)&\ \mapsto\ (x^2:y^2:2(z-w):2xy(z+w)).
	\end{align*}
	It is readily to confirm that $\pi$ is well-defined.
\end{nota}

\begin{prop}\label{prop:deg2,D4+3A1}
	Under Notation \ref{nota:D4+3A1}, the following statements hold.
	\begin{enumerate}
		\item $S$ is a Du Val del Pezzo surface of type $S_2(D_4+3A_1)$, and $T$ is a Du Val del Pezzo surface of type $S_4(A_3+2A_1)$.
		\item We have that $\rho(Y)=\rho(X)=6$.
	\end{enumerate}
\end{prop}
\begin{proof}
	For (1), note that we have an isomorphism
	\begin{align*}
		T=\Proj\left(\frac{\bQ_{(1,1,2,2)}[x,y,z,w]}{y^4-zw}\right)
		&\ {\rightarrow}\ 
		T^\prime=\Proj\left(\frac{\bQ_{(1,1,1,1,1)}[X,Y,Z,S,T]}{(Y^2-XZ, Z^2-ST)}\right),\\
		(x:y:z:w)&\ \mapsto\ (x^2:xy:y^2:z:w).
	\end{align*}
	For (2), write $Y$ and $X$ for the minimal resolution of $T$ and $S$.
	Note that the singular points on $T$ and $S$ are
	$$\Sing~T=\{ A_3:(1:0:0:0),A_1:(0:0:1:0),A_1:(0:0:0:1)\},$$
	$$\Sing~S=\{D_4:(1:0:0:0),A_1:(0:0:1:0), A_1:(0:-i:2:0),A_1:(0:i:2:0)\}.$$
	So $\Gal(\overline{\bQ}/\bQ)$ acts on $\Gamma(X)$ so that $\rho(X)=6$ as presented in Section \ref{deg2,D4+3A1}.
	The Galois action on $\Gamma(Y)$ is trivial since the two lines on $T$ are not interchanged by the action.
\end{proof}

	We present in the following example how $T\ra S$ is constructed.
	\begin{exam}[Alternative construction of the example]
		Define a group action on $\bP^2_\bC$ by the quaternion group $Q_8=\{ \pm 1, \pm I,\pm J,\pm K \}$ by
		$$I:(X:Y:Z)\mapsto (iX,-iY:Z),\quad J:(X:Y:Z)\mapsto (\zeta_8 Y:\zeta_8^3 X:Z),$$
		where $\zeta_8=\sqrt{2}(1+i)/2$.
		Then we have a tower of quasi-\'etale covers
		$$U:=\bP^2_\bC/\langle I^2 \rangle\overset{\pi_2}{\ra} T:=\bP^2_\bC/\langle I \rangle\overset{\pi_1}{\ra} S:=\bP^2_\bC/Q_8.$$
		The rings of invariants are
		\begin{align*}
			\bC[X,Y,Z,W]^{\langle I\rangle}&=\bC[Z,XY,X^4,Y^4],\\
			\bC[X,Y,Z,W]^{Q_8}&=\bC[Z,X^2Y^2,X^4-Y^4,XY(X^4+Y^4)]\\
			&=\bC[f_1,f_4,g_4,f_6]/(f_6^2-f_4g_4^2-4f_4^3).
		\end{align*}
		Thus
		\begin{align*}
			S&=\Proj\left(\frac{k_{(1,4,4,6)}[x,y,z,w]}{w^2-yz^2-4y^3}\right)
			\cong \Proj\left(\frac{k_{(1,1,1,2)}[X,Y,Z,W]}{W^2-XY(Z^2+Y^2)}\right).
		\end{align*}
		The equations of $T$ and $\pi$ can be obtained in the same way.
	\end{exam}

\subsection{Degree $2$, type $3A_2$}\label{sec:deg2,3A2}
\begin{nota}\label{nota:deg2, 3A2}
	Let $S$ be the hypersurface in $\bP_\bQ(1,1,1,2)$ defined by the equation
	%$$W^2=-4X^4 + 12X^2YZ -8XY^3 -4XZ^3 + 3Y^2Z^2$$
	%The $3$ singular points on $S$ are defined over the field $\bQ[t]/(t^3-2)$, whose Galois group is $\fS_3$.
	$$W^2+(3X^2-YZ)W+9X^4-6X^2YZ+X(Y^3+Z^3)=0.$$
	Let $T$ be the subvariety of $\bP^6_{b,a_1,\cdots,a_6}$ cutting out by
	$$\{ ba_i=a_{i-1}a_{i+1},b^2=a_ia_{i+3}\mid i=1,\cdots,6 \}.$$
	Define the morphism $\pi:T\ra S$ by
	$$\pi:(b:a_1:\cdots:a_6)\mapsto (b:a_1+a_3+a_5:a_2+a_4+a_6:a_1a_2+a_3a_4+a_5a_6)=(X:Y:Z:W).$$
	We will show that $\pi$ is well-defined.
\end{nota}

\begin{prop}\label{prop:deg2, 3A2}
		The following statements hold.
	\begin{enumerate}
		\item $S$ is a Du Val del Pezzo surface of type $S_2(3A_2)$, and $T$ is a del Pezzo surface of degree $6$.
		\item The morphism $\pi$ is a well-defined quasi-\'etale cover.
		\item We have that $\rho(Y)=6>5=\rho(X)$.
	\end{enumerate}
\end{prop}
\begin{proof}
	We will show these assertions by constructing $S$ as a quotient variety of $T$.
	It is well-known that $T$ is the anticanonical model of a del Pezzo surface of degree $6$ which is split over $\bQ$.
	Let $G=S_3$ acts on $T$ by
	$$(b,a_1,a_2,a_3,a_4,a_5,a_6)\mapsto (b,a_3,a_4,a_5,a_6,a_1,a_2).$$
	The action interchanges diagonal lines.
	Hence the quotient variety $T/G$ is of type $S_2(3A_2)$ by the diagram in Section \ref{deg2,3A2}.
	We obtain the equations of $S$ and $\pi$ as in Notation \ref{nota:deg2, 3A2} by computing the ring of invariants by $\mathtt{Magma}$.
	
 The singular locus of $S$ is
%	{@ (1/3 : 3 : 3 : 1 : 3 : 1 : 1), (1/3 : -3*a - 3 : 3*a : a : 3 : -a - 1 : 1),
%		(1/3 : 3*a : -3*a - 3 : -a - 1 : 3 : a : 1) @}
	$$\Sing(S)=\{ (1:3:3:3),(1:3\zeta_3,3\zeta_3^2:3),(1:3\zeta_3^2,3\zeta_3:3) \}.$$
	We see that the Galois group acts on $\Sing(S)$ as $C_2$ and trivially on $\Gamma(T)$.
	There are three conjugate classes of subgroups of $\Cris(S)$ that act on $\Sing(S)$ as $C_2$.
	The only one that fixes $\Gamma(T)$ is $C_2(1)$ (see Section \ref{deg2,3A2}).
	Thus, the action on $\Gamma(S)$ is $C_2(1)$,
	and we have that $\rho(T)=6>5=\rho(S)$.
\end{proof}

\subsection{Degree $1$, type $E_6+A_2$}
\begin{nota}\label{nota:deg1, E6+A2}
	Let $S$ be the hypersurface in $\bP_\bQ(1,1,2,3)$ defined by the equation
	$$W^2+Z^3+X^4Y^2=0.$$
%	Let $T$ be the cubic surface in $\bP_\bQ^3$ defined by the equation
%	$$(x_1x_2x_3)^2+(x_1x_2+x_2x_3+x_3x_1)^3+x_0^4(x_1+x_2+x_3)^2=0.$$
\end{nota}

\begin{prop}\label{prop:deg1,E6+A2}
The following assertions hold.
\begin{enumerate}
	\item $S$ is a Du Val del Pezzo surfaces of type $S_1(E_6+A_2)$.
	\item There exits a quasi-\'etale cover $\pi:T\ra S$ where $T$ is a Du Val del Pezzo surface of type $S_3(D_4)$, satisfying $\rho(Y)=7>6=\rho(X)$.
\end{enumerate}
\end{prop}
\begin{proof}
	$S$ is of type $S_1(E_6+A_2)$ since the same equation is listed in \cite[Section 8]{CP20}.
	The four lines on $S$ are defined by
	\begin{align*}
		E_9&=\{ Z=W-iX^2Y=0 \},\quad
		&E_{10}&=\{ Z=W+iX^2Y=0 \},\\
		E_{11}&=\{ Y=W^2+Z^3=0 \},\quad
		&E_{12}&=\{ X=W^2+Z^3=0 \},
	\end{align*}
	which can be deduced from those on a twist form of $S$ in \cite{Derenthal2013}.
	As presented in Section \ref{sec:deg1,E6+A2}, the Galois group acts on $\Gamma(X)$ as $C_2$ and $\rho(X)=6$.
	
	Since $E_{11}$ and $E_{12}$ are fixed by the Galois action, there exists a finite descent type $\overline{T}$ over $S$ of type $S_3(D_4)\ra S_1(E_6+A_2)$.
	By Proposition \ref{prop-descend}, this descends to a quasi-\'etale cover $\pi: T \to S$ over $\bQ$ with minimal $\rho(T)$, which is $7$, as presented in Section \ref{sec:deg1,E6+A2}.
\end{proof}

\section{Supplements}\label{sec-supp}
\subsection{Proof of Theorem \ref{theo-guiding-principle} and \ref{theo-examples}}\label{sec-3.1}
Recall the following definition in \cite{BL18}.
\begin{defi}
	A smooth projective geometrically integral variety $X$ over a
	field F is called \textit{almost Fano} if
	\begin{itemize}
		\item $H^i(X,\cO_X)=0$ for $i=1,2$.
		\item The geometric Picard group $\Pic(\overline{X})$ is torsion-free.
		\item The anticanonical divisor $-K_X$ is big.
	\end{itemize}
\end{defi}
In particular, a weak del Pezzo surfaces is an almost Fano variety.
%\begin{lemm}
%	\label{lemm:mainthmC}
%	Let $\pi:T\ra S$ be a quasi-\'etale cover of Du Val del Pezzo surfaces.
%	
%	\begin{enumerate}
%		\item  Suppose that $\rho(\tT)\geq\rho(\tS)$ and that there exists a dense open subset $V$ of $T$ such that the equation
%		$$N(V,-\cK_T,B)\sim c(T,-\cK_T)B(\log B)^{\rho(\tT)-1}$$
%		holds true for any adelically metrized anticanonical line bundle $-\cK_T$ on $T$.
%		Then for any open subset $U$ of $S$, there exists an adelically metrized anticanonical line bundle $-\cK_S$ on $S$ such that the equation
%		\begin{equation}\label{eq11}
%			N(U,-\cK_S,B)\sim c(S,-\cK_S)B(\log B)^{\rho(\tS)-1}
%		\end{equation}
%does not hold true.
%		
%		\item  Suppose that $\rho(\tT)>\rho(\tS)$ and that for any dense open subset $V$ of $T$ and any adelically metrized anticanonical line bundle $-\cK_T$ on $T$, we have that
%		\begin{equation}\label{eq2}
%			N(V,-\cK_T,B)\gg B(\log B)^{\rho(\tT)-1}.
%		\end{equation}
%		Then for any dense open subset $U$ of $S$ and any adelically metrized anticanonical line bundle $-\cK_S$ on $S$, the following assertion holds true:
%For any constant $C>0$, there exists $B_0$ such that for any $B\geq B_0$, we have that
%		$$N(U,-\cK_S,B)\geq CB(\log B)^{\rho(\tS)-1}.$$
%	\end{enumerate}
%\end{lemm}
\begin{proof}[Proof of Theorem \ref{theo-guiding-principle}]
	Since $\pi$ is quasi-\'etale, we have that $\pi^\ast K_S\sim K_T$ and we attach the pulled-back metric to $-K_T$.
	In the case of (1), fix a dense open subset $U$ of $S$ and
	suppose Manin's conjecture for $(U,-\cK_S)$ is true for any adelically metrized anticanonical line bundle $-\cK_S$.
	By \cite[Proposition 5.0.1]{Peyre95}, the rational points are equidistributed on $U$.
	But then
	\begin{align*}
		\lim_{B\ra\infty}\frac{N(\pi(T(\bQ))\cap U,-\cK_S,B)}{N(U,-\cK_S,B)}
		=\frac{c(T,-\cK_T)}{c(S,-\cK_S)}>0,
	\end{align*}
	which contradicts \cite[Theorem 1.2]{BL18}.
	
	In the case of (2),
	let $U_0$ be an dense open subset of $S$ over which $\pi$ is \'etale.
	 Let $U$ be any dense open subset of $X$.
  We have that
	\begin{align*}
		N(U,-\cK_S,B)\geq\frac{1}{\deg \pi} N(\pi^{-1}(U\cap U_0),-\cK_T,B)
	\end{align*}
	Moreover, the assumption (\ref{eq2}) implies that
 $$\liminf_{B\ra\infty}\frac{N(\pi^{-1}(U\cap U_0),-\cK_T,B)}{B(\log B)^{\rho(\tT)-1}}>0.$$
	Hence
	\begin{align*}
		\liminf_{B\ra\infty}\frac{N(U,-\cK_S,B)}{B(\log B)^{\rho(\tS)-1}}
		&\geq \frac{1}{\deg\pi}\liminf_{B\ra\infty}\left(\frac{N(\pi^{-1}(U\cap U_0),-\cK_T,B)}{B(\log B)^{\rho(\tT)-1}}\cdot(\log B)^{\rho(\tT)-\rho(\tS)}\right)\\
		&=+\infty.\qedhere
	\end{align*}\qedhere
\end{proof}

\begin{proof}[Proof of Theorem \ref{theo-examples}]
	Let $S$ be either of the surfaces in (\ref{enum:1a}) or (\ref{enum:1b}).
	By Proposition \ref{prop:deg3, type 4A1} and Proposition \ref{prop:deg2,D4+3A1}, there exists a quasi-\'etale cover $T\ra S$ such that $\rho(\tT)=\rho(\tS)$ and that $T$ is a toric Du Val del Pezzo surface.
	The closed-set version of Manin's conjecture with Peyre's constant is known for $T$ with any adelically metrized line bundle $-\cK_T$ \cite{BT98toric,chambert2010integral}, see also \cite{huang2021equidistribution}.
	Then, the assertion follows by (1) of Theorem \ref{theo-guiding-principle}.
	
	Let $S$ be the surface in (\ref{enum:2a}).
	By Proposition \ref{prop:deg2, 3A2}, there exists a quasi-\'etale cover $T\ra S$ such that $\rho(\tT)>\rho(\tS)$ where $T$ is a del Pezzo surface of degree $6$.
	The closed-set version of Manin's conjecture with Peyre's constant is known for $T$ with any anticanonical height \cite{BT98toric,chambert2010integral}.
	Thus, the assertion follows from (2) of Theorem \ref{theo-guiding-principle}.
	
	Let $S$ be the surface in (\ref{enum:2b}).
	By Proposition \ref{prop:deg1,E6+A2}, there exists a quasi-\'etale cover $T\ra S$ such that $b(T,-K_T)=7>6=b(S,-K_S)$, where $T$ is of type $S_3(D_4)$ and is split over $F$.
%	The surface $T$ admits a conic bundle structure since it is split.
%	By \cite[Theorem 1.1]{FLS18}, we have that
%	\begin{equation}
	%		N(V,-\cK_T,B)\gg B(\log B)^{\rho(\tT)-1}
	%	\end{equation}
%	for any dense open subset $V$ of $T$.
	Over $\bC$, there are two isomorphism classes of surfaces of type $S_3(D_4)$; see Section \ref{sec:deg1,E6+A2}.
	The example in Proposition \ref{prop:deg1,E6+A2} falls into case (1) of Section \ref{sec:deg1,E6+A2}.
	Let $\overline{T}$ be the unique surface in case (1) over $\bC$, up to isomorphism.
	By \cite[Claim 4.1]{virin2023automorphisms}, we have the short exact sequence
	\begin{equation}
		%\label{eq-linearizable1}
		0\to \bG_m\to\Aut(\overline{T})\to \Cris(\overline{T})\to 0.
	\end{equation}
	This induces the long exact sequence
	\begin{equation}
		%\label{eq-linearizable2}
		0\to H^1(\bQ,\Aut(\overline{T}))\to H^1(\bQ,\Cris(\overline{T}))\to \Br(\bQ)\to\cdots.
	\end{equation}
	Therefore, the Galois module structure of $\Pic(T)$ uniquely determines the isomorphism class of $T$ as a $\bQ$-variety.
	Thus $T$ in the example is isomorphic to the surface $T_0$ in $\bP^3_\bQ$ defined by
	\begin{equation}\label{eq-D4cubic}
		x_1x_2x_3 = x_4(x_1 + x_2 + x_3)^2.
	\end{equation}
	In \cite{browning2006density}, a weak version of Manin's conjecture was established for $T_0$:
	$$N(T_0\backslash Z,-K_{T_0},B)\asymp B(\log B)^{7-1},$$
	where $Z$ is the union of $(-1)$-curves.
	Since any two anticanonical height functions differ by a constant, this implies that
	$$N(U,-K_{T},B)\gg B(\log B)^{7-1},$$
	for any Zariski dense open subset $U$ of $T$.
	Thus, the assumption of Theorem \ref{theo-guiding-principle} (2) is satisfied, and Theorem \ref{theo-examples} follows.

\end{proof}

\subsection{Rationality}\label{sec-3.2}
It is also worth noting that the four examples in Theorem \ref{theo-examples} are all rational.
As an immediate conclusion, the assumption in Manin's conjecture that rational points do not form a thin set is satisfied.
\begin{prop}\label{prop:rationality}
    Let $S$ be any of the four surfaces defined in Theorem \ref{theo-examples}.
    Then $S$ is rational over its base field $\bQ$.
\end{prop}
\begin{proof}
%    It is easy to check that each of the four surfaces defined in Theorem \ref{theo-examples} has a rational point on its smooth locus.
	We only present the last case since the others can be proved in a similar way.

    Let $S$ denote the surface defined in (\ref{enum:2b}) of Theorem \ref{theo-examples} and $X$ be its minimal resolution.
    Using the notations from Section \ref{sec:deg1,E6+A2},
    we may contract the orbit $\{E_9,E_{10}\}$ to obtain a surface of type $S_3(D_4)$, and then contract the image of $\{E_2,E_6\}$ to obtain a surface of type $S_5(A_2)$, and then contract the image of $\{ E_3,E_5\}$ to obtain a surface of type $S_7(A_1)$.
    This can be contracted to a twist of $\bP^2$.
    Since $S$ has a smooth rational point, we conclude that $S$ is $\bQ$-rational.
\end{proof}

In fact, Trepalin \cite{trepalin2018quotients} proved that a finite quotient of a del Pezzo surface of degree $\geq5$ is rational.
This implies that $S$ in (\ref{enum:1a}) and (\ref{enum:2a}) are rational.
However, not all quasi-\'etale quotients of Du Val del Pezzo surfaces are rational,
and sometimes they are arithmetically interesting objects to study.
We will address this topic further in future work.

\subsection{Proof of Theorem \ref{theo-geometric-consistency-fails} and \ref{theo-geometrically-rational}}
\begin{proof}[Proof of Theorem \ref{theo-geometric-consistency-fails}]
	Let $f:Y\ra X$ be an adjoint rigid $a$-cover of a weak del Pezzo surface $X$.
	Let $\psi:X\ra S$ be the contraction of $(-2)$-curves.
	By \cite[Theorem 6.8]{LTDuke}, there exist a quasi-\'etale cover $\pi:T\ra S$ of Du Val del Pezzo surfaces and a birational morphism $\varphi:Y\ra T$ commuting with $\psi$, satisfying
	\begin{itemize}
		\item $a(Y,\varphi^\ast (-K_{T}))=a(Y,f^\ast L)$;
		\item $S,T,X$ and $Y$ are adjoint rigid with respective to the pullback of $-K_S$;
		\item $\varphi$ is the composition of a birational morphism to a weak del Pezzo surface and the contraction of $(-2)$-curves.
	\end{itemize}
	The last property implies that $b(Y,\varphi^\ast\pi^\ast K_S)=b(S,-K_S)$ by \cite[Lemma 4.12]{LST}.
	Thus there exists an adjoint rigid $a$-cover $\pi:Y\ra X$ satisfying \ref{eq-ab} if and only if there exists a quasi-\'etale cover $\pi:T\ra S$ with $\rho(\widetilde{T})\geq\rho(X)$.
	The theorem then follows from Proposition \ref{prop:picard rank} and \ref{prop:classification2}.
\end{proof}

\begin{proof}[Proof of Theorem \ref{theo-geometrically-rational}]
	 By \cite[Theorem 6.10]{LTDuke}, there exists a birational morphism $\psi:X\ra V$ satisfying
	 \begin{itemize}
	 	\item $V$ is a weak del Pezzo surface;
	 	\item $a(Y,f^\ast\varphi^\ast (-K_{V}))=a(Y,f^\ast L)$;
	 	\item $(Y,f^\ast\varphi^\ast (-K_{V}))$ is adjoint rigid.
	 \end{itemize}
	 We thus obtain an adjoint rigid $a$-cover $\psi \circ f$ of $(V, -K_V)$, which reduces to Theorem \ref{theo-geometric-consistency-fails}.
\end{proof}

\subsection{Geometric exceptional sets}\label{sec-ges}
We restate the conjectural exceptional set proposed in \cite{LST} for Du Val del Pezzo surfaces (see \cite[Section 9]{LT19}).
% where we make a slight improvement using Proposition \ref{prop:twist}.
\begin{nota}(Geometric exceptional set)
	\label{nota:manin conjecture}
	Let $S$ be a Du Val del Pezzo surface of degree $d$ defined over a number field $F$
	and $-\cK_S$ be an adelically metrized anticanonical line bundle on $X$.
	\begin{enumerate}
		\item (Accumulating subvarieties) Let $Z_1$ be the union of $(-1)$-curves on $S$.
		\item (Adjoint rigid $a$-covers) Let $Z_2$ be
		% \begin{enumerate}
			% [leftmargin=0cm, itemindent=0.5cm]
			% \item the empty set when $S$ is toric;
			the union of $\pi(T(F))$ where $\pi:T\ra S$ is a quasi-\'etale cover of degree $\geq2$ such that $b(T)\geq b(S)$ and $\pi$ is face-contracting.
			
		% \end{enumerate}
		\item (Non-adjoint rigid $a$-covers) Let $Z_3$ be
		\begin{enumerate}
			[leftmargin=0cm, itemindent=0.5cm]
			\item the empty set if either $d\geq3$ or $\rho(\tS)\geq2$;
			\item the union of rational curves in $\left|-K_S\right|$ when $d=2$ and $\rho(\tS)=1$;
			\item the union of rational curves in $\left|-2K_S\right|$ and rational curves $C$ on $S$ with $C^2=2$ and $K_S\cdot C=-2$, when $d=1$ and $\rho(\tS)=1$.
		\end{enumerate}
	\end{enumerate}
\end{nota}
We call $Z:=Z_1\cup Z_2\cup Z_3$ the \textit{geometric exceptional set} of $(S,-K_S)$.
%\begin{conj}
%% (Manin's conjecture with the conjectural exceptional set)
%	Under Notation \ref{nota:manin conjecture},
%    Conjecture \ref{conj:maninfordvdp} holds true for
%    $Z:=Z_1\cup Z_2\cup Z_3$.
%\end{conj}

We verify that the quasi-\'etale covers in the examples of Section \ref{sec:examples} are the only ones contributing to the geometric exceptional sets.
\begin{prop}\label{prop:ges-for-examples}
    Let $S$ be one of the Du Val del Pezzo surface in Theorem \ref{theo-examples} and
    let $\pi:T\ra S$ be the corresponding quasi-\'etale cover considered in Section \ref{sec:examples}.
    Then under Notation \ref{nota:manin conjecture}, we have that $Z_2=\pi(T(\bQ))$ and $Z_3=\emptyset$.
\end{prop}
\begin{proof}
    The set $Z_3$ is empty since $\rho(X)\geq2$ for any Du Val del Pezzo surface $S$ which is not smooth.
    To prove $Z_2=\pi(T(\bQ))$, one first check Proposition \ref{prop:classification2} that $\Type(\overline{\pi})$ is the only type over $\Type(\overline{S})$ for which there exists a subgroup $H_S\times H_T\subseteq\Cris(\pi)$ such that $\rho(\tT,H_T)\geq\rho(\tS,H_S)$.
    
    There exists only one arrow of $\Type(T)\ra\Type(S)$ in Theorem \ref{theo:classification} , except when $\Type(S)=S_2(D_4+3A_1)$:
    In this case, only one of the three quasi-\'etale covers is a finite descent type over $S_{\sm}$. 
    In each case, the quasi-\'etale cover is face-contracting since there exist two extremal curves map to one.
    So we have
    $$Z_2=\cup \{ \pi^\sigma(T^\sigma(F))\mid\forall \sigma, b(T^\sigma(F))\geq b(S) \}.$$
%    Then, there is only one candidate of a branch divisor on $S$, which determines a unique quasi-\'etale cover $\overline{\pi}:\overline{T}\ra \overline{S}$ up to automorphisms of $\overline{T}$ over $\overline{S}$.
%    Moreover, $\overline{\pi}$ descents to $\bQ$ since there is a smooth $\bQ$-point on $S$.
%    For $\sigma\in H^1(\Gal(\overline{\bQ}/\bQ),\Aut(\overline{T}/\overline{S}))$, write $\pi^\sigma:T^\sigma\ra S$ for the twist of $T$ by $\sigma$ and write $Y^\sigma$ for the minimal resolution of $T^\sigma$.
%    Then $Z_2$ is the union of $\pi(T^\sigma(\bQ))$ such that $\rho(Y^\sigma)\geq \rho(X^\sigma)$ and that $\pi^\sigma$ is face-contracting.

    For each of the four examples in Theorem \ref{theo-examples}, one see $\Gal(\overline{F}/F)$ acts trivially on the vertical Picard group $\Pic(T^\prime_{\overline{F}})_\bQ/\Pic(\overline{X})_\bQ$.
    Thus by Proposition \ref{prop-descend}, we have
    $$\rho(Y)=\max\{ \rho(Y,H_T)\mid \forall H_T \text{ satisfying } H_S\times H_T\in\Cris(\pi) \}.$$
    One can verify that $\rho(Y,H_T)<\rho(X,H_S)$ for any $H_T$ satisfying $H_S\times H_T\in\Cris(\pi)$, except for the unique one attaining the maximum.
%    \begin{enumerate}
%        \item $T$ corresponds to the neutral element in $H^1(\Gal(\overline{\bQ}/\bQ),\Aut(\overline{T}))$; and
%         \item The natural map $\Aut(\overline{T}/\overline{S})\ra \Cris(\overline{T})$ is injective; and
%        \item The only subgroup $H_T^\prime$ of $\Cris(\overline{T})$ such that $H_S\times H_T^\prime\subseteq\Cris(\pi)$ and $\rho(\tT, H_T^\prime) \geq \rho(\tS)$ is $H_T^\prime=H_T$.
%    \end{enumerate}
%    These assertions show that any nontrivial twist of $T$ over $S$ would lower the $b$-invariant of $T$ below that of $S$.
    So the only twist of $\pi$ that contributes to the geometric exceptional set is $\pi$ itself.
    Hence $Z_2=\pi(T(\bQ))$.
\end{proof}

\appendix
\section{Appendix: tables and figures}\label{sec:app}

In this section, we list possible Galois actions on accumulating quasi-\'etale covers of Du Val del Pezzo surfaces.
In Table \ref{bigtable}, we list all types of Du Val del Pezzo surfaces that admit a nontrivial quasi-\'etale cover, except in degree $1$.
In the following sections, all group actions leading to accumulating are listed. Some cases are omitted because the dual graph is too complicated.

In the dual graphs, there are two types of vertices: a rectangle represents a $(-1)$-curve and a circle represents a $(-2)$-curve. The number of edges between two vertices means the intersection number of the corresponding curves.
A $(-2)$-curve where $\pi$ is branched is denoted by a thick circle.
The negative curves on $S$ and $T$ are labeled $E_i$ and $F_j$, where $i$ and $j$ reflect the correspondence defined in Definition \ref{defi-corr}.

\begin{table}[hbt]
	\caption{A summary of the following sections,
		where $(\geq)$ (\resp~$(>)$) means there exists an $H_T\times H_S\subseteq\Cris(\pi)$ such that $\rho(\tT,H_T)\geq\rho(\tS,H_S)$ (\resp~$\rho(\tT,H_T)>\rho(\tS,H_S)$).}
	\label{bigtable}
	\small
	\centering

	\begin{tabular}{cccccccc}
		\toprule
		reference              & $\deg(S)$ & lines & $\mathrm{Type}(S)$ & $\deg(T)$ & $\mathrm{Type}(T)$ & $(\geq)$? & $(>)$? \\ \hline
		                    & $4$       & $2$   & $A_3+2A_1$         & $8$       & $A_1$              & NO        & NO     \\
		(\ref{deg4,4A1})       & $4$       & $4$   & $4A_1$             & $8$       & $\bP^1\times\bP^1$ & NO        & NO     \\
		                    & $3$       & $2$   & $A_5+A_1$          & $6$       & $A_2$              & NO        & NO     \\
		                    & $3$       & $3$   & $3A_2$             & $9$       & $\bP^2$            & NO        & NO     \\
		                    & $3$       & $5$   & $A_3+2A_1$         & $6$       & $A_1(4l)$          & NO        & NO     \\
		(\ref{deg3,4A1})       & $3$       & $9$   & $4A_1$             & $6$       & smooth             & YES       & NO     \\
		(\ref{deg2,A7})        & $2$       & $2$   & $A_7$              & $4$       & $A_3(4l)$          & NO        & NO     \\
		                    & $2$       & $2$   & $D_6+A_1$          & $4$       & $D_4$              & NO        & NO     \\
		                    & $2$       & $3$   & $A_5+A_2$          & $6$       & $A_1$              & NO        & NO     \\
		(\ref{deg2,2A3+A1})    & $2$       & $4$   & $2A_3+A_1$         & $4$       & $4A_1$             & YES       & NO     \\
		(\ref{deg2,D4+3A1})    & $2$       & $4$   & $D_4+3A_1$         & $4$       & $A_3+2A_1$         & YES       & NO     \\
		(\ref{deg2,2A3})       & $2$       & $6$   & $2A_3$             & $4$       & $2A_1(8l)$         & YES       & NO     \\
		(\ref{deg2,A5+A1})     & $2$       & $6$   & $A_5+A_1(6l)$      & $4$       & $A_2$              & NO        & NO     \\
		(\ref{deg2,D4+2A1})    & $2$       & $6$   & $D_4+2A_1$         & $4$       & $A_3(4l)$          & YES       & NO     \\
		(\ref{deg2,3A2})       & $2$       & $8$   & $3A_2$             & $6$       & smooth             & YES       & YES    \\
		(\ref{deg2,A3+3A1})    & $2$       & $8$   & $A_3+3A_1$         & $4$       & $3A_1$             & YES       & NO     \\
		(\ref{deg2,6A1})       & $2$       & $10$  & $6A_1$             & $4$       & $4A_1$             & YES       & YES    \\
		                    & $2$       & $12$  & $A_3+2A_1(12l)$    & $4$       & $A_1$              & NO        & NO     \\
		                    & $2$       & $14$  & $5A_1$             & $4$       & $2A_1(8l)$         & YES       & YES    \\
		                    & $2$       & $20$  & $4A_1(20l)$        & $4$       & smooth             & YES       & YES    \\
		                    & $1$       & $3$   & $E_7+A_1$          & $2$       & $E_6$              & NO        & NO     \\
		(\ref{sec:deg1,E6+A2}) & $1$       & $4$   & $E_6+A_2$          & $3$       & $D_4$              & YES       & YES    \\ \bottomrule
	\end{tabular}
\end{table}
%\end{center}

\subsection{Degree $4$, type $4A_1$}
\label{deg4,4A1}
\[
\begin{tikzpicture}[node distance={24pt},
	minus1/.style = {draw, inner sep=2pt},
	minus2/.style = {draw, circle, inner sep=0pt}
	]
	\def\unit{48pt}
	\def\out{2}
	
	\node[] at (-4.8,0)(41){$F_{5,7}$};
	\node[] at (-3.2,0)(42) {$F_{6,8}$};
	\draw (41)--(42);
	\node[] at (-4,-2)(43){$T=\bP_1\times\bP_1$};
	\node[] at (-1.65,0)(40){$\longrightarrow$};
	
	\node[minus2,very thick] at (0,-1)(5b){$E_{4}$};
	\node[minus2,very thick] at (0,1)(3a){$E_{1}$};
	\node[minus2,very thick] at (2,-1)(3b){$E_{3}$};
	\node[minus2,very thick] at (2,1)(5a){$E_{2}$};
	\node[minus1] at (0,0)(16b){$E_{7}$};
	\node[minus1] at (1,-1)(8b){$E_{6}$};
	\node[minus1] at (2,0)(16a){$E_{5}$};
	\node[minus1] at (1,1)(8a){$E_{8}$};
	\draw (3a)--(16b)--(5b)--(8b)--(3b)--(16a)--(5a)--(8a)--(3a);
	
	\node(53)[below=12pt of 8b]{$S:$ degree $4$, type $4A_1$};

\end{tikzpicture}
\]

$$ \Cris(S)\cong\fD_4.$$
There is no $H_\tS$ and $H_\tT$ such that
$\rho(\tT,H_T)\geq\rho(\tS,H_S).$

\subsection{Degree $3$, type $4A_1$}
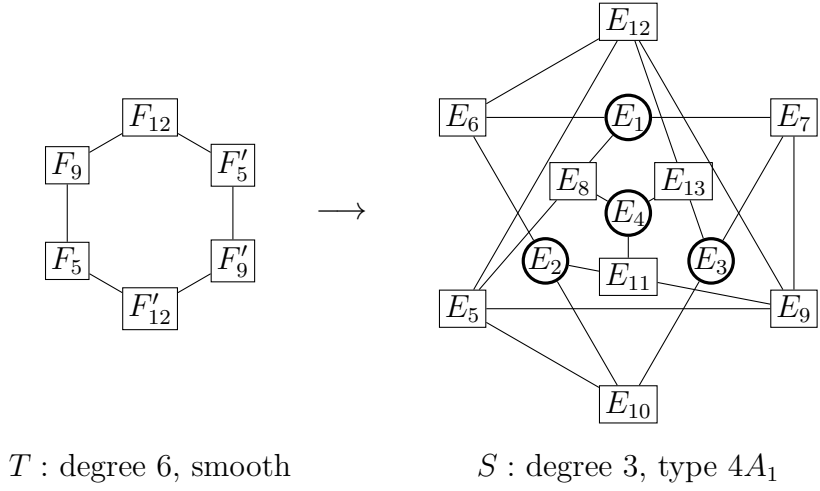
\label{deg3,4A1}
\[
\begin{tikzpicture}[node distance={24pt},
	minus1/.style = {draw, inner sep=2pt},
	minus2/.style = {draw, circle, inner sep=0pt}
	]
	\def\unit{24pt}
	\def\out{3}
	\def\unitb{36pt}
	\node[draw,minus1] (8a) at (90:\unitb) {$F_{12}$};
	\node[draw,minus1] (11a) at (30:\unitb) {$F_{5}^\prime$};
	\node[draw,minus1] (13a) at (-30:\unitb) {$F_{9}^\prime$};
	\node[draw,minus1] (8b) at (-90:\unitb) {$F_{12}^\prime$};
	\node[draw,minus1] (11b) at (-150:\unitb) {$F_{5}$};
	\node[draw,minus1] (13b) at (-210:\unitb) {$F_{9}$};
	\draw (8a)--(11a)--(13a)--(8b)--(11b)--(13b)--(8a);
	\node[] (51) at (-90:\unit*\out+24pt) {$T:$ degree $6$, smooth};
	
	\node[](40) [shift={(72pt,0)}]{$\longrightarrow$};
	
	\def\shift{180pt}
	\node[draw,very thick,minus2,shift={(\shift,0)}] (4) at (0:0) {$E_4$};
	\node[draw,very thick,minus2,shift={(\shift,0)}] (1) at (90:\unit*\out/2) {$E_1$};
	\node[draw,minus1,shift={(\shift,0)}] (7) at (30:\unit) {$E_{13}$};
	\node[draw,very thick,minus2,shift={(\shift,0)}] (3) at (-30:\unit*\out/2) {$E_3$};
	\node[draw,minus1,shift={(\shift,0)}] (6) at (-90:\unit) {$E_{11}$};
	\node[draw,very thick,minus2,shift={(\shift,0)}] (2) at (-150:\unit*\out/2) {$E_2$};
	\node[draw,minus1,shift={(\shift,0)}] (5) at (-210:\unit) {$E_8$};
	
	\node[draw,minus1,shift={(\shift,0)}] (13) at (90:\unit*\out) {$E_{12}$};
	\node[draw,minus1,shift={(\shift,0)}] (10) at (30:\unit*\out) {$E_{7}$};
	\node[draw,minus1,shift={(\shift,0)}] (11) at (-30:\unit*\out) {$E_{9}$};
	\node[draw,minus1,shift={(\shift,0)}] (12) at (-90:\unit*\out) {$E_{10}$};
	\node[draw,minus1,shift={(\shift,0)}] (8) at (-150:\unit*\out) {$E_5$};
	\node[draw,minus1,shift={(\shift,0)}] (9) at (-210:\unit*\out) {$E_6$};
	\node[shift={(\shift,0)}] (52) at (-90:\unit*\out+24pt) {$S:$ degree $3$, type $4A_1$};
	
	\draw (5)--(4)--(6);\draw (4)--(7);
	\draw (3)--(7)--(13)--(9);
	\draw (2)--(6)--(11)--(10);
	\draw (1)--(5)--(8)--(12);
	\draw (9)--(1)--(10)--(3)--(12)--(2)--(9);
	\draw (13)--(8)--(11)--(13);
\end{tikzpicture}
\]
$$ \Cris(S)=\langle \sigma_1,\sigma_2,\sigma_3,\sigma_4 \rangle\cong\fS_4,$$
where
\begin{align*}
	\sigma_1&=(1, 4)(2, 3)(6, 13)(7, 11),\\
	\sigma_2&=(1, 2)(3, 4)(7, 11)(8, 10),\\
	\sigma_3&=(1, 3)(5, 12)(6, 10)(8, 13),\\
	\sigma_4&=(1, 2)(5, 9)(7, 10)(8, 11).
\end{align*}

The subgroups $H_T\times H_S\subseteq \Cris(\pi)$ up to conjugacy satisfying $\rho(\tT,H_T)\geq\rho(\tS,H_S)$ are listed in Table \ref{tab-deg3,4A1}, where $C_2^2(2)$ is the normal subgroup of $\fS_4$ isomorphic to $C_2^2$.
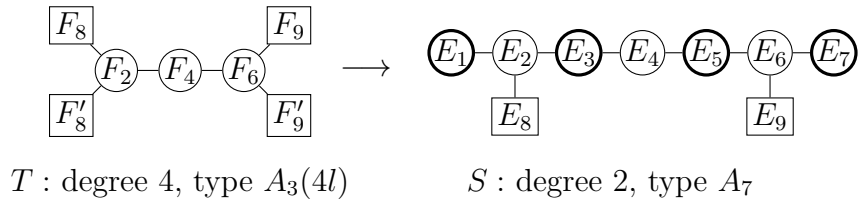
\begin{center}
	\begin{table}[H]
		\caption{Degree $3$, type $4A_1$}
		\label{tab-deg3,4A1}
		\begin{tabular}{cccc}
			\toprule
			$H_\tS$&$H_\tT$&$\rho(\tS,H_S)$&$\rho(\tT,H_T)$\\
			\midrule
			%			$S_3$&$C_2$&$4$&$$\\
			%			$C_3=\langle(1,2,4)\rangle$&$C_2=\langle(9,10)\rangle$&$4$&$$\\
			$\fS_4$&$\fS_3$&$2$&$2$\\
			$\fA_4$&$C_3$&$2$&$2$\\
			$\fD_4$&$C_2$&$3$&$3$\\
			%				$\fS_3$&$C_3$&$3$&$2$\\
			%				$C_2^2$(1)&$C_2$&$4$&$3$\\
			$C_2^2(2)$&$C_1$&$4$&$4$\\
			$C_4$&$C_2$&$3$&$3$\\
			%				$C_3$&$C_3$&$3$&$2$\\
			%				$C_2(1)$&$C_2$&$5$&$3$\\
			%				$C_2(2)$&$0$&$5$&$4$\\
			%				$0$&$0$&$7$&$4$\\
			\bottomrule
		\end{tabular}
	\end{table}
\end{center}

%\section{Degree $2$, type $5A_1$}
%\label{deg2,5A1}

%
%\section{Degree $2$, type $4A_1(20l)$}
%\label{deg2,4A1(20l)}

\subsection{Degree $2$, type $A_7$}
\label{deg2,A7}
\[
\begin{tikzpicture}[node distance={24pt},
	minus1/.style = {draw, inner sep=2pt},
	minus2/.style = {draw, circle, inner sep=0pt}
	]
	\node[minus2] (52){$F_2$};
	\node[minus2] (54) [right of=52] {$F_4$};
	\node[minus2] (56) [right of=54] {$F_6$};
	\node[minus1] (58) [above left of=52] {$F_8$};
	\node[minus1] (58a) [below left of=52] {$F_8^\prime$};
	\node[minus1] (59) [above right of=56] {$F_9$};
	\node[minus1] (59a) [below right of=56] {$F_9^\prime$};
	\draw (58)--(52)--(54)--(56)--(59);
	\draw (58a)--(52);
	\draw (59a)--(56);
	
	\node[](ra) [right=24pt of 56]{$\longrightarrow$};
	
	\node[minus2,very thick] (1) [above right=-6pt and 12pt of ra]{$E_1$};
	\node[minus2] (2) [right of=1] {$E_2$};
	\node[minus2,very thick] (3) [right of=2] {$E_3$};
	\node[minus2] (4) [right of=3] {$E_4$};
	\node[minus2,very thick] (5) [right of=4] {$E_5$};
	\node[minus2] (6) [right of=5] {$E_6$};
	\node[minus2,very thick] (7) [right of=6] {$E_7$};
	\node[minus1] (8) [below of=2] {$E_8$};
	\node[minus1] (9) [below of=6] {$E_9$};
	\draw (1)--(2)--(3)--(4)--(5)--(6)--(7);
	\draw (2)--(8);
	\draw (6)--(9);
	\node[](47) [below=24pt of 54] {$T:$ degree $4$, type $A_3(4l)$};
	\node[](51) [right=36pt of 47] {$S:$ degree $2$, type $A_7$};
\end{tikzpicture}
\]
$$ \Cris(S)\cong C_2.$$
There is no $H_\tS$ and $H_\tT$ such that
$\rho(\tT,H_T)\geq\rho(\tS,H_S).$
%\begin{center}
%	\begin{table}[h]
	%		\begin{tabular}{cccc}
		%			\toprule
		%			$H_\tS$&$H_\tT$&$\rho(\tS,H_S)$&$\rho(\tT,H_T)$\\
		%			\midrule
		%			$C_2=\langle(2,6)\rangle$&$C_2=\langle(2,6)\rangle$&$5$&$4$\\
		%			$0$&$0$&$8$&$6$\\
		%			\bottomrule
		%		\end{tabular}
	%	\end{table}
%\end{center}

\subsection{Degree $2$, type $2A_3+A_1$}
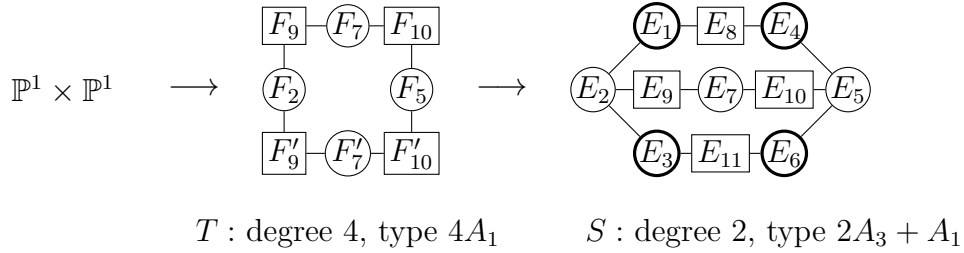
\label{deg2,2A3+A1}
\[
\begin{tikzpicture}[node distance={24pt},
	minus1/.style = {draw, inner sep=2pt},
	minus2/.style = {draw, circle, inner sep=0pt}
	]
	\node[minus2] (12){$F_2$};
	\node[minus1] (19) [above of=12] {$F_9$};
	\node[minus2] (17) [right of=19] {$F_7$};
	\node[minus1] (20) [right of=17] {$F_{10}$};
	\node[minus2] (15) [below of=20] {$F_5$};
	\node[minus1] (29) [below of=12] {$F_9^\prime$};
	\node[minus2] (27) [right of=29] {$F_7^\prime$};
	\node[minus1] (30) [right of=27] {$F_{10}^\prime$};
	\draw (12)--(19)--(17)--(20)--(15)--(30)--(27)--(29)--(12);
	
	\node[](40) [right=12pt of 15]{$\longrightarrow$};
	
	\node[minus2] (2) [right=12pt of 40] {$E_2$}; 
	\node[minus1] (9) [right of=2] {$E_9$}; 
	\node[minus2] (7) [right of=9] {$E_{7}$}; 
	\node[minus1] (8) [above of=7] {$E_8$}; 
	\node[minus2,very thick] (1) [left of=8]{$E_1$};
	\node[minus2,very thick] (4) [right of=8] {$E_4$};
	\node[minus1] (11) [below of=7] {$E_{11}$};
	\node[minus1] (10) [right of=7] {$E_{10}$}; 
	\node[minus2] (5) [right of=10] {$E_{5}$}; 
	\node[minus2,very thick] (3) [left of=11] {$E_3$}; 
	\node[minus2,very thick] (6) [right of=11] {$E_6$}; 
	\draw (1)--(8)--(4)--(5)--(6)--(11)--(3)--(2)--(1);
	\draw (2)--(9)--(7)--(10)--(5);
	
	\node[](40) [left=12pt of 12] {$\longrightarrow$};
	\node[](41) [left=12pt of 40] {$\bP^1\times\bP^1$};

	\node[](47) [below=12pt of 27] {$T:$ degree $4$, type $4A_1$};
	\node[](51) [base right=of 47] {$S:$ degree $2$, type $2A_3+A_1$};
\end{tikzpicture}
\]
$$ \Cris(S)\cong C_2^2.$$

The subgroups $H_T\times H_S\subseteq \Cris(\pi)$ up to conjugacy satisfying $\rho(\tT,H_T)\geq\rho(\tS,H_S)$ are listed in Table \ref{tab-deg2,2A3+A1}.
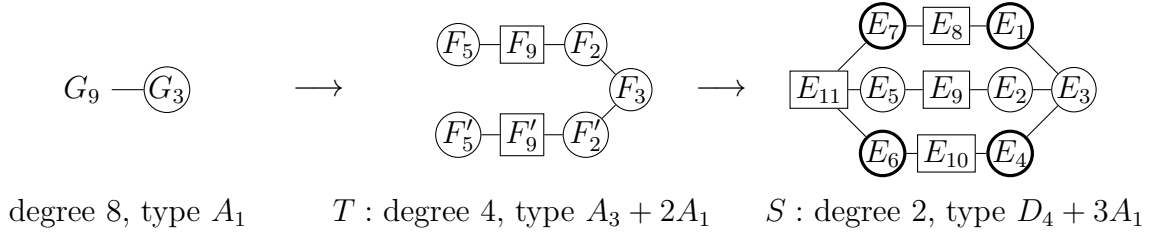
\begin{center}
	\begin{table}[hbt]
		\caption{Degree $2$, type $2A_3+A_1$.}
		\label{tab-deg2,2A3+A1}
		\begin{tabular}{cccc}
			\toprule
			$H_\tS$&$H_\tT$&$\rho(\tS,H_S)$&$\rho(\tT,H_T)$\\
			\midrule
			$C_2^2$&$C_2$&$4$&$4$\\
			%			$C_2=\langle(1,4)\rangle$&$C_2=\langle(9,10)\rangle$&$5$&$4$\\
			%			$C_2=\langle(1,6)\rangle$&$C_2=\langle(9,10)\rangle$&$5$&$4$\\
			$C_2=\langle(1,3)(4,6)(8,11)\rangle$&$C_1$&$6$&$6$\\
			%			$0$&$0$&$8$&$6$\\
			\bottomrule
		\end{tabular}
	\end{table}
\end{center}

\subsection{Degree $2$, type $D_4+3A_1$}
\label{deg2,D4+3A1}
\[
\begin{tikzpicture}[node distance={24pt},
	minus1/.style = {draw, inner sep=2pt},
	minus2/.style = {draw, circle, inner sep=0pt}
	]

	\node[minus2](25){$F_5$};
	\node[minus1](29)[right of=25]{$F_9$};
	\node[minus2](22)[right of=29]{$F_2$};
	\node[minus2](23)[below right of=22]{$F_3$};
	\node[minus2](32)[below left of=23]{$F_2^\prime$};
	\node[minus1](39)[left of=32]{$F_9^\prime$};
	\node[minus2](35)[left of=39]{$F_5^\prime$};
	\draw (25)--(29)--(22)--(23)--(32)--(39)--(35);
	
	\node[](40) [right=12pt of 23]{$\longrightarrow$};
	\node[](41)[left=124pt of 40]{$\longrightarrow$};
	\node[](42)[left=64pt of 41]{$G_9$};
	\node[minus2](44)[right=12pt of 42]{$G_3$};
	
	\draw (42)--(44);
	
	\node[minus1] (2) [right=12pt of 40] {$E_{11}$}; 
	\node[minus2] (9) [right of=2] {$E_5$}; 
	\node[minus1] (7) [right of=9] {$E_{9}$}; 
	\node[minus2] (10) [right of=7] {$E_{2}$}; 
	\node[minus2] (5) [right of=10] {$E_{3}$}; 
	\node[minus1] (8) [above of=7] {$E_8$}; 
	\node[minus2,very thick] (1) [left of=8]{$E_7$};
	\node[minus2,very thick] (4) [right of=8] {$E_1$};
	\node[minus1] (11) [below of=7] {$E_{10}$};
	\node[minus2,very thick] (3) [left of=11] {$E_6$}; 
	\node[minus2,very thick] (6) [right of=11] {$E_4$}; 
	\draw (1)--(8)--(4)--(5)--(6)--(11)--(3)--(2)--(1);
	\draw (2)--(9)--(7)--(10)--(5);
	\node[](47) [below=12pt of 39] {$T:$ degree $4$, type $A_3+2A_1$};
	\node[](51) [right=12pt of 47] {$S:$ degree $2$, type $D_4+3A_1$};
	\node[](43)[left=24pt of 47]{degree 8, type $A_1$};
\end{tikzpicture}
\]
$$ \Cris(S)\cong \fS_3.$$

The subgroups $H_T\times H_S\subseteq \Cris(\pi)$ up to conjugacy satisfying $\rho(\tT,H_T)\geq\rho(\tS,H_S)$ are listed in Table \ref{tab-deg2,D4+3A1}.
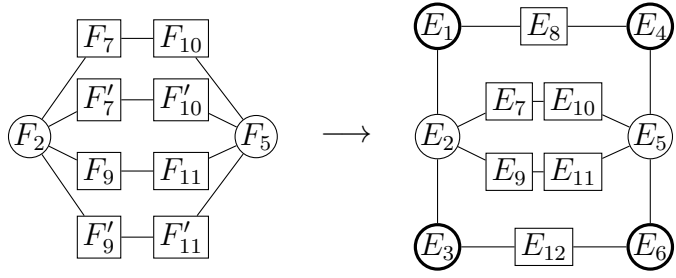
\begin{center}
	\begin{table}[hbt]
		\caption{Degree $2$, type $D_4+3A_1$.}
		\label{tab-deg2,D4+3A1}
		\begin{tabular}{cccc}
			\toprule
			$H_\tS$&$H_\tT$&$\rho(\tS,H_S)$&$\rho(\tT,H_T)$\\
			\midrule
			%			$S_3$&$C_2$&$4$&$$\\
			%			$C_3=\langle(1,2,4)\rangle$&$C_2=\langle(9,10)\rangle$&$4$&$$\\
			$C_2$&$C_1$&$6$&$6$\\
			%			$0$&$0$&$8$&$6$\\
			\bottomrule
		\end{tabular}
	\end{table}
\end{center}

\subsection{Degree $2$, type $2A_3$}
\label{deg2,2A3}
\[
\begin{tikzpicture}[node distance={24pt},
	minus1/.style = {draw, inner sep=2pt},
	minus2/.style = {draw, circle, inner sep=0pt}
	]
	
	\node[minus2](2a){$F_2$};
	\node[minus1](7b)[above right=0pt and 12pt of 2a]{$F_7^\prime$};
	\node[minus1](9a)[below right=0pt and 12pt of 2a]{$F_9$};
	\node[minus1](7a)[above right=24pt and 12pt of 2a]{$F_7$};
	\node[minus1](9b)[below right=24pt and 12pt of 2a]{$F_9^\prime$};
	\node[minus1](10b)[right=0pt and 12pt of 7b]{$F_{10}^\prime$};
	\node[minus1](11a)[right=0pt and 12pt of 9a]{$F_{11}$};
	\node[minus1](10a)[right=24pt and 12pt of 7a]{$F_{10}$};
	\node[minus1](11b)[right=24pt and 12pt of 9b]{$F_{11}^\prime$};
	\node[minus2](5a)[below right=0pt and 12pt of 10b]{$F_5$};
	\draw (2a)--(7a)--(10a)--(5a);
	\draw (2a)--(7b)--(10b)--(5a);
	\draw (2a)--(9a)--(11a)--(5a);
	\draw (2a)--(9b)--(11b)--(5a);
	
	\node[](40) [right=12pt of 5a]{$\longrightarrow$};
	
	\node[minus2](2)[right=12pt of 40]{$E_2$};
	\node[minus1](7)[above right=0pt and 12pt of 2]{$E_7$};
	\node[minus1](10)[right of=7]{$E_{10}$};
	\node[minus2](5)[below right=0pt and 12pt of 10]{$E_5$};
	\node[minus1](9)[below right=0pt and 12pt of 2]{$E_9$};
	\node[minus1](11)[right of=9]{$E_{11}$};
	\node[minus2,very thick](1)[above=24pt of 2]{$E_1$};
	\node[minus2,very thick](3)[below=24pt  of 2]{$E_3$};
	\node[minus2,very thick](4)[above=24pt  of 5]{$E_4$};
	\node[minus2,very thick](6)[below=24pt  of 5]{$E_6$};
	\node[minus1] (8) at ($(1)!0.5!(4)$) {$E_8$};
	\node[minus1] (12) at ($(3)!0.5!(6)$) {$E_{12}$};
	\draw (1)--(2)--(3)--(12)--(6)--(5)--(4)--(8)--(1);
	\draw (2)--(7)--(10)--(5)--(11)--(9)--(2);
	\node[](51) [below=12pt of 12] {$S:$ degree $2$, type $2A_3$};
	\node[](47) [left=24pt of 51)] {$T:$ degree $4$, type $2A_1(8l)$};
\end{tikzpicture}
\]
$$ \Cris(S)=\langle a,b,c\rangle\cong C_2^3,$$
where
\begin{align*}
	a&=(7, 9)(10, 11),\\
	b&=(1, 4)(2, 5)(3, 6)(7, 10)(9, 11),\\
	c&=(1, 3)(4, 6)(8, 12).
\end{align*}

The subgroups $H_T\times H_S\subseteq \Cris(\pi)$ up to conjugacy satisfying $\rho(\tT,H_T)\geq\rho(\tS,H_S)$ are listed in Table \ref{tab-deg2, 2A3}.
%The group $C_2^3$ has $7$ subgroups isomorphic to $C_2^2$ and $7$ subgroups isomorphic to $C_2$, and all of them are not conjugate.
\begin{center}
	\begin{table}[h]
		\caption{Degree $2$, type $2A_3$.}
		\label{tab-deg2, 2A3}
		%		\footnotesize
		\begin{tabular}{cccc}
			\toprule
			$H_\tS$&$H_\tT$&$\rho(\tS,H_S)$&$\rho(\tT,H_T)$\\
			\midrule
			%			$C_2^3$&$C_2^2$&$3$&$2$\\
			%			$C_2^2=\langle(1, 4)(2, 5)(3, 6)(7, 10)(9, 11),
			%			(1, 3)(4, 6)(8, 12)\rangle$&$-$&$3$&$2$\\
			%			$C_2^2=\langle (1, 3)(4, 6)(7, 9)(8, 12)(10, 11),
			%			(1, 6)(2, 5)(3, 4)(7, 10)(8, 12)(9, 11) \rangle$&$-$&$3$&$2$\\
			%			$C_2^2=\langle (1, 3)(4, 6)(7, 9)(8, 12)(10, 11),
			%			(1, 4)(2, 5)(3, 6)(7, 10)(9, 11) \rangle$&$-$&$3$&$2$\\
			%			$C_2^2=\langle (1, 4)(2, 5)(3, 6)(7, 11)(9, 10),
			%			(1, 3)(4, 6)(8, 12) \rangle$&$-$&$4$&$2$\\
			%			$C_2^2=\langle (7, 9)(10, 11),
			%			(1, 6)(2, 5)(3, 4)(7, 10)(8, 12)(9, 11)) \rangle$&$-$&$4$&$2$\\
			%			$C_2^2=\langle (7, 9)(10, 11),
			%			(1, 4)(2, 5)(3, 6)(7, 10)(9, 11) \rangle$&$-$&$4$&$2$\\
			%			$C_2^2=\langle (7, 9)(10, 11),
			%			(1, 3)(4, 6)(8, 12) \rangle$&$-$&$5$&$4$\\
			%			
			%			$C_2=\langle(7, 9)(10, 11)\rangle$&$-$&$7$&$4$\\
			%			$C_2=\langle(1, 6)(2, 5)(3, 4)(7, 10)(8, 12)(9, 11)\rangle$&$-$&$4$&$2$\\
			%			$C_2=\langle(1, 4)(2, 5)(3, 6)(7, 10)(9, 11)\rangle$&$-$&$4$&$2$\\
			%			$C_2=\langle(1, 6)(2, 5)(3, 4)(7, 11)(8, 12)(9, 10)\rangle$&$-$&$5$&$2$\\
			%			$C_2=\langle(1, 3)(4, 6)(7, 9)(8, 12)(10, 11)\rangle$&$-$&$5$&$4$\\
			%			$C_2=\langle(1, 4)(2, 5)(3, 6)(7, 11)(9, 10)\rangle$&$-$&$5$&$2$\\
			$C_2=\langle c\rangle$&$C_1$&$6$&$6$\\
			%			$0$&$0$&$8$&$6$\\
			\bottomrule
		\end{tabular}
	\end{table}
\end{center}

\subsection{Degree $2$, type $A_5+A_1(6l)$}
\label{deg2,A5+A1}
\[
\begin{tikzpicture}[node distance={24pt},
	minus1/.style = {draw, inner sep=2pt},
	minus2/.style = {draw, circle, inner sep=0pt}
	]
	\def\unit{1.3}
	\def\shift{-150pt}
	
	\node[minus2,shift={(\shift,0)}] at (-\unit,\unit)(2){$F_2$};
	\node[minus2,shift={(\shift,0)}] at (\unit,\unit)(4){$F_4$};
	\node[minus1,shift={(\shift,0)}] at (-\unit,0)(1){$F_{10}^\prime$};
	\node[minus1,shift={(\shift,0)}] at (\unit,0)(5){$F_{11}$};
	\node[minus1,shift={(\shift,0)}] at (-\unit,-\unit)(9){$F_8^\prime$};
	\node[minus1,shift={(\shift,0)}] at (\unit,-\unit)(12){$F_{7}$};
	\node[minus1,shift={(\shift,0)}] at (-0.5,0.5)(10){$F_{10}$};
	\node[minus1,shift={(\shift,0)}] at (0.5,0.5)(11){$F_{11}^\prime$};
	\node[minus1,shift={(\shift,0)}] at (-0.5,-0.5)(7){$F_{7}^\prime$};
	\node[minus1,shift={(\shift,0)}] at (0.5,-0.5)(8){$F_{8}$};
	\draw (2)--(4)--(5)--(12)--(9)--(1)--(2);
	\draw (2)--(10)--(8)--(12);
	\draw (4)--(11)--(7)--(9);
	\draw (7)--(8);
	
	\node[] at (-2.5,0){$\longrightarrow$};
	
	\node[minus2] at (-\unit,\unit)(2){$E_2$};
	\node[minus2,very thick] at (0,\unit)(3){$E_3$};
	\node[minus2] at (\unit,\unit)(4){$E_4$};
	\node[minus2,very thick] at (-\unit,0)(1){$E_1$};
	\node[minus2,very thick] at (\unit,0)(5){$E_5$};
	\node[minus1] at (-\unit,-\unit)(9){$E_9$};
	\node[minus2,very thick] at (0,-\unit)(6){$E_6$};
	\node[minus1] at (\unit,-\unit)(12){$E_{12}$};
	
	\node[minus1] at (-0.5,0.5)(10){$E_{10}$};
	\node[minus1] at (0.5,0.5)(11){$E_{11}$};
	\node[minus1] at (-0.5,-0.5)(7){$E_{7}$};
	\node[minus1] at (0.5,-0.5)(8){$E_{8}$};
	\draw (2)--(3)--(4)--(5)--(12)--(6)--(9)--(1)--(2);
	\draw (2)--(10)--(8)--(12);
	\draw (4)--(11)--(7)--(9);
	\draw (7)--(8);
	
	\node[](51) [below=12pt of 6] {$S:$ degree $2$, type $A_5+A_1(6l)$};
	\node[](47) [left=12pt of 51] {$T:$ degree $4$, type $A_2$};
\end{tikzpicture}
\]
$$ \Cris(S)\cong C_2.$$
There is no $H_\tS$ and $H_\tT$ such that
$\rho(\tT,H_T)\geq\rho(\tS,H_S).$

\subsection{Degree $2$, type $D_4+2A_1$}
\label{deg2,D4+2A1}
\[
\begin{tikzpicture}[node distance={24pt},
	minus1/.style = {draw, inner sep=2pt},
	minus2/.style = {draw, circle, inner sep=0pt}
	]
	
	\node[minus2] at (-4.3,1.5)(52){$F_1$};
	\node[minus2] (54) [right of=52] {$F_3$};
	\node[minus2] (56) [right of=54] {$F_1^\prime$};
	\node[minus1] (58) [above left of=52] {$F_7$};
	\node[minus1] (58a) [below left of=52] {$F_{11}$};
	\node[minus1] (59) [above right of=56] {$F_7^\prime$};
	\node[minus1] (59a) [below right of=56] {$F_{11}^\prime$};
	\draw (58)--(52)--(54)--(56)--(59);
	\draw (58a)--(52);
	\draw (59a)--(56);
	
	\node[] at (-1,1.5){$\longrightarrow$};
	
	\node[minus2] at (1,1.5)(1){$E_1$};
	\node[minus2,very thick] at (1,0)(2){$E_2$};
	\node[minus2] at (0,1.5)(3){$E_3$};
	\node[minus2,very thick] at (1,3)(4){$E_4$};
	\node[minus2,very thick] at (3,0)(5){$E_5$};
	\node[minus2,very thick] at (4,3)(6){$E_6$};
	\node[minus1] at (2,2)(7){$E_7$};
	\node[minus1] at (3,2)(8){$E_8$};
	\node[minus1] at (2,0)(9){$E_9$};
	\node[minus1] at (2,3)(10){$E_{10}$};
	\node[minus1] at (2,1)(11){$E_{11}$};
	\node[minus1] at (4,1)(12){$E_{12}$};
	\draw (2)--(3)--(4)--(10)--(6)--(12)--(5)--(9)--(2);
	\draw (3)--(1)--(7)--(8)--(6);
	\draw (1)--(11)--(12);
	\draw (8)--(5);
	
	\node[](51) [below=12pt of 9] {$S:$ degree $2$, type $D_4+2A_1$};
	\node[](47) [left=12pt of 51] {$T:$ degree $4$, type $A_3(4l)$};
\end{tikzpicture}
\]
$$ \Cris(S)=\langle a,
b\rangle\cong C_2^2,$$
where
\begin{align*}
	a&=(7, 11)(8, 12),\\
	b&=(2, 4)(5, 6)(9, 10).
\end{align*}

The subgroups $H_T\times H_S\subseteq \Cris(\pi)$ up to conjugacy satisfying $\rho(\tT,H_T)\geq\rho(\tS,H_S)$ are listed in Table \ref{tab-deg2, D4+2A1}.
\begin{center}
	\begin{table}[hbt]
		\caption{Degree $2$, type $D_4+2A_1$}
		\label{tab-deg2, D4+2A1}.
		\begin{tabular}{cccc}
			\toprule
			$H_\tS$&$H_\tT$&$\rho(\tS,H_S)$&$\rho(\tT,H_T)$\\
			\midrule
			%			$C_2^2$&$C_2$&$5$&$4$\\
			$C_2=\langle b\rangle$&$C_1$&$6$&$6$\\
			%			$C_2(2)=\langle(2, 4)(5, 6)(7, 11)(8, 12)(9, 10)\rangle$&$C_2$&$5$&$4$\\
			%			$C_2(3)=\langle(7, 11)(8, 12)\rangle$&$C_2$&$7$&$4$\\
			%			$0$&$0$&$8$&$6$\\
			\bottomrule
		\end{tabular}
	\end{table}
\end{center}

\subsection{Degree $2$, type $3A_2$}
\label{deg2,3A2}
\[
\begin{tikzpicture}[node distance={24pt},
	minus1/.style = {draw, inner sep=2pt},
	minus2/.style = {draw, circle, inner sep=0pt}
	]
	
	\def\unita{90pt}
	\def\out{3}
	\def\unitb{36pt}
	\def\unitc{30pt}
	\def\shift{160pt}
	\node[draw,minus1] (8a) at (90:\unitb) {$F_{9}$};
	\node[draw,minus1] (11a) at (30:\unitb) {$F_{10}$};
	\node[draw,minus1] (13a) at (-30:\unitb) {$F_{9}^{\prime}$};
	\node[draw,minus1] (8b) at (-90:\unitb) {$F_{10}^{\prime}$};
	\node[draw,minus1] (11b) at (-150:\unitb) {$F_{9}^{\prime\prime}$};
	\node[draw,minus1] (13b) at (-210:\unitb) {$F_{10}^{\prime\prime}$};
	\draw (8a)--(11a)--(13a)--(8b)--(11b)--(13b)--(8a);

	\node[](40) [shift={(72pt,0)}]{$\longrightarrow$};
	
	\node[minus1,shift={(\shift,0)}] (9) at (0,0) {$E_9$};
	\node[minus1,shift={(\shift,0)}] (13) at (90:\unitc) {$E_{13}$};
	\node[minus1,shift={(\shift,0)}] (12) at (-30:\unitc) {$E_{12}$};
	\node[minus1,shift={(\shift,0)}] (7) at (-150:\unitc) {$E_{7}$};
	\node[minus1,shift={(\shift,0)}] (14) at (30:\unita) {$E_{14}$};
	\node[minus1,shift={(\shift,0)}] (11) at (-90:\unita) {$E_{11}$};
	\node[minus1,shift={(\shift,0)}] (8) at (-210:\unita) {$E_{8}$};
	\node[minus2, very thick] (1) at ($(8)!1/3!(14)$) {$E_1$};
	\node[minus2, very thick] (2) at ($(8)!2/3!(14)$) {$E_2$};
	\node[minus2, very thick] (3) at ($(8)!1/3!(11)$) {$E_3$};
	\node[minus2, very thick] (4) at ($(8)!2/3!(11)$) {$E_4$};
	\node[minus2, very thick] (5) at ($(11)!1/3!(14)$) {$E_5$};
	\node[minus2, very thick] (6) at ($(11)!2/3!(14)$) {$E_6$};
	\node[minus1,shift={(\shift,0)}] (10) at (-210:56pt) {$E_{10}$};
	
	\draw[double, double distance=2pt] (10)--(9);
	\draw (13)--(9)--(7);
	\draw (9)--(12);
	\draw (8)--(1)--(2)--(14)--(6)--(5)--(11)--(4)--(3)--(8);
	\draw (1)--(7)--(5);
	\draw (2)--(12)--(4);
	\draw (3)--(13)--(6);
	\draw (11)edge[out=180,in=225](10);
	\draw (10)edge[out=75,in=120](14);
	\draw (10)--(8);
	
	\node[](51) [below=12pt of 11] {$S:$ degree $2$, type $3A_2$};
	\node[](47) [left=24pt of 51)] {$T:$ degree $6$, smooth};
\end{tikzpicture}
\]

$$ \Cris(S)\cong\fD_6.$$
%$$ \Cris(S)=\langle a,b,c \rangle\cong\fD_6$$
%where
%\begin{align*}
%	a&=(1, 2)(3, 6)(4, 5)(7, 12)(8, 14),\\
%	b&=(1, 6, 4)(2, 5, 3)(7, 13, 12)(8, 14, 11),\\
%	c&=(3, 5)(4, 6)(7, 8)(9, 10)(11, 13)(12, 14).
%\end{align*}

The subgroups $H_T\times H_S\subseteq \Cris(\pi)$ up to conjugacy satisfying $\rho(\tT,H_T)\geq\rho(\tS,H_S)$ are listed in Table \ref{tab-deg2,3A2}.
\begin{center}
	\begin{table}[hbpt]
		
		\centering
		\caption{Degree $2$, type $3A_2$.}
		\label{tab-deg2,3A2}
		\begin{tabular}{cccc}
			\toprule
			$H_\tS$&$H_\tT$&$\rho(\tS,H_S)$&$\rho(\tT,H_T)$\\
			\midrule
			$\fD_6$&$C_2$&$2$&$2$\\
			$\fS_3(1)$&$C_1$&$\bf 3$&$\bf 6$\\
			%			$\fS_3^\prime=\langle (3, 5)(4, 6)(7, 8)(9, 10)(11, 13)(12, 14),
			%			(1, 6, 4)(2, 5, 3)(7, 13, 12)(8, 14, 11) \rangle$&$C_3$&$3$&$2$\\
			$C_6$&$C_2$&$2$&$2$\\
			%			$C_2^2=\langle (1, 3)(2, 4)(5, 6)(7, 13)(11, 14),
			%			(1, 4)(2, 3)(7, 11)(8, 12)(9, 10)(13, 14)\rangle$&$C_3$&$3$&$2$\\
			$C_3$&$C_1$&$\bf 3$&$\bf 6$\\
			$C_2(1)$&$C_1$&$\bf 5$&$\bf 6$\\
			%			$C_2^\prime=\langle(1, 4)(2, 3)(7, 11)(8, 12)(9, 10)(13, 14)\rangle$&$C_3$&$5$&$2$\\
			%			$C_2^{\prime\prime}=\langle(1, 2)(3, 4)(5, 6)(7, 14)(8, 12)(9, 10)(11, 13)\rangle$&$C_3$&$4$&$2$\\
			%			$0$&$0$&$8$&$6$\\
			\bottomrule
		\end{tabular}
	\end{table}
\end{center}

In Table \ref{tab-deg2,3A2}, we define the subgroups by
\begin{align*}
	&\fS_3(1)=\langle(1, 3)(2, 4)(5, 6)(7, 13)(11, 14),
	(1, 6, 4)(2, 5, 3)(7, 13, 12)(8, 14, 11)\rangle,\\
	%	&C_6=\langle (1, 2)(3, 4)(5, 6)(7, 14)(8, 12)(9, 10)(11, 13),
	%	(1, 6, 4)(2, 5, 3)(7, 13, 12)(8, 14, 11) \rangle,\\
	%	&C_3=\langle  (1, 6, 4)(2, 5, 3)(7, 13, 12)(8, 14, 11)\rangle\\
	&C_2(1)=\langle(1, 3)(2, 4)(5, 6)(7, 13)(11, 14)\rangle.
\end{align*}

\subsection{Degree $2$, type $A_3+3A_1$}
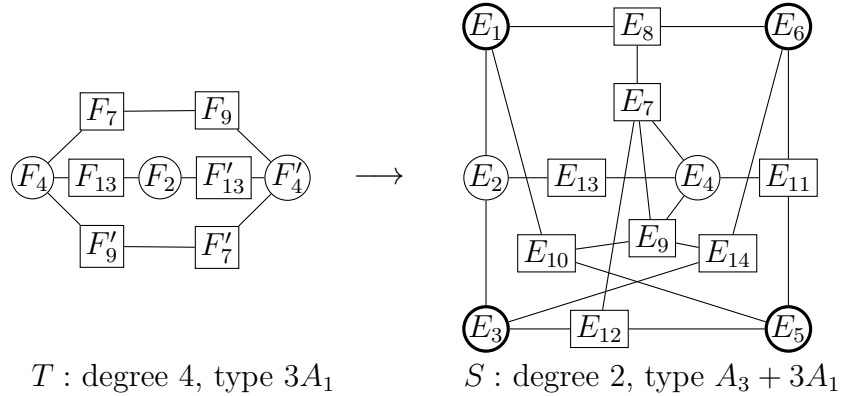
\label{deg2,A3+3A1}

\[
\begin{tikzpicture}[node distance={24pt},
	minus1/.style = {draw, inner sep=2pt},
	minus2/.style = {draw, circle, inner sep=0pt}
	]
	\node[minus2] (4a) at (-6,-2) {$F_4$}; 
	\node[minus1] (13a) [right of=4a] {$F_{13}$}; 
	\node[minus2] (2a) [right of=13a] {$F_{2}$}; 
	\node[minus1] (13b) [right of=2a] {$F_{13}^\prime$}; 
	\node[minus2] (4b) [right of=13b]{$F_4^\prime$};
	\node[minus1] (7a) [above right=12pt and 12pt of 4a] {$F_7$}; 
	\node[minus1] (7b) [below right=12pt and 12pt of 4a]{$F_9^\prime$};
	\node[minus1] (9b) [above left=12pt and 12pt of 4b] {$F_9$}; 
	\node[minus1] (9a) [below left=12pt and 12pt of 4b]{$F_7^\prime$};
	\draw (4a)--(13a)--(2a)--(13b)--(4b);
	\draw (4a)--(7a)--(9b)--(4b)--(9a)--(7b)--(4a);
	
	\node[](40) [right=12pt of 4b]{$\longrightarrow$};
	
	\node[minus2, very thick] at (0,0)(1){$E_1$};
	\node[minus2] at (0,-2)(2){$E_2$};
	\node[minus2, very thick] at (0,-4)(3){$E_3$};
	\node[minus1] at (1.2,-2)(13){$E_{13}$};
	\node[minus1] at (0.8,-3)(10){$E_{10}$};
	\node[minus1] at (1.5,-4)(12){$E_{12}$};
	\node[minus1] at (2,0)(8){$E_{8}$};
	\node[minus1] at (2,-1)(7){$E_{7}$};
	\node[minus1] at (2.2,-2.8)(9){$E_{9}$};
	\node[minus2] at (2.8,-2)(4){$E_{4}$};
	\node[minus1] at (3.2,-3)(14){$E_{14}$};
	\node[minus2, very thick] at (4,0)(6){$E_6$};
	\node[minus1] at (4,-2)(11){$E_{11}$};
	\node[minus2, very thick] at (4,-4)(5){$E_5$};
	\draw (1)--(2)--(3)--(12)--(5)--(11)--(6)--(8)--(1);
	\draw (1)--(10)--(5);
	\draw (3)--(14)--(6);
	\draw (10)--(9)--(14);
	\draw (8)--(7)--(4)--(9)--(7);
	\draw (2)--(13)--(4)--(11);
	\draw (12)--(7);
	\node[](51) [below=36pt of 9] {$S:$ degree $2$, type $A_3+3A_1$};
	\node[](47) [left=40pt of 51] {$T:$ degree $4$, type $3A_1$};
\end{tikzpicture}
\]
$$ \Cris(S)=\langle a,b\rangle\cong C_2^2,$$
where
\begin{align*}
	a&=(1, 3)(7, 9)(8, 14)(10, 12),\\
	b&=(5, 6)(7, 9)(8, 10)(12, 14).
\end{align*}
The subgroups $H_T\times H_S\subseteq \Cris(\pi)$ up to conjugacy satisfying $\rho(\tT,H_T)\geq\rho(\tS,H_S)$ are listed in Table \ref{tab-deg2,A3+3A1}.
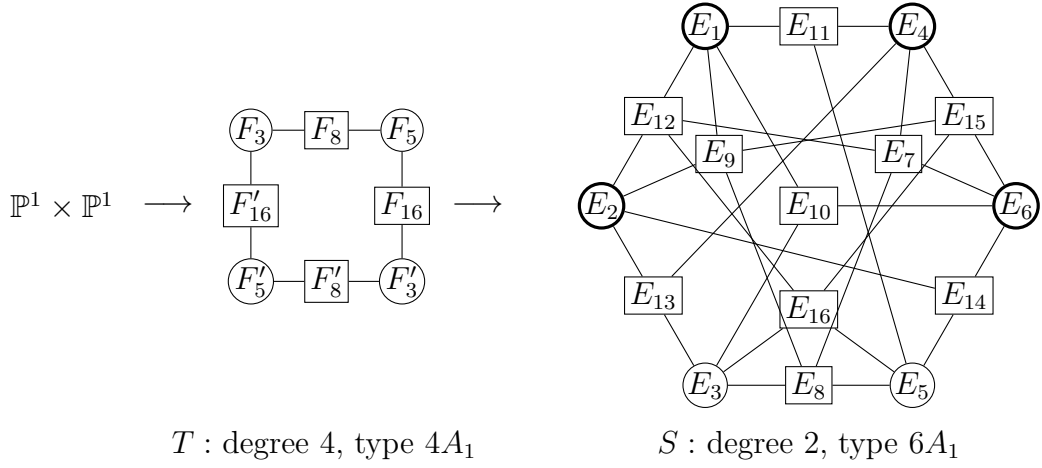
\begin{center}
	\begin{table}[h]
		\caption{Degree $2$, type $A_3+3A_1$.}
		\label{tab-deg2,A3+3A1}
		\begin{tabular}{cccc}
			\toprule
			$H_\tS$&$H_\tT$&$\rho(\tS,H_S)$&$\rho(\tT,H_T)$\\
			\midrule
			$C_2^2$&$C_2=\langle(7,9^\prime)(7^\prime,9)\rangle$&$5$&$5$\\
			%			$C_2=\langle(1, 3)(7, 9)(8, 14)(10, 12)\rangle$&$C_2$&$6$&$4$\\
			$C_2=\langle ab\rangle$&$C_1$&$6$&$6$\\
			%			$C_2=\langle(5, 6)(7, 9)(8, 10)(12, 14)\rangle$&$C_2$& $6$& $4$\\
			%			$0$&$0$&$8$&$6$\\
			\bottomrule
		\end{tabular}
	\end{table}
\end{center}

\subsection{Degree $2$, type $6A_1$}
\label{deg2,6A1}
\[
\begin{tikzpicture}[node distance={24pt},
	minus1/.style = {draw, inner sep=2pt},
	minus2/.style = {draw, circle, inner sep=0pt}
	]
	\def\unit{39pt}
	\def\out{2}
	
	\node[] at (-2.5,0)(41) {$\bP^1\times\bP^1$};
	
	\node[] at (-1.1,0)(40){$\longrightarrow$};
	
	\node[minus2] at (0,-1)(5b){$F_{5}^\prime$};
	\node[minus2] at (0,1)(3a){$F_{3}$};
	\node[minus2] at (2,-1)(3b){$F_{3}^\prime$};
	\node[minus2] at (2,1)(5a){$F_{5}$};
	\node[minus1] at (0,0)(16b){$F_{16}^\prime$};
	\node[minus1] at (1,-1)(8b){$F_{8}^\prime$};
	\node[minus1] at (2,0)(16a){$F_{16}$};
	\node[minus1] at (1,1)(8a){$F_{8}$};
	\draw (3a)--(16b)--(5b)--(8b)--(3b)--(16a)--(5a)--(8a)--(3a);

	\node[] at (3,0)(40){$\longrightarrow$};
	
	\def\shift{210pt}
	\node[draw,minus1,shift={(\shift,0)}] (10) at (0:0) {$E_{10}$};
	\node[draw,minus1,shift={(\shift,0)}] (9) at (-210:\unit) {$E_9$};
	\node[draw,minus1,shift={(\shift,0)}] (16) at (-90:\unit) {$E_{16}$};
	\node[draw,minus1,shift={(\shift,0)}] (7) at (30:\unit) {$E_{7}$};
	
	\node[draw,very thick,minus2,shift={(\shift,0)}] (1) at (-240:\unit*\out) {$ E_1$};
	\node[draw,very thick,minus2,shift={(\shift,0)}] (2) at (-180:\unit*\out) {$E_2$};
	\node[draw,minus2,shift={(\shift,0)}] (3) at (-120:\unit*\out) {$E_{3}$};
	\node[draw,minus2,shift={(\shift,0)}] (5) at (-60:\unit*\out) {$E_{5}$};
	\node[draw,very thick,minus2,shift={(\shift,0)}] (6) at (0:\unit*\out) {$E_{6}$};
	\node[draw,very thick,minus2,shift={(\shift,0)}] (4) at (60:\unit*\out) {$E_{4}$};
	
	\node[minus1] (12) at ($(1)!0.5!(2)$) {$E_{12}$};
	\node[minus1] (13) at ($(2)!0.5!(3)$) {$E_{13}$};
	\node[minus1] (8) at ($(3)!0.5!(5)$) {$E_{8}$};
	\node[minus1] (14) at ($(5)!0.5!(6)$) {$E_{14}$};
	\node[minus1] (15) at ($(6)!0.5!(4)$) {$E_{15}$};
	\node[minus1] (11) at ($(4)!0.5!(1)$) {$E_{11}$};
	
	\draw (1)--(12)--(2)--(13)--(3)--(8)--(5)--(14)--(6)--(15)--(4)--(11)--(1);
	\draw (10)--(1);\draw (10)--(3);\draw (10)--(6);
	\draw (2)--(9)--(1);\draw (8)--(9)--(15);
	\draw (12)--(7)--(4);\draw (8)--(7)--(6);
	\draw (12)--(16)--(3);\draw (5)--(16)--(15);
	\draw (2)--(14);\draw (5)--(11);\draw (4)--(13);

	\node[shift={(\shift,0)}] (52) at (-90:\unit*\out+12pt) {$S:$ degree $2$, type $6A_1$};
	\node(53)[left=60pt of 52]{$T:$ degree $4$, type $4A_1$};

\end{tikzpicture}
\]

$$ \Cris(S)\cong C_2\times\fS_4.$$

The subgroups $H_T\times H_S\subseteq \Cris(\pi)$ up to conjugacy satisfying $\rho(\tT,H_T)\geq\rho(\tS,H_S)$ are listed in Table \ref{tab-deg2, 6A1}.

\begin{center}
	\begin{table}[hbt]
		\caption{Degree $2$, type $6A_1$}
		\label{tab-deg2, 6A1}
		\begin{tabular}{cccc}
			\toprule
			$H_\tS$&$H_\tT$&$\rho(\tS,H_S)$&$\rho(\tT,H_T)$\\
			\midrule
			%			$C_2\times\fD_4$&$(4,4)$&$3$&$2$\\
			%			$C_2\times C_4$&$(4,4)$&$3$&$2$\\
			%			$\fD_4$&$(4,4)$&$3$&$2$\\
			$\fD_4$&$C_2$&$4$&$4$\\
			%			$C_2^3$&$(4,4)$&$4$&$2$\\
			$C_2^3$&$C_2$&$4$&$4$\\
			$C_4$&$C_2$&$4$&$4$\\
			%			$C_4$&$(4,4)$&$3$&$2$\\
			$C_2^2(1)$&$C_1$&$\bf 5$&$\bf 6$\\
			%			$C_2^2$&$(2,2,2,2)$&$5$&$4$\\
			%			$C_2^2$&$(4,4)$&$5$&$2$\\
			$C_2^2(2)$&$C_2$&$4$&$4$\\
			$C_2^2(3)$&$C_2$&$4$&$4$\\
			%			$C_2$&$(2,2,2,2)$&$7$&$4$\\
			$C_2$&$C_1$&$6$&$6$\\
			%			$C_2$&$(2,2,2,2)$&$6$&$4$\\
			%			$C_2$&$(2,2,2,2)$&$5$&$4$\\
			%			$C_2$&$(4,4)$&$5$&$2$\\
			%			$0$&$0$&$8$&$6$\\
			\bottomrule
		\end{tabular}
	\end{table}
\end{center}
%In Table \ref{tab-deg2, 6A1}, we have
%$$C_2^2(1)=\langle(1, 2)(4, 6)(10, 13)(11, 14),
%		(1, 6)(2, 4)(7, 9)(11, 14)(12, 15)
%		\rangle.$$
In Table \ref{tab-deg2, 6A1}, we define the subgroups by
\begin{small}
	\begin{align*}
		%		&\fD_4=\langle(1, 2)(3, 5)(10, 14)(11, 13),
		%		(1, 6)(2, 4)(7, 9)(11, 14)(12, 15),
		%		(3, 5)(4, 6)(10, 11)(13, 14)\rangle\\
		%		&C_2^3=\langle(1, 2)(4, 6)(10, 13)(11, 14),
		%		(1, 6)(2, 4)(7, 9)(11, 14)(12, 15),
		%		(7, 15)(8, 16)(9, 12)
		%		\rangle\\
		%		&C_4=\langle (1, 6, 2, 4)(3, 5)(7, 9)(10, 14, 13, 11)(12, 15)\rangle\\
		&C_2^2(1)=\langle(1, 2)(4, 6)(10, 13)(11, 14),
		(1, 6)(2, 4)(7, 9)(11, 14)(12, 15)
		\rangle.\\
		&C_2^2(2)=\langle(1, 6)(2, 4)(7, 12)(8, 16)(9, 15)(11, 14),
		(1, 2)(4, 6)(10, 13)(11, 14)
		\rangle,\\
		&C_2^2(3)=\langle(1, 2)(4, 6)(7, 15)(8, 16)(9, 12)(10, 13)(11, 14),
		(1, 6)(2, 4)(7, 9)(11, 14)(12, 15)
		\rangle.
		%		&C_2=\langle(1, 2)(4, 6)(10, 13)(11, 14), (1, 6)(2, 4)(7, 9)(11, 14)(12, 15)\rangle
	\end{align*}
\end{small}

\subsection{Degree $1$, type $E_6+A_2$}
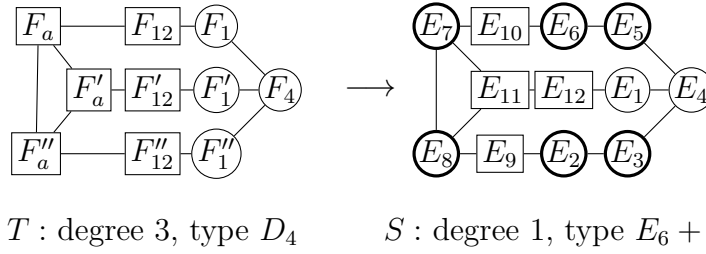
\label{sec:deg1,E6+A2}
\[
\begin{tikzpicture}[node distance={24pt},
	minus1/.style = {draw, inner sep=2pt},
	minus2/.style = {draw, circle, inner sep=0pt}
	]
	\node[minus1] (12a)  {$F_{12}$}; 
	\node[minus1] (12b) [below of=12a] {$F_{12}^\prime$}; 
	\node[minus1] (12c) [below of=12b] {$F_{12}^{\prime\prime}$};
	\node[minus2] (1a) [right of=12a] {$F_{1}$}; 
	\node[minus2] (1b) [right of=12b] {$F_{1}^\prime$}; 
	\node[minus2] (1c) [right of=12c] {$F_{1}^{\prime\prime}$}; 
	\node[minus2] (4a) [right of=1b] {$F_{4}$}; 
	\node[minus1] (11b) [left of=12b] {$F_{a}^\prime$}; 
	\node[minus1] (11a) [left=24pt of 12a] {$F_{a}$}; 
	\node[minus1] (11c) [left=24pt of 12c] {$F_{a}^{\prime\prime}$}; 
	\draw (11a)--(12a)--(1a)--(4a);
	\draw (11b)--(12b)--(1b)--(4a);
	\draw (11c)--(12c)--(1c)--(4a);
	\draw (11a)--(11b)--(11c)--(11a);
	
	\node[](40) [right=12pt of 4a]{$\longrightarrow$};
	
	\node[minus1] (11) [right=24pt of 40] {$E_{11}$}; 
	\node[minus1] (10) [above of=11] {$E_{10}$}; 
	\node[minus1] (9) [below of=11] {$E_{9}$}; 
	\node[minus2,very thick] (7) [left of=10] {$E_{7}$}; 
	\node[minus2,very thick] (8) [left of=9] {$E_{8}$}; 
	\node[minus2,very thick] (6) [right of=10] {$E_{6}$}; 
	\node[minus1] (12) [right of=11] {$E_{12}$}; 
	\node[minus2,very thick] (2) [right of=9] {$E_{2}$}; 
	\node[minus2,very thick] (5) [right of=6] {$E_{5}$}; 
	\node[minus2] (1) [right of=12] {$E_{1}$}; 
	\node[minus2,very thick] (3) [right of=2] {$E_{3}$}; 
	\node[minus2] (4) [right of=1] {$E_{4}$}; 
	\draw (7)--(10)--(6)--(5)--(4)--(3)--(2)--(9)--(8)--(7);
	\draw (7)--(11)--(8);
	\draw (11)--(12)--(1)--(4);
	
	\node[](47) [below=12pt of 12c] {$T:$ degree $3$, type $D_4$};
	\node[](51) [base right=of 47] {$S:$ degree $1$, type $E_6+A_2$};
\end{tikzpicture}
\]
In the above diagram, we write $F_a, F_a^\prime$, and $F_a^{\prime\prime}$, as they may map to $E_{11}$ or non-extremal curves, depending on the isomorphic classes of $T$ and $S$.
There are precisely two isomorphic classes of Du Val del Pezzo surfaces of type $S_3(D_4)$ and $S_1(E_6+A_2)$, respectively. This phenomenon becomes particularly clear from the perspective of quasi-\'etale covers (see also Remark \ref{rema:counterexample}):

\begin{enumerate}
	\item If $E_7 \cap E_8 \cap E_{11} \neq \emptyset$, then they intersect at a single point, and similarly, $F_a$, $F_a^\prime$, and $F_a^{\prime\prime}$ also intersect at a single point. In this case, each of $F_a$, $F_a^\prime$, or $F_a^{\prime\prime}$ maps to $E_{11}$. Therefore, we may write $a = 11$ in this case.
	
	\item If $E_7 \cap E_8 \cap E_{11} = \emptyset$, then $F_a\cap F_a^\prime\cap F_a^{\prime\prime}=\emptyset$ as well.
	In this case, $F_a$, $F_a^\prime$, $F_a^{\prime\prime}$, and $E_{11}$ are not in correspondence: each of $F_a$, $F_a^\prime$, or $F_a^{\prime\prime}$ maps to a same nodal rational curve in the linear system $\lvert -K_X \rvert$, and the preimage of $E_{11}$ is a nodal rational curve in $Y$ with a self-intersection number of $3$.
\end{enumerate}

We have that
$$ \Cris(S)\cong C_2.$$
In both cases, the subgroups $H_T\times H_S\subseteq \Cris(\pi)$ up to conjugacy satisfying $\rho(\tT,H_T)\geq\rho(\tS,H_S)$ are listed in Table \ref{tab-deg1,E6+A2}.
\begin{center}
	\begin{table}[hbt]
		\caption{Degree $1$, type $E_6+A_2$.}
		\label{tab-deg1,E6+A2}
		\begin{tabular}{cccc}
			\toprule
			$H_\tS$&$H_\tT$&$\rho(\tS,H_S)$&$\rho(\tT,H_T)$\\
			\midrule
			$C_2$&$C_1$&$\bf 6$&$\bf 7$\\
			\bottomrule
		\end{tabular}
	\end{table}
\end{center}

%It would be interesting to consider to what extent the above proof is valid in general.

%$$X_\text{prim}(\mathbb{Q}) = \bigsqcup_{L\in G\text{-Field}(\bQ)} Y(L)^G/\mathrm{Aut}_{\mathbb{Q}}(L) $$

%\nocite{*}
\normalem%用于解决\usepackage{ulem}的bug
\bibliographystyle{alpha}
%\bibliography{weakdp}
\bibliography{myref}

\newcommand{\etalchar}[1]{$^{#1}$}
\begin{thebibliography}{BHPVdV15}

\bibitem[Art77]{Artin_1977}
M.~Artin.
\newblock {\em Coverings of the Rational Double Points in Characteristic p},
  pages 11--22.
\newblock Cambridge University Press, 1977.

\bibitem[BCP97]{magma}
Wieb Bosma, John Cannon, and Catherine Playoust.
\newblock The {M}agma algebra system. {I}. {T}he user language.
\newblock {\em J. Symbolic Comput.}, 24(3-4):235--265, 1997.
\newblock Computational algebra and number theory (London, 1993).

\bibitem[BHB20]{BHB20}
T.~D. Browning and D.~R. Heath-Brown.
\newblock {Density of rational points on a quadric bundle in
  $\mathbb{P}^{3}\times \mathbb{P}^{3}$}.
\newblock {\em Duke Mathematical Journal}, 169(16):3099 -- 3165, 2020.

\bibitem[BHPVdV15]{barth2015compact}
Wolf Barth, Klaus Hulek, Chris Peters, and Antonius Van~de Ven.
\newblock {\em Compact complex surfaces}, volume~4.
\newblock Springer, 2015.

\bibitem[BL17]{BL17}
T.~D. Browning and D.~Loughran.
\newblock Varieties with too many rational points.
\newblock {\em Math. Z.}, 285(3):1249--1267, 2017.

\bibitem[BL18]{BL18}
Tim Browning and Daniel Loughran.
\newblock Sieving rational points on varieties.
\newblock {\em Transactions of the American Mathematical Society},
  371(8):5757--5785, September 2018.

\bibitem[BM90]{BM90}
V.V. Batyrev and Yu.~I. Manin.
\newblock Sur le nombre des points rationnels de hauteur born{\'e} des
  vari{\'e}t{\'e}s alg{\'e}briques.
\newblock {\em Mathematische Annalen}, 286(1-3):27--44, 1990.

\bibitem[Bri13]{bright2013brauer}
Martin Bright.
\newblock Brauer groups of singular del {P}ezzo surfaces.
\newblock {\em Michigan Mathematical Journal}, 62(3):657--664, 2013.

\bibitem[Bro06]{browning2006density}
TD~Browning.
\newblock The density of rational points on a certain singular cubic surface.
\newblock {\em Journal of Number Theory}, 119(2):242--283, 2006.

\bibitem[BT96]{batyrev1996rational}
V.~V. Batyrev and Y.~Tschinkel.
\newblock Rational points on some {F}ano cubic bundles.
\newblock {\em Comptes rendus de l'Acad{\'e}mie des sciences. S{\'e}rie 1,
  Math{\'e}matique}, 323(1):41--46, 1996.

\bibitem[BT98a]{BT98toric}
Victor~V Batyrev and Yuri Tschinkel.
\newblock Manin's conjecture for toric varieties.
\newblock {\em Journal of Algebraic Geometry}, 7(1):15--53, 1998.

\bibitem[BT98b]{BT98}
Victor~V. Batyrev and Yuri Tschinkel.
\newblock Tamagawa numbers of polarized algebraic varieties.
\newblock In Peyre Emmanuel, editor, {\em Nombre et r\'epartition de points de
  hauteur born\'ee}, number 251 in Ast\'erisque. Soci\'et\'e math\'ematique de
  France, 1998.

\bibitem[CDJ{\etalchar{+}}22]{corvaja2022}
Pietro Corvaja, Julian~Lawrence Demeio, Ariyan Javanpeykar, Davide Lombardo,
  and Umberto Zannier.
\newblock On the distribution of rational points on ramified covers of abelian
  varieties.
\newblock {\em Compositio Mathematica}, 158(11):2109--2155, 2022.

\bibitem[CLT10]{chambert2010integral}
Antoine Chambert-Loir and Yuri Tschinkel.
\newblock Integral points of bounded height on toric varieties.
\newblock {\em arXiv preprint arXiv:1006.3345}, 2010.

\bibitem[CP21]{CP20}
Ivan Cheltsov and Yuri Prokhorov.
\newblock {D}el {P}ezzo surfaces with infinite automorphism groups.
\newblock {\em Algebraic Geometry}, pages 319--357, 2021.

\bibitem[Der13]{Derenthal2013}
Ulrich Derenthal.
\newblock Singular del {P}ezzo surfaces whose universal torsors are
  hypersurfaces.
\newblock {\em Proceedings of the London Mathematical Society},
  108(3):638--681, September 2013.

\bibitem[DJT08]{derenthal2008nef}
Ulrich Derenthal, Michael Joyce, and Zachariah Teitler.
\newblock The nef cone volume of generalized del pezzo surfaces.
\newblock {\em Algebra \& Number Theory}, 2(2):157--182, 2008.

\bibitem[dlB02]{delaBretche2002}
Régis de~la Bretèche.
\newblock Nombre de points de hauteur bornée sur les surfaces de del pezzo de
  degré 5.
\newblock {\em Duke Mathematical Journal}, 113(3), June 2002.

\bibitem[dLBBP12]{de2012manin}
R{\'e}gis de~La~Bret{\`e}che, Tim Browning, and Emmanuel Peyre.
\newblock On manin's conjecture for a family of ch{\^a}telet surfaces.
\newblock {\em Annals of Mathematics}, pages 297--343, 2012.

\bibitem[dlBDL{\etalchar{+}}19]{manincase19}
Régis de~la Bretèche, Kevin Destagnol, Jianya Liu, Jie Wu, and Yongqiang
  Zhao.
\newblock On a certain non-split cubic surface.
\newblock {\em Science China Mathematics}, 62(12):2435–2446, June 2019.

\bibitem[Dol12]{Dolgachev2012}
Igor~V. Dolgachev.
\newblock {\em Classical Algebraic Geometry}.
\newblock Cambridge University Press, August 2012.

\bibitem[DP20]{derenthal2020split}
Ulrich Derenthal and Marta Pieropan.
\newblock The split torsor method for manin’s conjecture.
\newblock {\em Transactions of the American Mathematical Society},
  373(12):8485--8524, 2020.

\bibitem[DY22]{DY22}
Ratko Darda and Takehiko Yasuda.
\newblock The {B}atyrev-{M}anin conjecture for {DM} stacks.
\newblock {\em arXiv preprint arXiv:2207.03645, to appear in the Journal of the
  European Mathematical Society}, 2022.

\bibitem[ESZB23]{ESZ23}
Jordan~S. Ellenberg, Matthew Satriano, and David Zureick-Brown.
\newblock Heights on stacks and a generalized {B}atyrev-{M}anin-{M}alle
  conjecture.
\newblock In {\em Forum of Mathematics, Sigma}, volume~11, page e14. Cambridge
  University Press, 2023.

\bibitem[FMT89]{FMT89}
J.~Franke, Yu.~I. Manin, and Y.~Tschinkel.
\newblock Rational points of bounded height on {F}ano varieties.
\newblock {\em Invent. Math.}, 95(2):421--435, 1989.

\bibitem[FP16]{manincase16}
Christopher Frei and Marta Pieropan.
\newblock O-minimality on twisted universal torsors and manin's conjecture over
  number fields.
\newblock {\em Annales Scientifiques de l'{\'E}cole Normale Sup{\'e}rieure},
  49(4):757--811, 2016.

\bibitem[Gao23]{gao2023zariski}
Runxuan Gao.
\newblock A {Z}ariski dense exceptional set in {M}anin's {C}onjecture:
  dimension 2.
\newblock {\em Research in Number Theory}, 9(2):42, 2023.

\bibitem[Gao24a]{gao2024geometric}
Runxuan Gao.
\newblock The geometric exceptional set in manin’s conjecture for batyrev and
  tschinkel’s example.
\newblock {\em European Journal of Mathematics}, 10(4):60, 2024.

\bibitem[Gao24b]{github}
Runxuan Gao.
\newblock {weakdp}.
\newblock \url{https://github.com/runxg/weakdp}, 2024.

\bibitem[Har77]{Hartshorne1977AlgebraicG}
Robin Hartshorne.
\newblock {\em Algebraic Geometry}.
\newblock Springer New York, 1977.

\bibitem[HTT15]{HTT15}
Brendan Hassett, Sho Tanimoto, and Yuri Tschinkel.
\newblock Balanced line bundles and equivariant compactifications of
  homogeneous spaces.
\newblock {\em International Mathematics Research Notices},
  2015(15):6375--6410, 2015.

\bibitem[Hua21]{huang2021equidistribution}
Zhizhong Huang.
\newblock Equidistribution of rational points and the geometric sieve for toric
  varieties.
\newblock {\em arXiv preprint arXiv:2111.01509}, 2021.

\bibitem[HW24]{harpaz2024supersolvable}
Yonatan Harpaz and Olivier Wittenberg.
\newblock Supersolvable descent for rational points.
\newblock {\em Algebra \& Number Theory}, 18(4):787--814, 2024.

\bibitem[KKV89]{knop1989picard}
Friedrich Knop, Hanspeter Kraft, and Thierry Vust.
\newblock The picard group of a g-variety.
\newblock {\em Algebraische Transformationsgruppen und Invariantentheorie
  Algebraic Transformation Groups and Invariant Theory}, pages 77--87, 1989.

\bibitem[KM98]{Kollar-Mori}
Janos Koll{\'{a}}r and Shigefumi Mori.
\newblock {\em Birational Geometry of Algebraic Varieties}.
\newblock Cambridge University Press, September 1998.

\bibitem[Kol09]{Kol2009}
J{\'a}nos Koll{\'a}r.
\newblock {\em Positive Sasakian Structures on 5-Manifolds}, pages 93--117.
\newblock Birkh{\"a}user Boston, Boston, MA, 2009.

\bibitem[LR19a]{LR19}
C{\'e}cile Le~Rudulier.
\newblock Points alg{\'e}briques de hauteur born{\'e}e sur une surface.
\newblock {\em Bull. Soc. Math. France}, 147(4):705--748, 2019.

\bibitem[LR19b]{le2019points}
C{\'e}cile Le~Rudulier.
\newblock Points alg{\'e}briques de hauteur born{\'e}e sur une surface.
\newblock {\em Bull. Soc. Math. France}, 147(4):705--748, 2019.

\bibitem[LST22]{LST}
Brian Lehmann, Akash~Kumar Sengupta, and Sho Tanimoto.
\newblock Geometric consistency of {M}anin{\textquotesingle}s conjecture.
\newblock {\em Compositio Mathematica}, 158(6):1375--1427, June 2022.

\bibitem[LT17]{LTDuke}
B.~Lehmann and S.~Tanimoto.
\newblock On the geometry of thin exceptional sets in {M}anin's conjecture.
\newblock {\em Duke Math. J.}, 166(15):2815--2869, 2017.

\bibitem[LT19]{LT19}
Brian Lehmann and Sho Tanimoto.
\newblock On exceptional sets in {M}anin's conjecture.
\newblock {\em Research in the Mathematical Sciences}, 6(12), 2019.

\bibitem[LWZ17]{manincase17}
Jianya Liu, Jie Wu, and Yongqiang Zhao.
\newblock {Manin’s Conjecture for a Class of Singular Cubic Hypersurfaces}.
\newblock {\em International Mathematics Research Notices}, 2019(7):2008--2043,
  08 2017.

\bibitem[Man86]{manin1986cubic}
Yu~I Manin.
\newblock {\em Cubic forms: algebra, geometry, arithmetic}.
\newblock Elsevier, 1986.

\bibitem[MZ88]{MZ88}
M~Miyanishi and D.Q Zhang.
\newblock Gorenstein log del {P}ezzo surfaces of rank one.
\newblock {\em Journal of Algebra}, 118(1):63--84, 1988.

\bibitem[MZ93]{MZ93}
Masayoshi Miyanishi and De-Qi Zhang.
\newblock Gorenstein log del {P}ezzo surfaces, {II}.
\newblock {\em Journal of Algebra}, 156(1):183--193, 1993.

\bibitem[Pey95]{Peyre95}
E.~Peyre.
\newblock Hauteurs et mesures de {T}amagawa sur les vari\'et\'es de {F}ano.
\newblock {\em Duke Math. J.}, 79(1):101--218, 1995.

\bibitem[Pey03]{Peyre03}
E.~Peyre.
\newblock Points de hauteur born\'ee, topologie ad\'elique et mesures de
  {T}amagawa.
\newblock {\em J. Th\'eor. Nombres Bordeaux}, 15(1):319--349, 2003.

\bibitem[Ser07]{serre2016topics}
Jean-Pierre Serre.
\newblock {\em Topics in Galois theory}.
\newblock CRC Press, 2007.

\bibitem[Tre18]{trepalin2018quotients}
Andrey Trepalin.
\newblock Quotients of del pezzo surfaces of high degree.
\newblock {\em Transactions of the American Mathematical Society},
  370(9):6097--6124, 2018.

\bibitem[Vir23]{virin2023automorphisms}
Nikita Virin.
\newblock Automorphisms of {D}u {V}al del {P}ezzo surfaces.
\newblock {\em arXiv preprint arXiv:2312.08498}, 2023.

\end{thebibliography}

\end{document}